\documentclass[11pt]{amsart}
\usepackage{amsmath,cite}
\usepackage{amssymb,amsthm,tocvsec2}
\usepackage{amsmath}
\usepackage{sidecap}
\usepackage{subfigure}
\usepackage{wrapfig}
\usepackage{graphicx}
\usepackage{epstopdf}
\usepackage{pict2e}
\usepackage{threeparttable}
\usepackage{graphicx}
\usepackage{setspace}
\usepackage{float}
\usepackage{color}
\usepackage{mathrsfs}
\usepackage{epsfig,epsf,latexsym,subfig,version}
\usepackage{mdwlist}
\usepackage{paralist}
\usepackage{setspace}
\usepackage{enumitem}
\setlist{nolistsep}
\usepackage{indentfirst}

\catcode`\@=11
\@addtoreset{equation}{section} % Makes \section reset 'equation' counter.
\theoremstyle{plain}
\newtheorem{thm}{Theorem}[section]

\newtheorem{remark}{\textbf{Remark}}[section]
\newcommand{\eps}{\epsilon}

\newcommand{\bm}{\boldsymbol}
\newcommand{\bu}{{\mathbf u}}
\newcommand{\bv}{{\mathbf v}}

\newcommand{\Grad}[1]{\nabla #1}

\newcommand{\ph}{\phi}

\newcommand{\be}{\begin{equation}}
\newcommand{\ee}{\end{equation}}

\newcommand{\bse}{\begin{subequations}}
\newcommand{\ese}{\end{subequations}}
\def\benl{\begin{eqnarray*}}
\def\eenl{\end{eqnarray*}}

\def\bu{\bm{u}}

\def\be{\bm{e}}

\def\bx{\bm{x}}

\def\bmu1{\bm{\mu_1}}

\def\ln{{\rm ln}}

\newcommand{\ben}{\begin{eqnarray}}
\newcommand{\een}{\end{eqnarray}}
\newcommand{\beq}{\begin{equation}}
\newcommand{\eeq}{\end{equation}}
\newcommand{\bea}{\begin{array}}
\newcommand{\eea}{\end{array}}
\newcommand{\bef}{\begin{figure}[H]}
\newcommand{\eef}{\end{figure}}

\begin{document}
\bibliographystyle{plain}
\graphicspath{ {Figures/} }

\title[A new Lagrange Multiplier approach]
{A new  Lagrange Multiplier approach for gradient flows}

\thanks{
${}^\dag$ Department of Applied Mathematics, Illinois Institute of Technology, Chicago, IL 60616, USA (qcheng4$@$iit.edu, cliu124$@$iit.edu). The work of C.L. is supported in part by NSF grant DMS-1759536.\\
 ${}^\ddag$ Department of Mathematics, Purdue University, West Lafayette, IN 47907, USA  (shen7$@$purdue.edu). The work of J. S. is supported in part by NSF grants  DMS-1620262, DMS-1720442 and AFOSR  grant FA9550-16-1-0102.}

  \author{Qing Cheng$^\dag$
             \and Chun Liu$^\dag$ \and Jie Shen$^\ddag$}

\date{}
%\subjclass{65M12;  35K20; 35K35; 35K55; 65Z05}

\maketitle

\begin{abstract}
We propose a new  Lagrange Multiplier approach to design unconditional energy stable  schemes for gradient flows.   The new approach leads to unconditionally energy stable schemes that are as accurate and efficient as the recently proposed SAV approach \cite{SAV01}, but enjoys two additional advantages: (i)   schemes based on the new approach  dissipate the original energy, as opposed to a modified energy in the recently proposed SAV approach \cite{SAV01}; and (ii) they do not require the nonlinear part of the free energy to be bounded from below as is required in the SAV approach. The price we pay for these advantages is that a nonlinear algebraic equation has to be solved to determine the Lagrange multiplier. 
We present ample numerical results to validate the new approach, and, as a particular example of applications,  we consider  a coupled Cahn-Hilliard model for block copolymers (BCP), and carry out interesting simulations which  are consistent with  experiment results.

\end{abstract}

\section{Introduction}
Gradient flows   play  an  important  role in science and engineering as a large class of mathematical models which can be described by PDEs in the form of gradient flow,  some particularly interesting  examples include:    crystal growth \cite{tong2001phase,kobayashi1993modeling,braun1997adaptive},  solidification \cite{boettinger2002phase,wang1993thermodynamically,karma1996phase},  tumor growth \cite{oden2010general,wise2008three},  thin film \cite{karma1998spiral,wang2003phase},  liquid-vapor phase transformations \cite{wilding1998liquid,borcia2005phase},  fracture mechanics \cite{borden2012phase,miehe2010thermodynamically}.

How to design efficient and energy stable numerical schemes for gradient flows has attracted much attention in recent years, we refer to \cite{du2019phase} for a up-to-date extensive review on this subject. In particular, the recent proposed  SAV  approach \cite{shen2018scalar,shen2017new,cheng2018multiple}, which is inspired by the IEQ approach  \cite{yang2017efficient} which in its turn is originated from a Lagrange multiplier approach  in \cite{guillen2013linear,badia2011finite}, has proven to be very successful for a large class of gradient flows. 

We briefly describe the SAV approach below.
To fix the idea, we consider a system with total free energy  in the form 
\begin{equation}\label{orienergy}
 E(\phi)=\int_\Omega \frac12\mathcal{L}\phi\cdot \phi +F(\phi) d\bx,
\end{equation}
where $\mathcal{L}$ is certain linear positive operator. Then  
a general gradient flow with the above free energy  takes the following form
\begin{equation}\label{grad:flow}
\begin{split}
&\phi_t=-\mathcal{G}\mu,\\
&\mu=\frac{\delta E(\phi)}{\delta\phi}=\mathcal{L}\phi+F'(\phi),
\end{split}
\end{equation}
where  $\mathcal{G}$ is a positive operator describing the relaxation process of the system.

 Assuming $\int_{\Omega}F(\phi)d\bx>-C_0$ for some $C_0>0$,  we introduce a scalar auxiliary variable (SAV)  $r(t)=\sqrt{\int_{\Omega}F(\phi)d\bx+C_0}$ and rewrite \eqref{grad:flow} as 
\begin{eqnarray}
&&\partial_t\phi=-\mathcal{G}\mu,\label{sav:1}\\
&&\mu=\mathcal{L}\phi+\frac{r(t)}{\sqrt{\int_{\Omega}F(\phi)d\bx+C} }F'(\phi),\label{sav:2}\\
&&r_t=(\frac 12 \frac{F'(\phi)}{\sqrt{\int_{\Omega}F(\phi)d\bx+C}},\phi_t).\label{sav:3}
\end{eqnarray}
The system \eqref{sav:1}-\eqref{sav:3}, with the initial condition $r(0)=\sqrt{\int_{\Omega}F(\phi|_{t=0})d\bx+C_0}$ is equivalent to \eqref{grad:flow}. % and satisfies  the following energy dissipative law:
%\begin{eqnarray}\label{LAW:4}&&\frac{d}{dt} E_{tot}(\phi)=-(\mathcal{G}\mu,\mu),\end{eqnarray}
%with $E_{tot}=\int_{\Omega}\frac12|\Grad\phi|^2d\bx+r^2$. 
Then, one can construct  semi-implicit schemes, by treating the linear terms implicitly and nonlinear terms explicitly in the above, that are unconditionally energy stable schemes and  require solving only decoupled linear systems with constant coefficients at each time step. Hence they are very efficient. 

However, the unconditional energy stability is with respect to a modified energy $$\tilde E(\phi,r)=\int_{\Omega}\frac 12 {\mathcal L}\phi\cdot \phi d\bx+r^2,$$ not the original energy \eqref{orienergy}, and it requires $\int_{\Omega}F(\phi)d\bx$  to be bounded from below. A main purpose of this paper is to introduce a new Lagrange multiplier approach which, in addition to the main advantages of the SAV approach, leads to schemes which is unconditionally energy stable with the original energy and does not require the explicitly treated  part of the free energy to be bounded from below. 
Hence, it can be applied to a larger class of gradient flows. The small price we pay for these additional advantages is that we need to solve  a nonlinear algebraic equation for the Lagrange multiplier whose cost is negligible compared with the main cost. We present ample numerical evident to show that this new approach is an effective  approach for gradient flows.
As an example of interesting applications, we apply the new Lagrange Multiplier to simulate the dynamics of block-Copolymers modeled by a coupled Cahn-Hilliard equations \cite{avalos2016frustrated}.

The reminder of this paper is structured as follows. In Section 2, we introduce the new Lagrange Multiplier approach for gradient flows,  derive  its stability property and develop a fast implementation procedure and present numerical results to validate the new approach.  In Section 3, we develop the new Lagrange Multiplier approach for  gradient flows with multiple components. In Section 4, we describe an adaptive time stepping procedure. 
 In Section 5, we use the new approach to simulate  evolution of frustrated phases of block-Copolymers. Some concluding remarks are given in Section 6.

\section{New Lagrange Multiplier approach for gradient flows}

%\subsection{\bf Gradient flows of a single function}
We introduce below a new Lagrange Multiplier approach for  gradient flows. 

 As in the SAV approach, we introduce a scalar auxiliary function  $\eta(t)$,  and  reformulate the gradient flow \eqref{grad:flow} as
\begin{equation}\label{new:pde}
\begin{split}
&\frac{\partial\phi}{\partial t}=-\mathcal{G}\mu,\\
&\mu=\mathcal{L}\phi+\eta(t) F'(\phi),\\
&\frac{d}{dt}\int_{\Omega}F(\phi) d\bx=\eta(t) \int_\Omega F'(\phi)\phi_td\bx .
\end{split}
\end{equation}
If we set the initial condition for $\eta(t)$ to be $\eta(0)=1$, then it is easy to see that 
the new system \eqref{new:pde} is equivalent to the original system \eqref{grad:flow}, i.e., 
 $\eta(t)\equiv 1$ in \eqref{new:pde}.

Note that unlike  the SAV approach where $r(t)$ is just an auxiliary variable, the role of $\eta(t)$ here is to serve as  a Lagrange Multiplier to enforce  dissipation of the original energy. Indeed, 
taking the inner products of the first two equations in the above with $\mu$ and $-\phi_t$, respectively, summing up the results together with the third equation, we 
obtain the original energy dissipative law:
\begin{equation}
\frac{d}{dt}E(\phi)=-(\mu,\mathcal{G}\mu).
\end{equation}
\subsection{New schemes, their stability and efficient implementation}
%For SAV approach, $\frac{r^{n+1}}{\sqrt{F(\phi^n)+C}} \approx 1$.  
Similar to the SAV approach, we can construct efficient and accurate numerical schemes for \eqref{new:pde}. For example,  
a second order scheme based on Crank-Nicolson is as follows:
 \begin{eqnarray}
&&\frac{\phi^{n+1}-\phi^n}{\delta t}=-\mathcal{G}\mu^{n+\frac 12},\label{new:cn1}\\
&&\mu^{n+\frac 12}=\mathcal{L} \phi^{n+\frac 12}+F'(\phi^{\star,n})\eta^{n+\frac 12},\label{new:cn2}\\
&&(F(\phi^{n+1})-F(\phi^n),1)=\eta^{n+\frac 12}(F'(\phi^{\star,n}),\phi^{n+1}-\phi^n),\label{new:cn3}
\end{eqnarray} 
where $F'(\phi^{\star,n})=F'(\frac 32\phi^n-\frac 12\phi^{n-1})$. 

\begin{thm}\label{cn:the}
The Crank-Nicolson  scheme \eqref{new:cn1}-\eqref{new:cn3}, it satisfies the following energy dissipative law,  
\begin{eqnarray}\label{LAW:6}
\frac{E^{n+1}-E^n}{\delta t}=-(\mathcal{G}\mu^{n+\frac 12},\mu^{n+\frac 12}),
\end{eqnarray}
where $E^{n}=\frac 12(\mathcal{L}\phi^n,\phi^n)+\int_{\Omega}F(\phi^n)d\bx$.
\end{thm}
\begin{proof}
 Taking the inner products of \eqref{new:cn1} with $\mu^{n+\frac 12}$ and of \eqref{new:cn2} with $-\frac{\phi^{n+1}-\phi^n}{\delta t}$,  and multiplying \eqref{new:cn3} with $\frac1{\delta t}$, summing up the three relations, we obtain immediately the desired result.
\end{proof}
%\begin{comment}
We now show how to solve scheme \eqref{new:cn1}-\eqref{new:cn3} efficiently.  We derive from  \eqref{new:cn1}-\eqref{new:cn2}  that 
\begin{eqnarray}\label{step:1}
\frac{\phi^{n+1}}{\delta t}+\mathcal{G}\mathcal{L}\phi^{n+1}=\frac{\phi^n}{\delta t}-\mathcal{G} F'(\phi^{\star,n})\eta^{n+1/2}.
\end{eqnarray}
We define   a  linear operator $\chi$  by
\begin{equation}
 \chi(\phi):=(\frac{1}{\delta t}I+\mathcal{G}\mathcal{L} )\phi,
 \end{equation}
 and  apply  the operator $\chi^{-1}$ to both sides of equation \eqref{step:1} to get
\begin{equation}\label{step:2}
\begin{split}
\phi^{n+1}&=\chi^{-1}\{\frac{\phi^n}{\delta t}\}+\eta^{n+1/2}\chi^{-1}\{-\mathcal{G} F'(\phi^{\star,n})\}\\&:=p^n+\eta^{n+1/2}q^n,
\end{split}
\end{equation}
with $p^n$ and $q^n$ being 
\begin{equation}\label{pq}
\begin{split}
p^n=\chi^{-1}\{\frac{\phi^n}{\delta t}\}, \quad  q^n=\chi^{-1}\{-\mathcal{G} F'(\phi^{\star,n})\}.
\end{split}
\end{equation}
 
Now we plug $\phi^{n+1}$ into equation \eqref{new:cn3} to obtain
\begin{equation}\label{ch:nonlinear}
\begin{split}
&(F(p^n+\eta^{n+1/2}q^n)-F(\phi^n),1)\\&=\eta^{n+1/2}(F'(\phi^{\star,n}),(p^n+\eta^{n+1/2}q^n)-\phi^n).
\end{split}
\end{equation}
%\begin{remark}
The above equation  is a nonlinear algebraic equation  for $\eta^{n+1/2}$. The complexity of the nonlinear equation \eqref{ch:nonlinear} depends on $F(\phi)$, e.g., if $F(\phi)$ is the usual double well potential, i.e., $F(\phi)=\frac 14(\phi^2-1)^2$,  \eqref{ch:nonlinear} will be a fourth-order algebraic equation. In general, \eqref{ch:nonlinear} has multiple solutions, but since we are looking for $\eta^{n+1/2}$ as an approximation to 1, we can use  a Newton iteration with 1 as the initial condition, it will generally converge to a solution  closer to 1 (assuming $\delta t$ is not too large). 
%\end{remark}

To summarize, the  algorithm consists of the following steps:

 \begin{description}
 \item  [Step 1]   Compute  $p^n$ and $q^n$ from  \eqref{pq};       
  \item [Step 2]  Find $\eta^{n+1/2}$ by solving  \eqref{ch:nonlinear};
  \item [Step 3]  Set $\phi^{n+1}=p^n+\eta^{n+1/2}q^n$, goto the next step.
\end{description}

Note that at each time step, we only need to solve two linear,  constant coefficients  equations in the {\bf Step 1},  plus a nonlinear algebraic equation  in {\bf  Step 2}. Since the cost for solving the nonlinear algebraic equation is negligible compared with the cost in {\bf Step 1}, the new algorithm is essentially as efficient as the SAV  approach.

\begin{remark}
We note that for the scheme   \eqref{new:cn1}-\eqref{new:cn3},   the original energy is dissipative, as opposed to the fact that only  a modified energy is dissipative   in the SAV approach. Furthermore, our new Lagrange Multiplier approach does not require  the free energy $\int_{\Omega}F(\phi)d\bx$ to be bounded from below.   On the other hand, the new approach requires solving a nonlinear algebraic equation for the Lagrange multiplier which brings some additional costs and theoretical difficulties for its analysis.
\end{remark}

For dissipative systems, it is usually better to use schemes based on backward difference formula (BDF). A second-order scheme for \eqref{new:pde} based on BDF  can be constructed as follows:
\begin{eqnarray}
&&\frac{3\phi^{n+1}-4\phi^n+\phi^{n-1}}{2\delta t}=-\mathcal{G}\mu^{n+1},\label{new:bdf1}\\
&&\mu^{n+1}=\mathcal{L} \phi^{n+1}+F'(\phi^{\star,n})\eta^{n+1},\label{new:bdf2}\\
&&(3F(\phi^{n+1})-4F(\phi^n)+F(\phi^{n-1}),1)=\eta^{n+1}(F'(\phi^{\star,n}),3\phi^{n+1}-4\phi^n+\phi^{n-1}),\label{new:bdf3}
\end{eqnarray} 
where $F'(\phi^{\star,n})=F'(2\phi^n-\phi^{n-1})$.

% and we mainly consider the following  two boundary conditions,   
%\begin{eqnarray} &&(i)\mbox{ periodic; or } (ii)\,\,\partial_{\bf n} \ph|_{\partial\Omega}=\partial_{\bf n} \mu|_{\partial\Omega}=0, \label{ori:pdefhbd2} \end{eqnarray}
%where $\bf n$ is the unit outward normal on the boundary $\partial\Omega$.

\begin{thm}\label{bdf:thm}
The scheme \eqref{new:bdf1}-\eqref{new:bdf3} satisfies the following energy dissipative law:
 \begin{eqnarray}\label{LAW:5}
{\tilde E^{n+1}-\tilde E^n}=-{\delta t}(\mathcal{G}\mu^{n+1},\mu^{n+1})-\frac 14 (\mathcal{L}(\phi^{n+1}-2\phi^n+\phi^{n-1}),\phi^{n+1}-2\phi^n+\phi^{n-1}),
\end{eqnarray}
where $\tilde E^{n}=\frac 14(\mathcal{L}\phi^n, \phi^n)+(\mathcal{L}(2\phi^{n}-\phi^{n-1}), 2\phi^{n}-\phi^{n-1})+\int_{\Omega}\frac 12(3F(\phi^n)-F(\phi^{n-1}))d\bx$. 
\end{thm}
\begin{proof}
Taking the inner product of equation \eqref{new:bdf1} with $\mu^{n+1}$,  we obtain
\begin{equation}\label{id1}
\begin{split}
(\frac{3\phi^{n+1}-4\phi^n+\phi^{n-1}}{2\delta t},\mu^{n+1})=-(\mathcal{G}\mu^{n+1},\mu^{n+1}),
\end{split}
\end{equation}
Taking the inner product of  equation \eqref{new:bdf2} with $3\phi^{n+1}-4\phi^n+\phi^{n-1}$   and  using the identity
\begin{equation}\label{iden1}
(2a,3a-4b+c)=|a|^2-|b|^2+|2a-b|^2-|2b-c|^2+|a-2b+c|^2,
\end{equation}
we obtain
\begin{equation}\label{stab:bdf2}
\begin{split}
(\mu^{n+1}&,3\phi^{n+1}  -4\phi^n +\phi^{n-1})=(\mathcal{L}\phi^{n+1},3\phi^{n+1}-4\phi^n+\phi^{n-1})\\&+
(F'(\phi^{\star,n})\eta^{n+1},3\phi^{n+1}-4\phi^n+\phi^{n-1})\\&=\frac 12\{(\mathcal{L}\phi^{n+1},\phi^{n+1})-(\mathcal{L}\phi^n,\phi^n)+(\mathcal{L}(2\phi^{n+1}-\phi^n),2\phi^{n+1}-\phi^n)\\&-(\mathcal{L}(2\phi^{n}-\phi^{n-1}),2\phi^{n}-\phi^{n-1})+(\mathcal{L}(\phi^{n+1}-2\phi^n+\phi^{n-1}),\phi^{n+1}-2\phi^n+\phi^{n-1})\}\\&+\eta^{n+1}(F'(\phi^{\star,n}),3\phi^{n+1}-4\phi^n+\phi^{n-1}).
\end{split}
\end{equation}
Combining \eqref{id1}, \eqref{stab:bdf2} and \eqref{new:bdf3}, and using
\begin{equation*}\label{stab:id}
\begin{split}
3F(\phi^{n+1})-4F(\phi^n)+F(\phi^{n-1})=3F(\phi^{n+1})-F(\phi^n)-(3F(\phi^n)-F(\phi^{n-1})),
\end{split}
\end{equation*}
 we obtain the desired result.

\end{proof}
\begin{remark}
Note that $\tilde E^n$ defined in the above theorem is not exact in the form of original energy. But it is easy to check that $\tilde E^n$ is a second-order approximation of $E^n=\frac 12 ({\mathcal L}\phi^n, \phi^n)+\int_\Omega F(\phi^n) d\bx $.
\end{remark}

\subsection{Numerical validations}
We provide below some numerical tests to validate our new schemes.
\subsubsection{Accuracy test}
In this subsection,  in order to test accuracy and convergence rate of scheme BDF2 and Crank-Nicolson in time. 

We consider the  Allen-Cahn ($\mathcal{G}=I$) and Cahn-Hillard equation ($\mathcal{G}=-\Delta$) in the form
\begin{equation}\label{ac_ch}
\begin{split}
&\phi_t =- \mathcal{G}\mu +f, \\
&\mu=-\Delta\ph+\frac1{\epsilon^2}\phi(\phi^2-1),
\end{split}
\end{equation}
with periodic boundary conditions.
We choose a force function $f$   such that the exact solution is
\begin{equation*}
\phi(x,y)=(\frac{\sin(2x)\cos(2y)}{4}+0.48)(1-\frac{\sin^2(t)}{2}),
\end{equation*}
and use $128^2$ Fourier-spectral modes in space so that the error is dominated by the time discretization error.
In Table \ref{order_allen_cahn}, we present $L^{\infty}$ errors for both Allen-Cahn and Cahn-Hilliard equations at various time steps.
It is observed that both scheme BDF2 and CN achieve  second order convergence rate in time.

\begin{table}[ht!]
\centering
\begin{tabular}{r||c|c|c|c}
\hline
$\delta t$            & {$BDF2$}  & Order & {$CN$} & Order    \\ \hline
$2\times 10^{-3}$    &$1.11E(-5)$  & --   &$5.31E(-6)$ &$-$       \\\hline
$1\times 10^{-3}$     &$2.88E(-6)$ & $1.94$&$1.38E(-6)$ &$1.95$    \\\hline
$5\times 10^{-4}$     &$7.31E(-7)$ &$1.95$  &$3.57E(-7)$ &$1.98$  \\\hline
$2.5\times 10^{-4}$  &$1.85E(-7)$ &$1.95$ &$9.23E(-8)$ &$1.98$   \\ \hline
$1.25\times 10^{-4}$  &$4.64E(-8)$&$1.99$ &$2.32E(-8)$ &$1.99$ \\\hline
$6.25\times 10^{-5}$  &$1.17E(-8)$&$1.98$ &$5.86E(-9)$ &$1.91$ \\\hline
$3.125\times 10^{-5}$ &$2.93E(-9)$ &$2.00$ &$1.46E(-9)$ &$2.00$ \\\hline
\hline
\end{tabular}
\vskip 0.5cm
\caption{Accuracy test: with given exact solution for Allen-Cahn equation and $\eps^2=0.02$.  The $L^{\infty}$ errors at $t=0.1$ for the  phase variables $\phi $ computed by scheme BDF2 and Crank-Nicolson using various time steps.}\label{order_allen_cahn}
\end{table}

\begin{table}[ht!]
\centering
\begin{tabular}{r||c|c|c|c}
\hline
$\delta t$            & {$BDF2$}  & Order & {$CN$} & Order    \\ \hline
$2\times 10^{-3}$    &$2.75E(-6)$  & --   &$3.04E(-6)$ &$-$       \\\hline
$1\times 10^{-3}$     &$6.92E(-7)$ & $1.99$&$7.67E(-7)$ &$1.98$    \\\hline
$5\times 10^{-4}$     &$1.73E(-7)$ &$2.00$  &$1.92E(-7)$ &$1.99$  \\\hline
$2.5\times 10^{-4}$  &$4.36E(-8)$ &$1.98$ &$4.83E(-8)$ &$1.99$   \\ \hline
$1.25\times 10^{-4}$  &$1.09E(-9)$&$2.00$ &$1.20E(-8)$ &$2.00$ \\\hline
$6.25\times 10^{-5}$  &$2.72E(-9)$&$2.00$ &$3.02E(-9)$ &$1.99$ \\\hline
$3.125\times 10^{-5}$ &$6.80E(-10)$ &$2.00$ &$7.57E(-10)$ &$2.00$ \\\hline
\hline
\end{tabular}
\vskip 0.5cm
\caption{Accuracy test: with given exact solution for Cahn-Hilliard equation and $\eps^2=0.06$.  The $L^{\infty}$ errors at $t=0.1$ for the  phase variables $\phi $ computed by scheme BDF2 and Crank-Nicolson using various time steps.}\label{order_cahn_Hilliard}
\end{table}

Next we examine the energy stability. We take $f=0$ in domain $[0,2\pi)^2$ and start with a random initial condition
\begin{equation*}
\phi(x,y)=0.03+0.001\,{\rm rand}(x,y),
\end{equation*}
where ${\rm rand}(x,y)$ represents random data between $[-1,1]^2$.
 The energy curves are plotted for both Allen-Cahn equation  and Cahn-Hilliard equation with $\eps^2=0.005$ in 
 Fig.\,\ref{Energy_cahn}.  It is observed from this figure that the computed energy for both cases decay with time.

\begin{figure}
\centering
\includegraphics[width=0.5\textwidth,clip==]{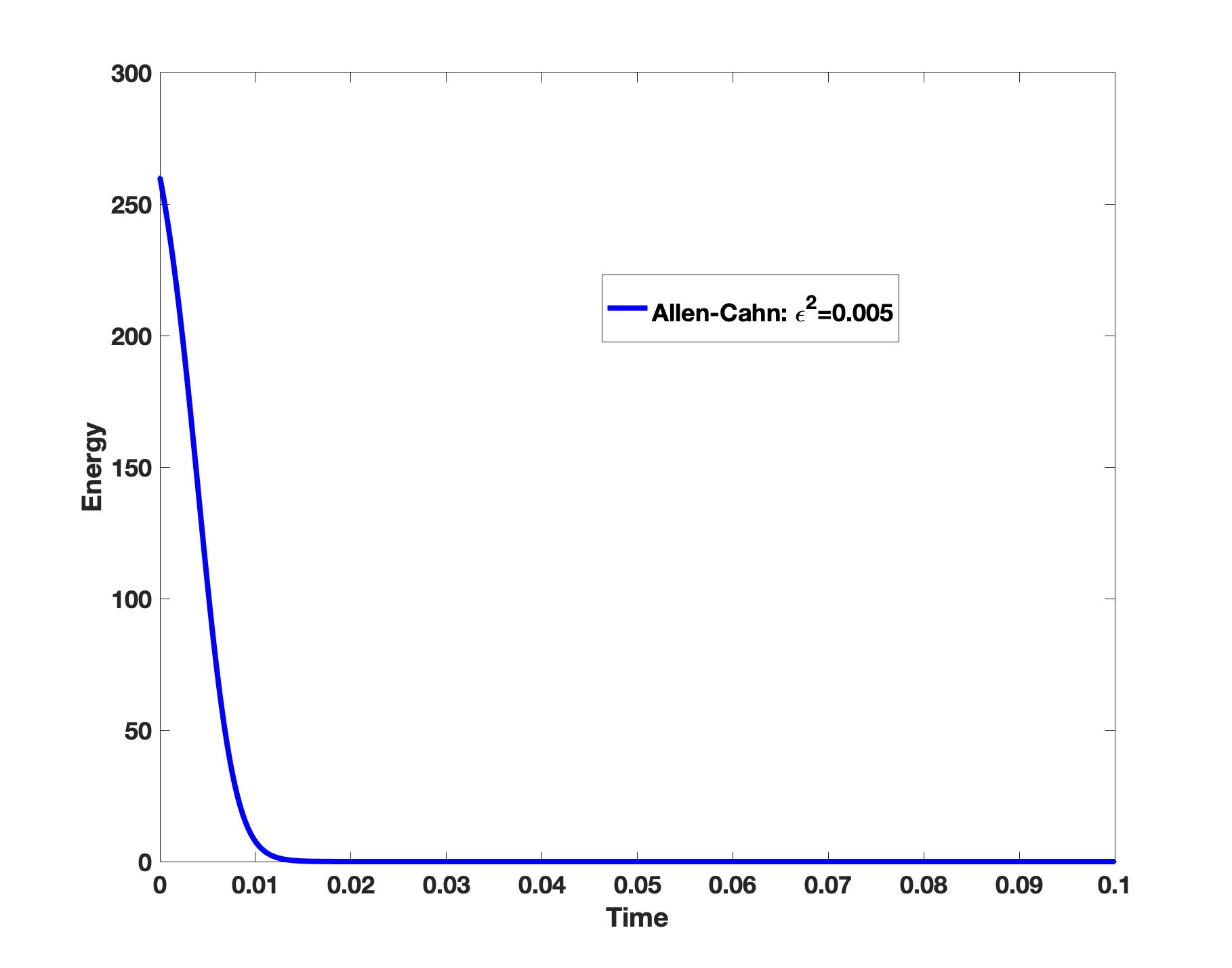}\hskip 0cm
\includegraphics[width=0.5\textwidth,clip==]{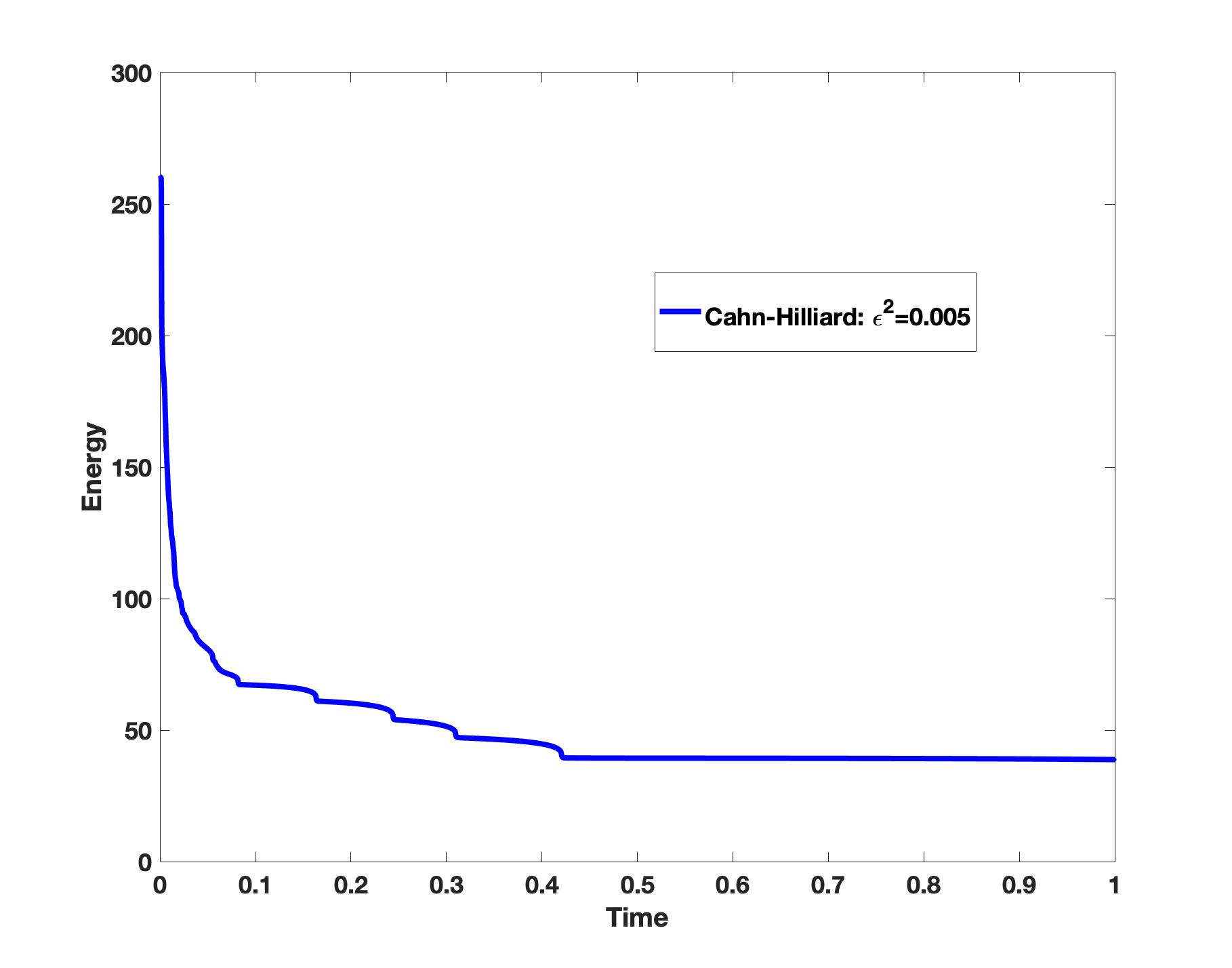}\hskip 0cm
\caption{Energy evolution of Allen-Cahn and Cahn-Hilliard equations  by using Scheme BDF2  of  new Lagrange Multiplier approach with time step $\delta t=10^{-5}$. }\label{Energy_cahn}
\end{figure}

\subsubsection{Molecular beam epitaxial without slope selection}
As an example where the nonlinear functional is unbounded from below, we consider  a model for molecular beam epitaxial (MBE)  without slope selection \cite{villain1991continuum}. For this model,
the total free energy is $E(\phi)=\int_{\Omega}\frac{\eps^2}{2}|\Delta \phi|^2+F(\phi)dx$, where
the nonlinear potential is 
\begin{equation}\label{ori:ene2}
F(\phi)=-\frac{1}{2}\ln(1+|\nabla \phi|^2).
\end{equation}
Since the Logarithmic potential  is not bounded from below, one has to use a different  energy splitting so that the SAV approach can be applied  \cite{cheng2019highly}.  However, our new approach can be directly applied since there is no requirement of boundedness from below.  More precisely, the $L^2$ gradient flow with respect to the free energy above is  
\begin{eqnarray}
&&\phi_t =-M\frac{\delta E(\phi)}{\delta \phi}= -M\big( \epsilon^2\Delta^2 \phi+F'(\phi)\big), \label{MBE:1}
\end{eqnarray}
% The computational domain is $\Omega=[0,L]^d, d=2,3$.
 with   boundary conditions:
\begin{eqnarray}\label{con:pdefhbd2}
&&(i)\,\,\phi \mbox{ is periodic with or } (ii)\,\,\partial_{\bm n} \ph|_{\partial\Omega}=\partial_{\bm n}\Delta\phi|_{\partial\Omega}=0,
\end{eqnarray}
where $\bf n$ is the unit outward normal on the boundary $\partial\Omega$.
  In the above,  $M$ is a mobility constant, and $F'(\phi)= \Grad \cdot\Big(\frac{\Grad \phi}{1+|\Grad \phi|^2}\big)$.

Then, a second-order scheme based on the new Lagrange Multiplier approach is:
\begin{eqnarray}
&&\frac{3\phi^{n+1}-4\phi^n+\phi^{n-1}}{2\delta t} +M\big( \epsilon^2\Delta^2 \phi^{n+1}+\eta^{n+1} F'( \phi^{\star,n})\big)=0,\label{mbe:bdf1}\\
&&\eta^{n+1}(F'( \phi^{\star,n}), 3\phi^{n+1}-4\phi^n+\phi^{n-1})= (3F(\Grad\phi^{n+1})-4F(\Grad\phi^n)+F(\Grad\phi^{n-1}),1).\label{mbe:bdf2}
\end{eqnarray}
Similarly as in the proof of Theorem \ref{bdf:thm}, we can easily show that the above equation is unconditionally energy stable.
 It is also clear that it can also be efficiently implemented.
 
 We present below  numerical simulations for MBE model by using the new Lagrange Multiplier approach \eqref{mbe:bdf1}-\eqref{mbe:bdf2}.  
 To simulate the coarsening dynamics,  we choose a random initial condition  varying from $-0.001$ to $0.001$ in  $\Omega=[0,2\pi)^2$. The parameters are
\begin{eqnarray}\label{Coarse_pa}
\epsilon= 0.03, \delta t=10^{-2}\;  M=1.
\end{eqnarray}
The time evolution of total energy and the Lagrange Multiplier are depicted in Fig. \ref{MBE_coarse}.  The total free energy  decays with time  while Lagrange Multiplier are almost equal to one except at the few initial time steps.  The  dynamic coarsening process is presented in Fig. \ref{NoSlopeCoarse}.
 
 \begin{comment}
 with the initial condition 
 \begin{equation}\label{mbe_initial}
 \phi(x,y,0)=0.1(\sin(3x)\sin(2y)+\sin(5x)\sin(5y))\quad (x,y)\in [0,2\pi)^2,
 \end{equation}
 with the parameters $\eps^2=0.1, M=1$, $\delta t=10^{-2}$, and 
 plot in Fig. \ref{MBE1} the time evolutions of total energy and the Lagrange multiplier. One observes that the total energy decays monotonically and that the values of the Lagrange multiplier are essentially one except at a few initial steps.

\begin{figure}[htbp]
\centering
%\includegraphics[width=0.45\textwidth,clip==]{adaptive_energy.jpg}
\includegraphics[width=0.45\textwidth,clip==]{MBE_energy.jpg}
\includegraphics[width=0.45\textwidth,clip==]{MBE_eta.jpg}
\caption{The evolution of energy and $\eta$ with time for MBE model without slope selection.  }\label{MBE1}
\end{figure}

 An example of the  dynamic process is shown in Fig. \ref{NoSlope_mbe} which is consistent with numerical results in \cite{cheng2019highly}.

\begin{figure}
\centering
\subfigure[$t=0.5$.]{
\includegraphics[width=0.23\textwidth,clip==]{mbe_1.jpg}}
\subfigure[$t=2.5$.]{
\includegraphics[width=0.23\textwidth,clip==]{mbe_2.jpg}}
\subfigure[$t=5$.]{
\includegraphics[width=0.23\textwidth,clip==]{mbe_3.jpg}}
\subfigure[$t=30$.]{
\includegraphics[width=0.23\textwidth,clip==]{mbe_4.jpg}}
\caption{The isolines of the numerical solutions of the height function $\phi$  for the model without slope selection with  initial condition \eqref{mbe_initial}.  Snapshots are taken at $t = 0.5 , 3.5 , 5 , 30$, respectively.}\label{NoSlope_mbe}
\end{figure}
\end{comment}

\begin{figure}[htbp]
\centering
\includegraphics[width=0.45\textwidth,clip==]{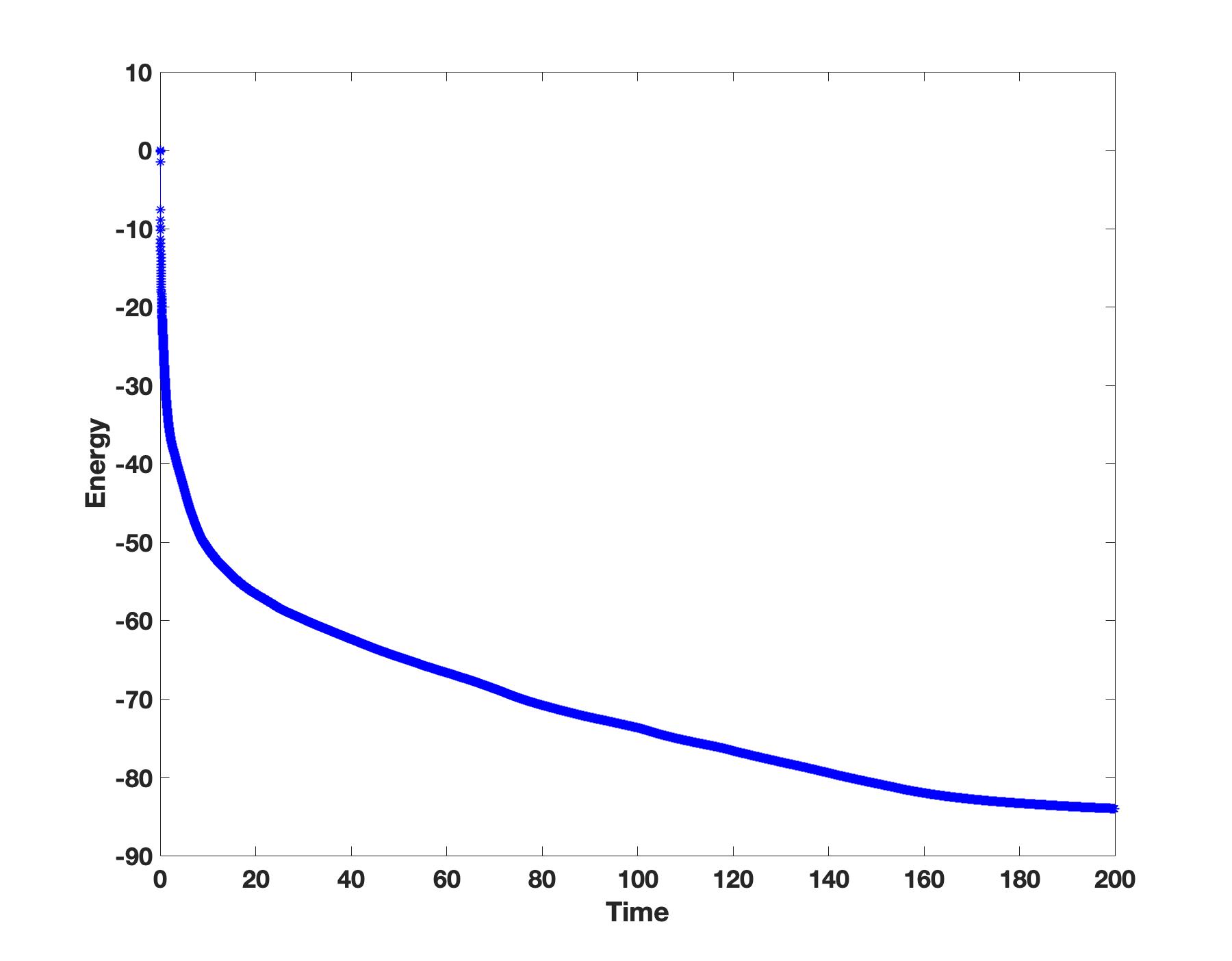}
\includegraphics[width=0.45\textwidth,clip==]{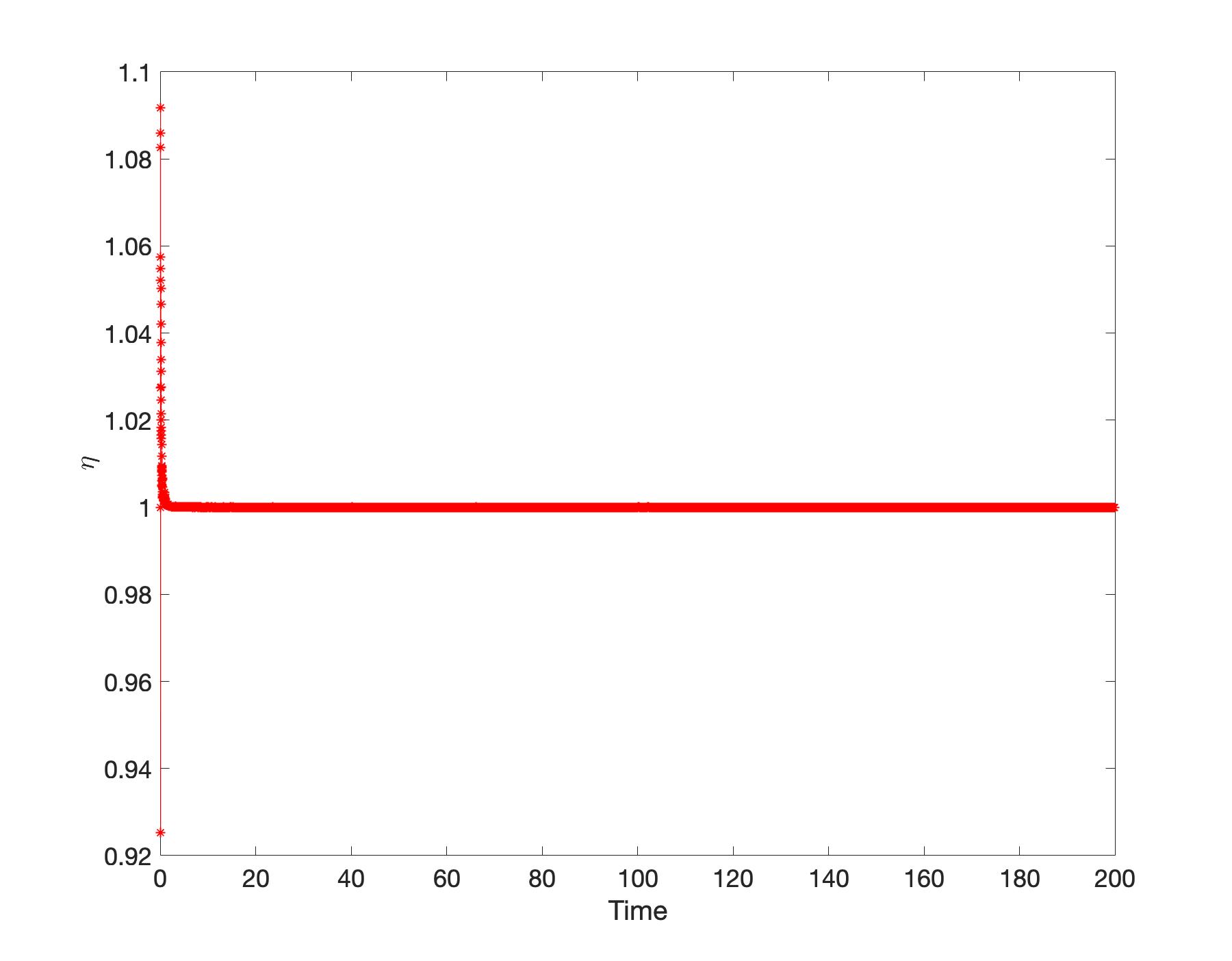}
\caption{The coarse dynamic  evolution of energy and $\eta$ with time for MBE model without slope selection.  }\label{MBE_coarse}
\end{figure}

\begin{figure}
\centering
\subfigure[$t=2.5$.]{
\includegraphics[width=0.23\textwidth,clip==]
{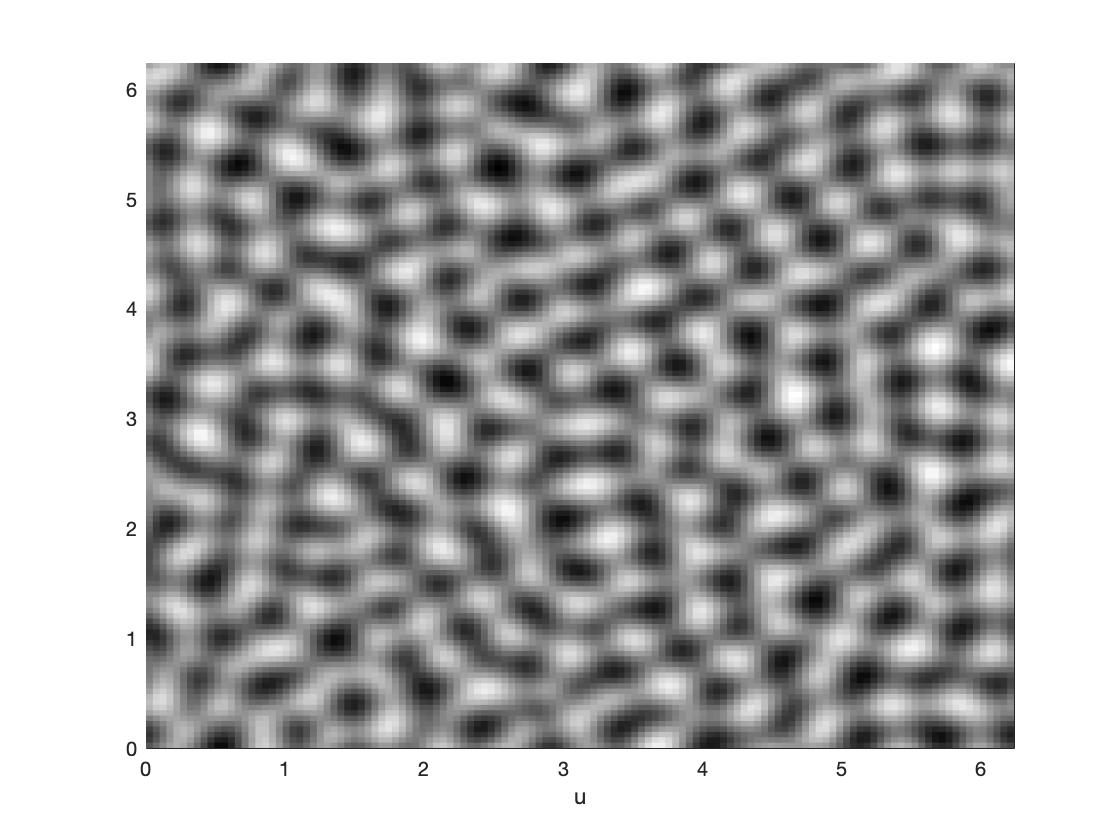}\hskip 0cm
\includegraphics[width=0.23\textwidth,clip==]{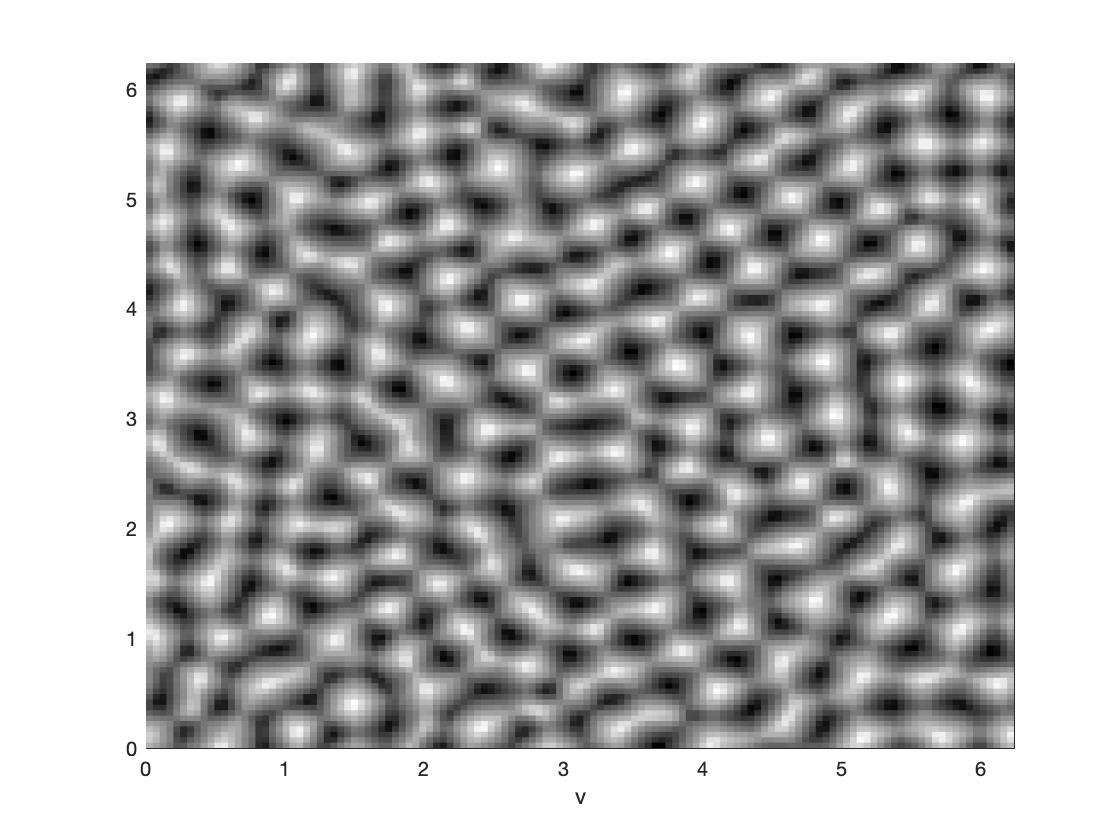}}
\subfigure[$t=5$.]{
\includegraphics[width=0.23\textwidth,clip==]{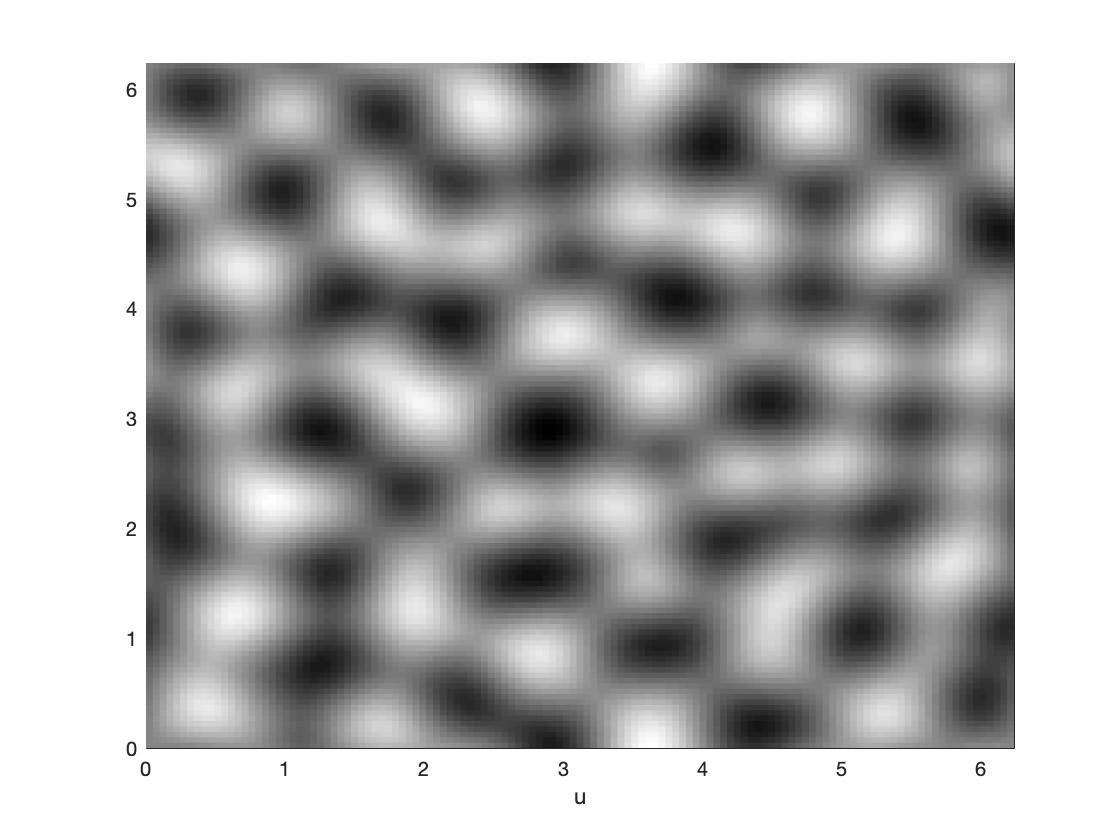}\hskip 0cm
\includegraphics[width=0.23\textwidth,clip==]{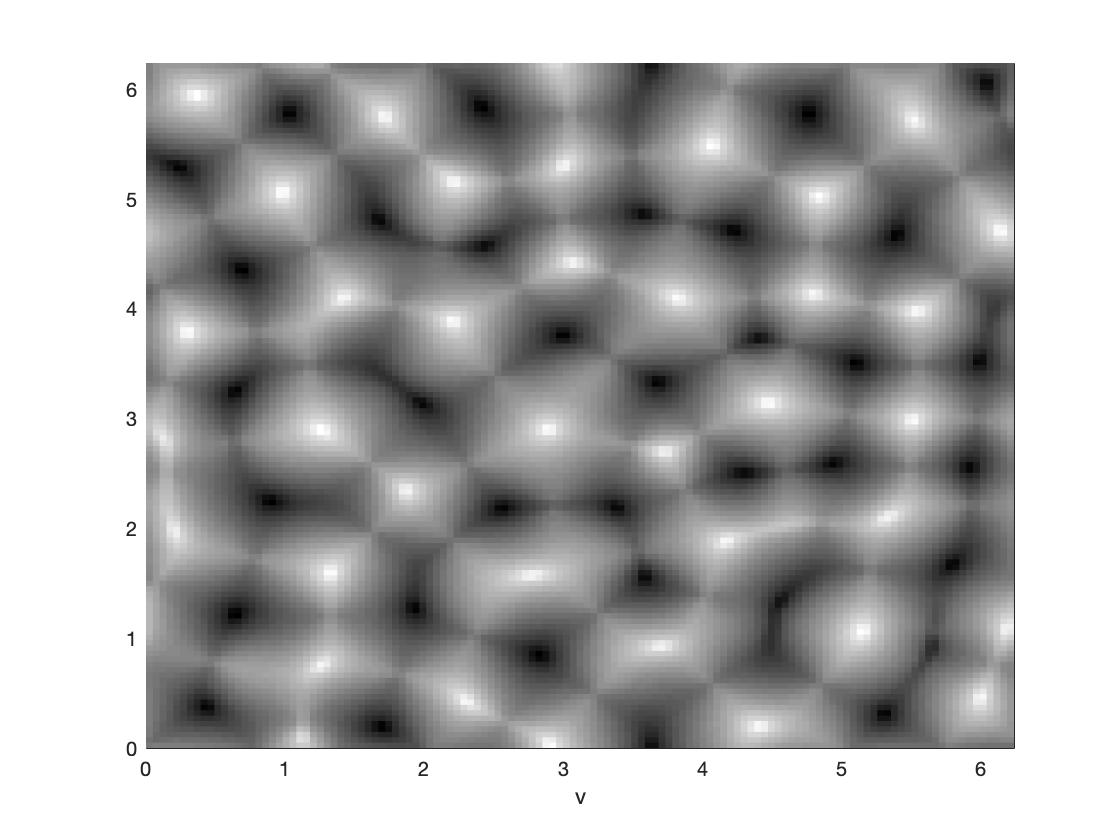}}
\subfigure[$t=10$.]{
\includegraphics[width=0.23\textwidth,clip==]{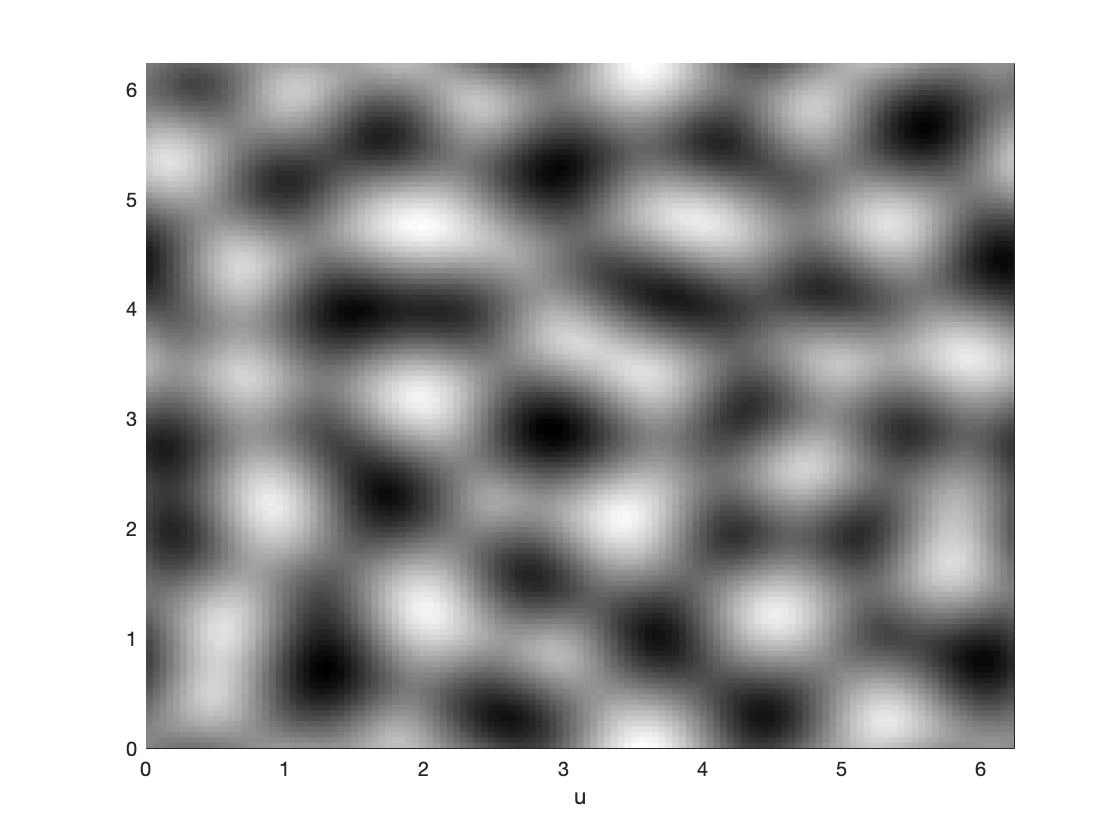}\hskip 0cm
\includegraphics[width=0.23\textwidth,clip==]{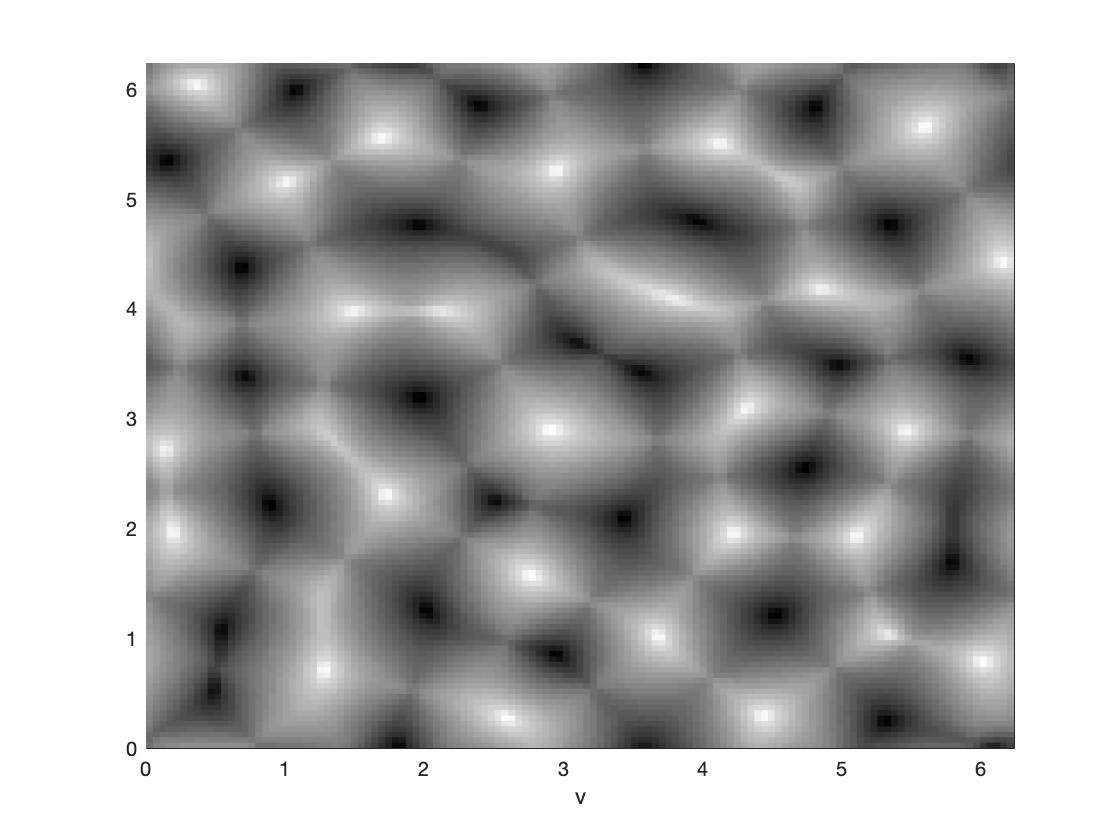}}
\subfigure[$t=30$.]{
\includegraphics[width=0.23\textwidth,clip==]{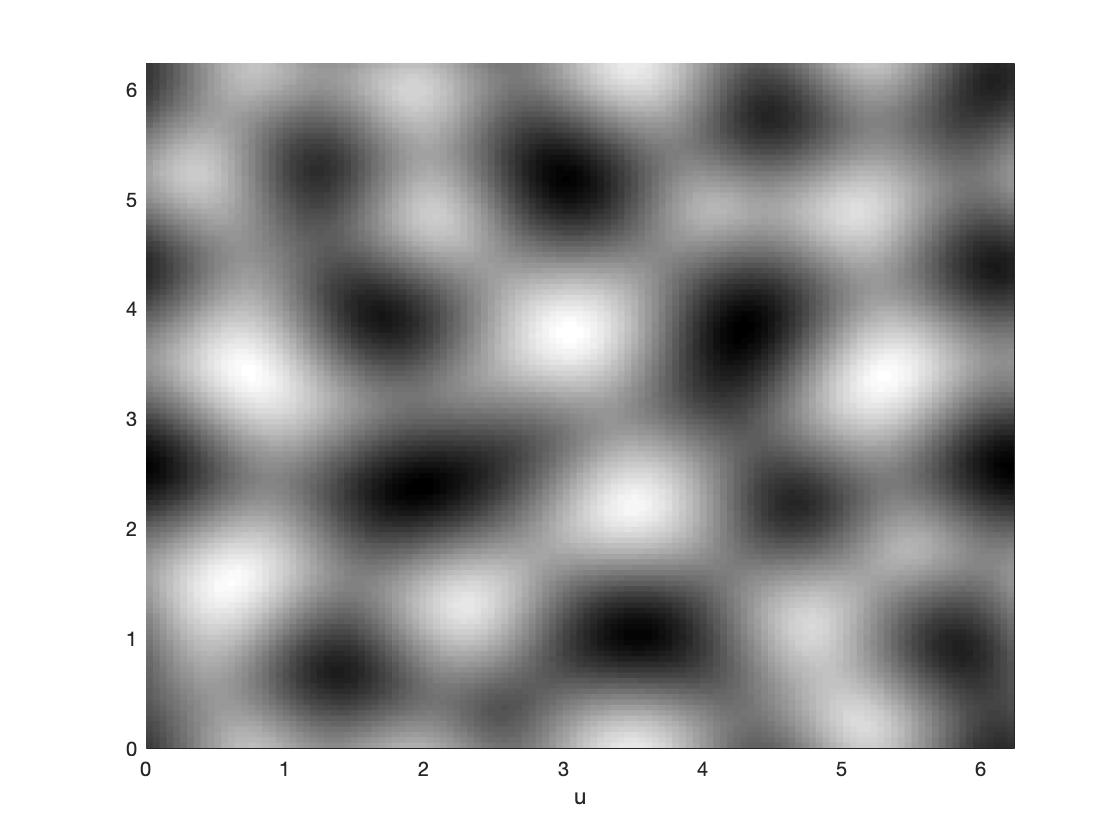}\hskip 0cm
\includegraphics[width=0.23\textwidth,clip==]{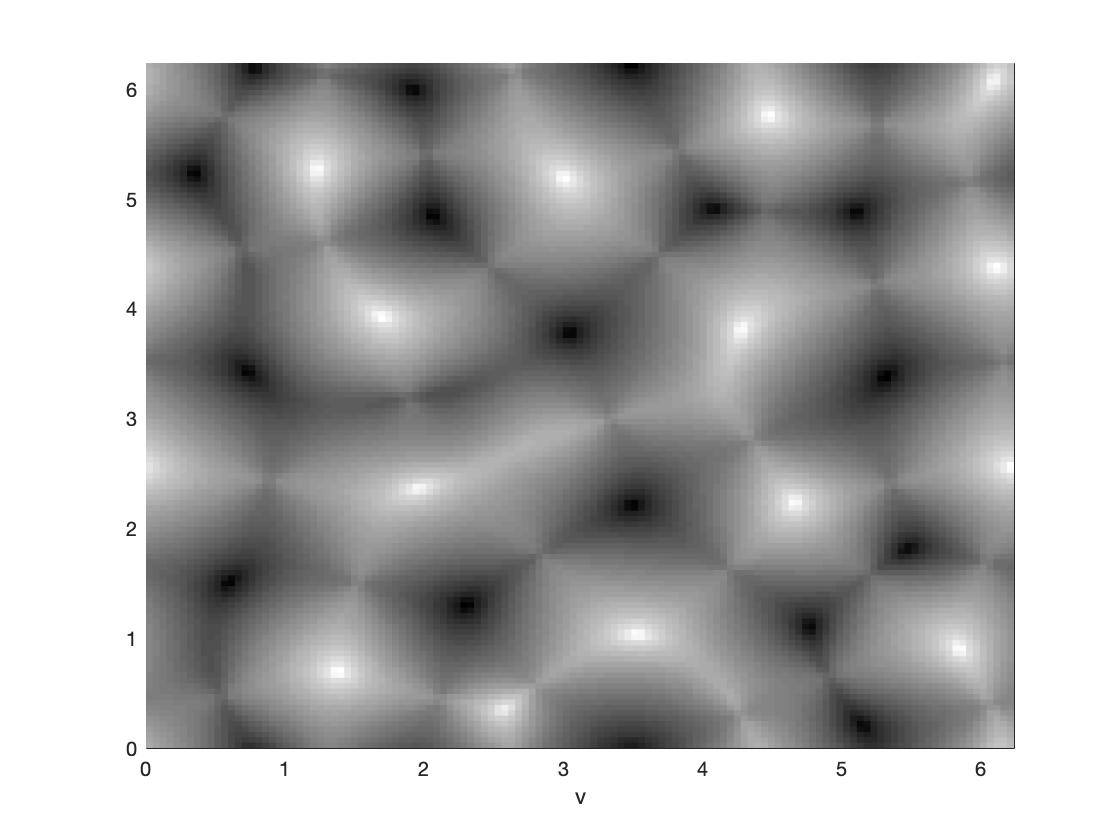}}
\subfigure[$t=50$.]{
\includegraphics[width=0.23\textwidth,clip==]{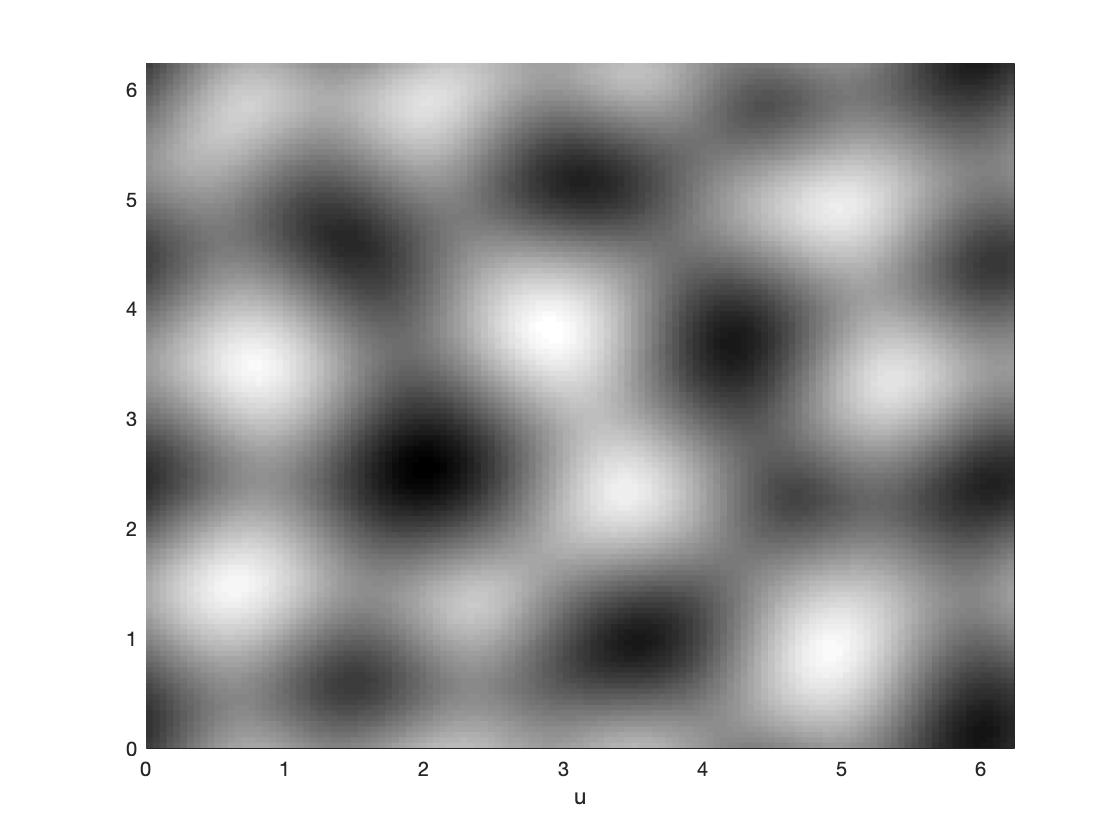}\hskip 0cm
\includegraphics[width=0.23\textwidth,clip==]{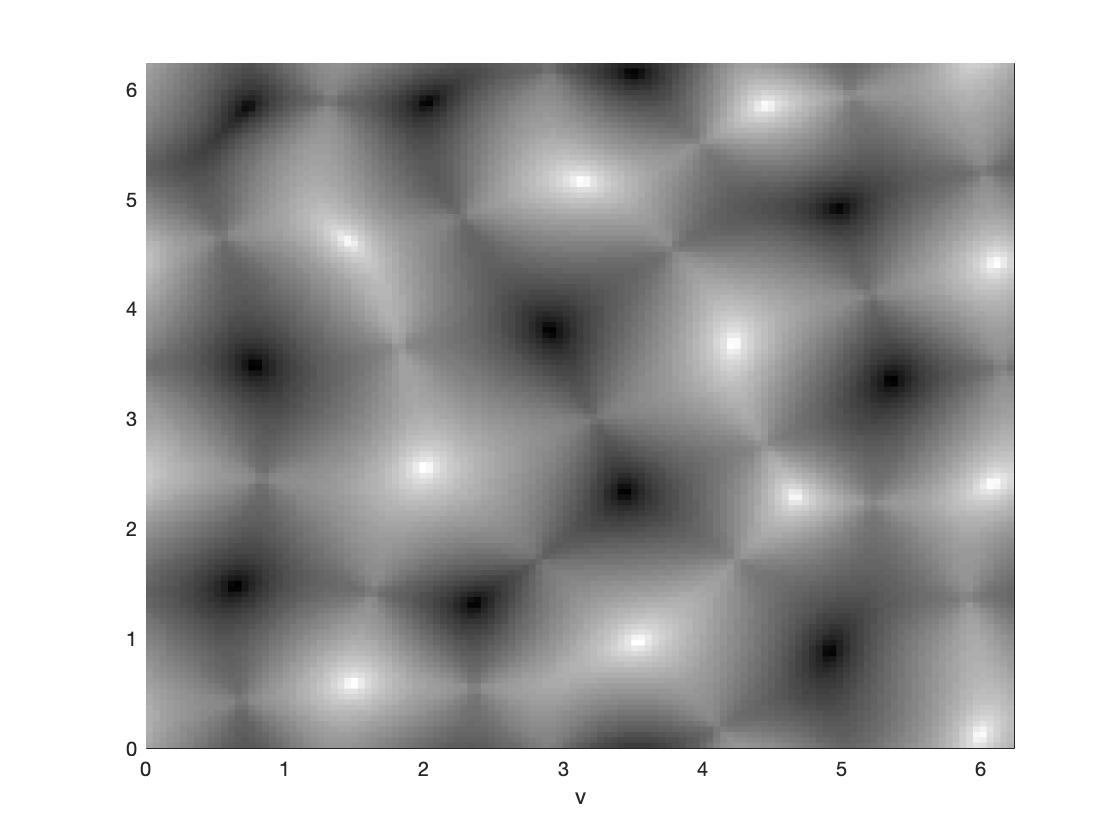}}
\subfigure[$t=200$.]{
\includegraphics[width=0.23\textwidth,clip==]{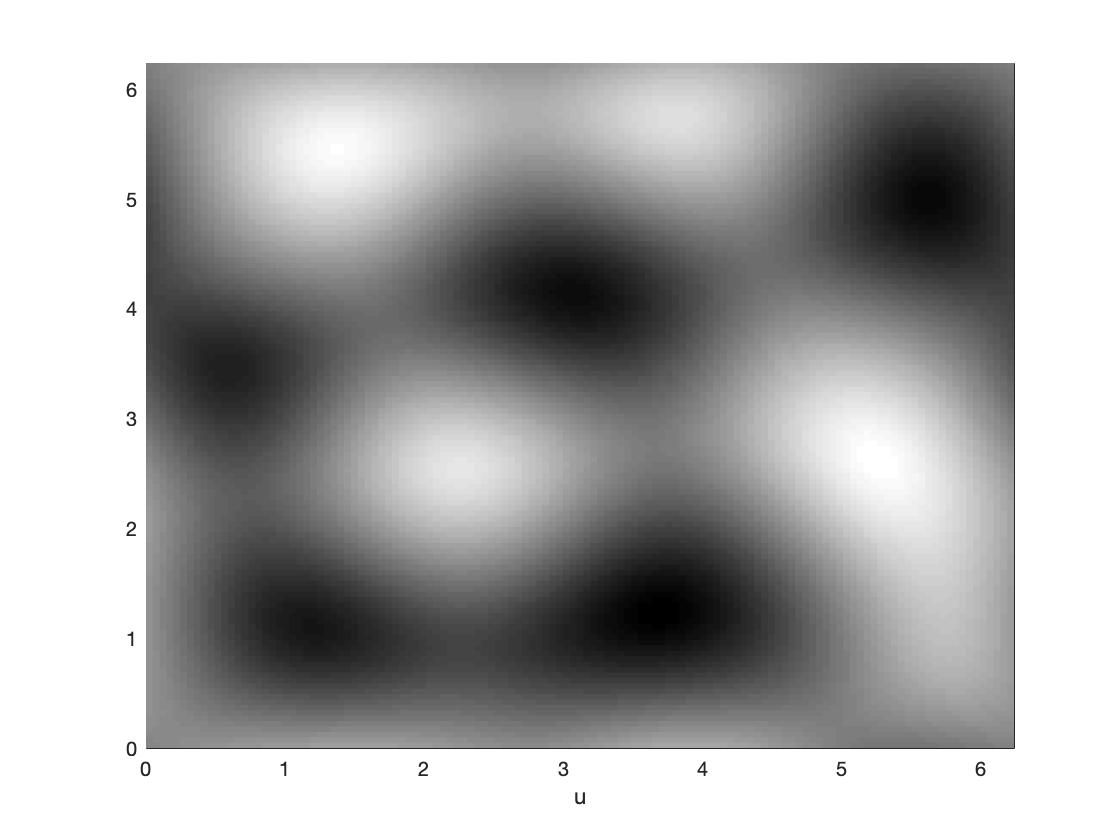}\hskip 0cm
\includegraphics[width=0.23\textwidth,clip==]{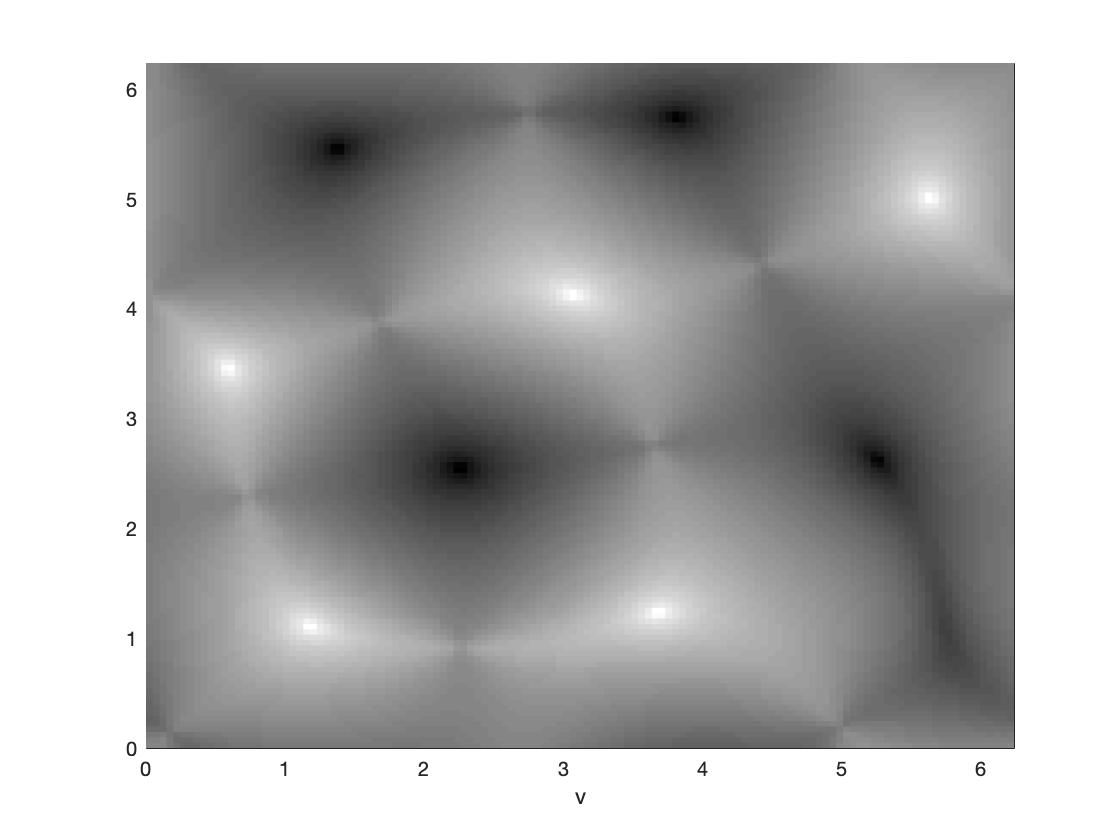}}
\caption{The isolines of the numerical solutions of the height function $\phi$ and its Laplacian $\Delta \phi$ for the model without slope selection with random initial condition. For each subfigure, the left is $\phi$ and the right is $\Delta \phi$ . Snapshots are taken at $t = 2.5 , 5 , 10 , 30 , 50 , 200$, respectively.}\label{NoSlopeCoarse}
\end{figure}

\begin{comment}

\subsubsection{ Comparison with the SAV approach}
In this subsection, we also compare our new Lagrange Multiplier approach with SAV approach for Cahn-Hilliard equation \eqref{ac_ch} with $\mathcal{L}=-\Delta$. We choose parameter $\eps^2=0.01$ in Fig.\,\ref{energy_comparsion_cahn}. It is observed from this figure both energy curves  of New Lagrange Multiplier and SAV approach with larger time step $\delta t=10^{-2}$  will deviate away from energy curves with smaller time step $\delta t=10^{-3}, 10^{-4}$. Especially for new Lagrange Multiplier approach may not converge when we solve the nonlinear algebra system if we choose a large time step, for example $\delta t=10^{-1},10^{-2}$.  We conclude that both new Lagrange Multiplier approach and SAV approach may admit  the same  admissiable time step range to obtain good accuracy. 

\begin{figure}[htbp]
\centering
%\includegraphics[width=0.45\textwidth,clip==]{adaptive_energy.jpg}
\includegraphics[width=0.45\textwidth,clip==]{Modify_Comparsion.jpg}
\includegraphics[width=0.45\textwidth,clip==]{Original_Comparsion.jpg}
\caption{Original and modified energy Comparison between new Lagrange Multiplier approach and SAV approach with various time step $\delta t$. }\label{energy_comparsion_cahn}
\end{figure}

\end{comment}

\section{Gradient flows  with multiple components}
As the SAV approach, the new Lagrange multiplier approach can be applied to  gradient flows with multiple components.
As a particular  example, we consider the following
 coupled non-local Cahn-Hilliard type system:
\begin{eqnarray}
&&u_t=M_u\Delta\mu_u, \label{ch:1}\\
&&\mu_u=-\eps_u^2\Delta u +(u^2-1)u + \alpha v+\beta v^2+2\gamma uv,\label{ch:2}\\
&&v_t=M_v\Delta\mu_v,\label{ch:3} \\
&&\mu_v=-\eps_v^2\Delta v +(v^2-1)v + \alpha u+2\beta uv+\gamma u^2-\sigma\Delta^{-1}(v-\overline v)\label{ch:4},
\end{eqnarray}
with boundary conditions 
\begin{eqnarray}
&&(i)\mbox{ periodic; or } (ii)\,\,\partial_{\bf n} u|_{\partial\Omega}=\partial_{\bf n}\Delta u|_{\partial\Omega}=
\partial_{\bf n} v|_{\partial\Omega}=\partial_{\bf n}\Delta v|_{\partial\Omega}=0.
\label{ori:pdefhbd3}
\end{eqnarray}

 The above system is introduced in \cite{avalos2016frustrated} to   model a blend of homopolymer and copolymer components. We defer the physical description about the system to Section \ref{BCP}, and consider how to apply the new Lagrange Multiplier approach to solve it efficiently.
 
 The system
\eqref{ch:1}-\eqref{ch:4} can be interpreted as a gradient flow as follows
\begin{eqnarray}
&&u_t=M_u\Delta\frac{\delta E(u,v)}{\delta u},  \label{ch2:1a}\\
&&v_t=M_v\Delta \frac{\delta E(u,v)}{\delta v}
,\label{ch2:3a} 
\end{eqnarray}
 with
 the  total free energy
 \begin{equation}
 E(u,v)=\int_{\Omega}\frac{\eps^2_u}{2}|\Grad u|^2+\frac{\eps^2_v}{2}|\Grad v|^2+W(u,v)+\frac{\sigma}{2}|(-\Delta)^{-\frac 12}(v-\overline v)|^2d\bx,
 \end{equation}
  where
  \begin{equation}\label{wuv}
W(u,v)=\frac{(u^2-1)^2}{4}+\frac{(v^2-1)^2}{4}+\alpha uv+\beta uv^2+\gamma u^2v.
\end{equation}
Indeed, one can easily check that 
$$\frac{\delta E(u,v)}{\delta u}=-\eps_u^2\Delta u +\frac{\delta W}{\delta u}=\mu_u,\quad \frac{\delta E(u,v)}{\delta v}=-\eps_u^2\Delta v +\frac{\delta W}{\delta v}-\sigma\Delta^{-1}(v-\overline v)=\mu_v.$$
\subsection{New schemes based on the  Lagrange multiplier approach}
 
We introduce a Lagrange multiplier $\eta(t)$ and rewrite \eqref{ch2:1a}-\eqref{ch2:3a} as
\begin{eqnarray}
&&u_t=M_u\Delta\mu_u, \label{ch2:1}\\
&&\mu_u=-\eps_u^2\Delta u+\frac{\delta W}{\delta u} \eta(t) ,\label{ch2:2}\\
&&v_t=M_v\Delta\mu_v
,\label{ch2:3} \\
&&\mu_v=-\eps_u^2\Delta v +\frac{\delta W}{\delta v}\eta(t)-\sigma\Delta^{-1}(v-\overline v),\label{ch2:4}\\
&&\frac{d}{dt}\int_{\Omega}W(u,v) d\bx=\eta(t) \int_\Omega (\frac{\delta W}{\delta u} u_t+\frac{\delta W}{\delta v}v_t )d\bx\label{ch2:5}.
\end{eqnarray}
Taking the inner products of \eqref{ch2:1}-\eqref{ch2:4} with $\mu_u$, $u_t$, $\mu_v$, $v_t$, respectively, summing up the results along with \eqref{ch2:5}, we obtain the energy dissipation law
\begin{equation*}
\frac d{dt}  E(u,v)=-M_u(\nabla \mu_u, \nabla \mu_u)-M_v(\nabla \mu_v, \nabla \mu_v). 
\end{equation*}

As in the previous section, we can construct a second-order scheme for the above system  as follows:
assuming that $u^{n-1}$, $u^{n}$ and $v^{n-1}$, $v^{n}$ are known, we find
$u^{n+1}$ and  $v^{n+1}$ as follows:
\begin{eqnarray}
&&\frac{3u^{n+1}-4u^n+u^{n-1}}{2\delta t}=M_u\Delta\mu^{n+1}_u, \label{ch:bdf2:1}\\
&&\mu^{n+1}_u=-\eps_u^2\Delta u^{n+1} +(\frac{\delta W}{\delta u})^{\star,n}\eta^{n+1},\label{ch:bdf2:2}\\
&&\frac{3v^{n+1}-4v^n+v^{n-1}}{2\delta t}=M_v\Delta\mu^{n+1}_v,\label{ch:bdf2:3} \\
&&\mu^{n+1}_v=-\eps_v^2\Delta v^{n+1} +(\frac{\delta W}{\delta v})^{\star,n}\eta^{n+1}-\sigma\Delta^{-1}(v^{n+1}-\overline v)\label{ch:bdf2:4},\\
&& (3W(u^{n+1},v^{n+1})-4W(u^n,v^n)+W(u^{n-1},v^{n-1}),1)\label{ch:bdf2:5}\\
&&\hskip 1cm = \eta^{n+1}\{((\frac{\delta W}{\delta u})^{\star,n},3u^{n+1}-4u^n+u^{n-1})+((\frac{\delta W}{\delta v})^{\star,n},3v^{n+1}-4v^n+v^{n-1})\},\nonumber
\end{eqnarray}
where $f^{\star,n}=2f^n-f^{n-1}$ for any function $f$, and the boundary conditions are given by \eqref{ori:pdefhbd3}. 
%\end{algorithm}

\begin{thm}
The numerical scheme \eqref{ch:bdf2:1}-\eqref{ch:bdf2:5} is unconditional energy stable  in the sense that 
\begin{eqnarray}\label{bdf2:stab}
\frac{E^{n+1}-E^n}{\delta t}\le -(M_u\|\Grad\mu_u^{n+1}\|+M_v\|\Grad\mu^{n+1}_v\|^2),
\end{eqnarray}
where
\begin{equation*}
\begin{split} 
E^{n+1}=&\int_{\Omega}\frac{\eps^2_u}{2}|\Grad u^{n+1}|^2+\frac{\eps^2_v}{2}|\Grad v^{n+1}|2+\frac 12\big(3W(u^{n+1},v^{n+1})-W(u^{n},v^{n})\big)\\
&+\frac{\sigma}{2}|(-\Delta)^{-\frac 12}(v^{n+1}-\overline v^{n+1})|^2d\bx.
\end{split}
\end{equation*}
\end{thm}
\begin{proof}
Taking the inner products of  \eqref{ch:bdf2:1} and \eqref{ch:bdf2:3} with $\mu_u^{n+1}$ and $\mu_v^{n+1}$, respectively, we find
\begin{eqnarray}
&&(\frac{3u^{n+1}-4u^n+u^{n-1}}{2\delta t}, \mu_{u}^{n+1})=-M_u\|\Grad\mu_u^{n+1}\|^2,\\
&&(\frac{3v^{n+1}-4v^n+v^{n-1}}{2\delta t}, \mu_{v}^{n+1})=-M_v\|\Grad\mu_v^{n+1}\|^2.
\end{eqnarray}
Taking the inner products of  \eqref{ch:bdf2:2}  and \eqref{ch:bdf2:4} with $3u^{n+1}-4u^n+u^{n-1}$ and  $3v^{n+1}-4v^n+v^{n-1}$, respectively, and using the identity \eqref{iden1}, we find
\begin{eqnarray*}\label{bdf2:pr:1}
\begin{aligned}
&(\mu_u^{n+1},3u^{n+1}-4u^n+u^{n-1})=\frac{\eps^2_u}{2}\{(\|\Grad u^{n+1}\|^2-\|2\Grad u^{n+1}-\Grad u^n\|^2)\\&-(\|\Grad u^{n}\|^2-\|2\Grad u^{n}-\Grad u^{n-1}\|^2)+
\|\Grad u^{n+1}-2\Grad u^n+\Grad u^{n-1}\|^2\}\\&
+((\frac{\delta W}{\delta u})^{\star,n}\eta^{n+1},3u^{n+1}-4u^n+u^{n-1}),
\end{aligned}
\end{eqnarray*}
and
\begin{eqnarray*}\label{bdf:pr:2}
\begin{aligned}
&(\mu_v^{n+1},3v^{n+1}-4v^n+v^{n-1})=\frac{\eps^2_v}{2}\{(\|\Grad v^{n+1}\|^2-\|2\Grad v^{n+1}-\Grad v^n\|^2)\\&-(\|\Grad v^{n}\|^2-\|2\Grad v^{n}-\Grad v^{n-1}\|^2)+
\|\Grad v^{n+1}-2\Grad v^n+\Grad v^{n-1}\|^2\}\\&
+((\frac{\delta W}{\delta v})^{\star,n}\eta^{n+1},3v^{n+1}-4v^n+v^{n-1}).
\end{aligned}
\end{eqnarray*}
Combing the above relations  and using \eqref{ch:bdf2:5},
we obtain the desired result.
\end{proof}

\begin{comment}
\begin{remark}
Since we can prove  $\frac{E^{n+1}-E^n}{\delta t}=\frac{d}{dt}E(u,v)|_{t^{n+1}}+O(\delta^2 t)$ for equation \eqref{bdf2:stab}, notice that
\begin{eqnarray}
\begin{aligned}
&\frac{1}{2\delta t}\{W(u^{n+1},v^{n+1})-W(u^n,v^n),1)+(W(2u^{n+1}-u^n,2v^{n+1}-v^n)\\&-W(2u^n-u^{n-1},2v^n-v^{n-1}),1)\} \\&= \frac{(W(u^{n+2},v^{n+2})-W(u^n,v^n),1)}{2\delta t} \approx \frac{d}{dt}\int_{\Omega}W(u^{n+1},v^{n+1} )dx + O(\delta t^2).
\end{aligned}
\end{eqnarray}
Which can be derived by using approximation of  $W(2u^n-u^{n-1},2v^n-v^{n-1})=W(u^{n+1},v^{n+1})+0(\delta^2 t)$ and 
$W(2u^{n+1}-u^{n},2v^{n+1}-v^{n})=W(u^{n+2},v^{n+2})+0(\delta^2 t)$.
\end{remark}
\end{comment}

Next,  we describe how to solve this coupled system \eqref{ch:bdf2:1}-\eqref{ch:bdf2:5}  efficiently. We first eliminate $\mu_u$ and $\mu_v$ to  rewrite equations  \eqref{ch:bdf2:1}-\eqref{ch:bdf2:2} and  \eqref{ch:bdf2:3}-\eqref{ch:bdf2:4} as
\begin{eqnarray}\label{solu:u}
\frac{3u^{n+1}}{2\delta t}+M_u\eps_u^2\Delta^2 u^{n+1} =\frac{4u^{n}-u^{n-1}}{2\delta t}+M_u\Delta(\frac{\delta W}{\delta u})^{\star,n}\eta^{n+1},
\end{eqnarray}
\begin{eqnarray}\label{solu:v}
\frac{3v^{n+1}}{2\delta t}+M_v\eps_v^2\Delta^2 v^{n+1} =\frac{4v^{n}-v^{n-1}}{2\delta t}+M_v\Delta(\frac{\delta W}{\delta v})^{\star,n}\eta^{n+1}.
\end{eqnarray}
Define two linear operators 
 $$\chi_uu:=(\frac{3}{2\delta t}+M_u\eps_u^2\Delta^2) u,\quad \chi_vv:=(\frac{3}{2\delta t}+M_v\eps_v^2\Delta^2) v,$$
 with boundary conditions specified in \eqref{ori:pdefhbd3}, and 
apply $\chi_u^{-1}$ and $\chi_v^{-1}$ to  equations \eqref{solu:u} and \eqref{solu:v}, respectively,   we obtain 
\begin{equation}\label{u:update}
\begin{split}
u^{n+1}&=\chi_u^{-1}\{\frac{4u^{n}-u^{n-1}}{2\delta t}+M_u\Delta(\frac{\delta W}{\delta u})^{\star,n}\eta^{n+1}\}.
\\&:=p_u^n+\eta^{n+1}q_u^n,
\end{split}
\end{equation}
\begin{equation}\label{v:update}
\begin{split}
v^{n+1}&=\chi_v^{-1}\{\frac{4v^{n}-v^{n-1}}{2\delta t}+M_v\Delta(\frac{\delta W}{\delta v})^{\star,n}\eta^{n+1}\}
\\&:=p_v^n+\eta^{n+1}q_v^n.
\end{split}
\end{equation}
where
\begin{equation}\label{pu_qu}
\begin{split}
& p_u^n=\chi_u^{-1}\{\frac{4u^{n}-u^{n-1}}{2\delta t}\}, \quad q_u^n=\chi_u^{-1}\{M_u\Delta(\frac{\delta W}{\delta u})^{\star,n}\},\\
& p_v^n=\chi_v^{-1}\{\frac{4v^{n}-v^{n-1}}{2\delta t}\}, \quad q_v^n=\chi_v^{-1}\{M_v\Delta(\frac{\delta W}{\delta v})^{\star,n}\}.
\end{split}
\end{equation}
Then  plugging $u^{n+1}$ and $v^{n+1}$ in \eqref{u:update}-\eqref{v:update} into \eqref{ch:bdf2:5},    we obtain a nonlinear algebraic equation for $\eta^{n+1}$:
\begin{equation}\label{nonlinear:system}
\begin{split}
&(\frac{\delta W}{\delta u})^{\star,n}\eta^{n+1},3(p_u^n+\eta^{n+1}q_u^n))+((\frac{\delta W}{\delta v})^{\star,n}\eta^{n+1},3(p_v^n+\eta^{n+1}q_v^n))\\&-(3W(p_u^n+\eta^{n+1}q_u^n,p_v^n+\eta^{n+1}q_v^n),1)=g^{\star,n},
\end{split}
\end{equation}
where  $g^{\star,n}$ represents all  known explicit terms.

Since $W(u,v)$ is a polynomial of degree four for $u$ and $v$,  we find that \eqref{nonlinear:system} is  a nonlinear algebraic equation of degree four for $\eta^{n+1}$, which can be easily solved by using  a Newton iteration with $1$ as the initial condition. To summarize, 
the new Lagrange multiplier algorithm for the  coupled Cahn-Hilliard system consists of the following steps:

\begin{description}
 \item  [Step 1]   Compute $p_u^n$, $q_u^n$, $p_v^{n}$ and $q_v^n$ from  \eqref{pu_qu};         
  \item [Step 2]  Solve equation \eqref{nonlinear:system}  to obtain $\eta^{n+1}$;
  \item [Step 3]  Update $u^{n+1}$ and $v^{n+1}$ with \eqref{solu:u} and \eqref{solu:v}, and goto the next step.  
\end{description}
The main cost of the  scheme \eqref{ch:bdf2:1}-\eqref{ch:bdf2:5}  at each time step is in {\bf Step 1} where one needs  to solve four linear, constant coefficient fourth-order equations. Hence, the scheme is very efficient.

\subsection{Numerical validations}
we will first present  some numerical results to validate the stability and the accuracy of the scheme \eqref{ch:bdf2:1}-\eqref{ch:bdf2:5} using a Fourier spectral method with 
$128^2$ modes in the spatial domain $\Omega=[0,2\pi)^2$.

We first test the accuracy with a manufactured
exact solution  given by 
\begin{eqnarray*}\label{exact}
\begin{aligned}
&&u(x,y,t)=(\frac{\sin(2x)\sin(2y)}{4}+0.48)(1-\frac{\sin^2(t)}{2}),\\
&&v(x,y,t)=(\frac{\sin(2x)\sin(2y)}{4}+0.48)(1-\frac{\sin^2(t)}{2}).
\end{aligned}
\end{eqnarray*}
We set the parameters to be
\begin{equation*}
\eps_u=0.075 \quad \eps_v=0.075 \quad \sigma =10 \quad \alpha=-0.1 \quad \beta =-0.1 \quad \gamma=0.
\end{equation*} 

In Table \ref{table1} and Table \ref{table2} , we list the $L^{\infty}$ errors at $t=0.1$ of  phase variables $u$ and $v$ with various time steps by using   new schemes based on BDF2 and Crank-Nicolson. We observe that both  schemes  have  second-order convergence rate in time.

\begin{table}[ht!]
\centering
\begin{tabular}{r||c|c|c|c}
\hline
$\delta t$            & {$u$}  & Order & {$v$} & Order    \\ \hline
$4\times 10^{-3}$    &$9.87E(-6)$  & $-$   &$8.29E(-6)$ &$-$       \\\hline
$2\times 10^{-3}$    &$2.46E(-6)$  & $2.00$  &$2.07E(-6)$ &$2.00$       \\\hline
$1\times 10^{-3}$     &$6.16E(-7)$ & $1.99$&$5.17E(-7)$ &$2.00$    \\\hline
$5\times 10^{-4}$     &$1.54E(-7)$ &$2.00$  &$1.29E(-7)$ &$2.00$  \\\hline
$2.5\times 10^{-4}$  &$3.85E(-8)$ &$1.99$ &$3.23E(-8)$ &$1.99$   \\ \hline
$1.25\times 10^{-4}$  &$9.63E(-9)$&$1.99$ &$8.07E(-9)$ &$1.99$ \\\hline
$6.25\times 10^{-5}$  &$2.40E(-9)$&$2.00$ &$2.01E(-9)$ &$2.00$ \\\hline
\hline
\end{tabular}
\vskip 0.5cm
\caption{Accuracy test: with given exact solution for the coupled  model \eqref{ch:1}-\eqref{ch:4}. The $L^{\infty}$ errors at $t=0.1$ for the  phase variables $u$ and $v$ computed by the scheme based on BDF2 using various time steps.}\label{table1}
\end{table}

\begin{table}[ht!]
\centering
\begin{tabular}{r||c|c|c|c}
\hline
$\delta t$            & {$u$}  & Order & {$v$} & Order    \\ \hline
$4\times 10^{-3}$    &$3.69E(-6)$  & $-$   &$3.03E(-6)$ &$-$       \\\hline
$2\times 10^{-3}$    &$9.80E(-7)$  & $1.91$  &$8.01E(-7)$ &$1.92$       \\\hline
$1\times 10^{-3}$     &$2.50E(-7)$ & $1.97$&$2.04E(-7)$ &$1.97$    \\\hline
$5\times 10^{-4}$     &$6.32E(-8)$ &$1.98$  &$5.15E(-8)$ &$1.98$  \\\hline
$2.5\times 10^{-4}$  &$1.60E(-8)$ &$1.98$ &$1.30E(-8)$ &$1.98$   \\ \hline
$1.25\times 10^{-4}$  &$3.98E(-9)$&$2.00$ &$3.24E(-9)$ &$2.00$ \\\hline
$6.25\times 10^{-5}$  &$9.34E(-10)$&$2.09$ &$7.77E(-10)$ &$2.00$ \\\hline
\hline
\end{tabular}
\vskip 0.5cm
\caption{Accuracy test: with given exact solution for the coupled  model \eqref{ch:1}-\eqref{ch:4}. The $L^{\infty}$ errors at $t=0.1$ for the  phase variables $u$ and $v$ computed by the scheme based on Crank-Nicolson  using various time steps.}\label{table2}
\end{table}

Next, we examine the energy dissipation and accuracy using a realistic simulation. Since excessively small time steps may be needed to obtain accurate results using the scheme \eqref{ch:bdf2:1}-\eqref{ch:bdf2:5}, we construct a "stabilized" version based on the  the following splitting of free energy:
\begin{equation}
\begin{split}
 E(u,v)=&\int_{\Omega}\frac{\eps^2_u}{2}|\Grad u|^2+\frac{\eps^2_v}{2}|\Grad v|^2+\frac{S_u}2|u|^2+\frac{S_v}2|v|^2 \bx\\
 &+\int_{\Omega} W(u,v)-\frac{S_u}2|u|^2-\frac{S_v}2|v|^2 +\frac{\sigma}{2}|(-\Delta)^{-\frac 12}(v-\overline v)|^2d\bx,
 \end{split}
 \end{equation}
  where $W(u,v)$ is given in \eqref{wuv} and $S_u,S_v\ge 0$ are stabilizing constants.

Setting $\tilde W(u,v)=W(u,v)-\frac{S_u}2|u|^2-\frac{S_v}2|v|^2$,  the "stablized" version of the scheme \eqref{ch:bdf2:1}-\eqref{ch:bdf2:5} corresponding to the above splitting is:
\begin{eqnarray}
&&\frac{3u^{n+1}-4u^n+u^{n-1}}{2\delta t}=M_u\Delta\mu^{n+1}_u, \label{ch:bdf2:1b}\\
&&\mu^{n+1}_u=-\eps_u^2\Delta u^{n+1}+S_u u^{n+1} +(\frac{\delta \tilde W}{\delta u})^{\star,n}\eta^{n+1},\label{ch:bdf2:2b}\\
&&\frac{3v^{n+1}-4v^n+v^{n-1}}{2\delta t}=M_v\Delta\mu^{n+1}_v,\label{ch:bdf2:3b} \\
&&\mu^{n+1}_v=-\eps_v^2\Delta v^{n+1} +S_v v^{n+1}+(\frac{\delta \tilde W}{\delta v})^{\star,n}\eta^{n+1}-\sigma\Delta^{-1}(v^{n+1}-\overline v)\label{ch:bdf2:4b},\\
&& (3\tilde  W(u^{n+1},v^{n+1})-4\tilde W(u^n,v^n)+\tilde W(u^{n-1},v^{n-1}),1)\label{ch:bdf2:5b}\\
&&\hskip 1cm = \eta^{n+1}\{((\frac{\delta \tilde W}{\delta u})^{\star,n},3u^{n+1}-4u^n+u^{n-1})+((\frac{\delta \tilde W}{\delta v})^{\star,n},3v^{n+1}-4v^n+v^{n-1})\}.\nonumber
\end{eqnarray}
Obviously, we can prove that the above scheme is also unconditionally energy stable as  in Theorem \ref{bdf:thm}, and can be efficiently implemented as the scheme \eqref{ch:bdf2:1}-\eqref{ch:bdf2:5}.

We take  uniformly distributed random functions in $[-1,1]^2$ with zero mean  as  initial conditions, and set 
  the parameters to be
\begin{equation}
\eps_u=0.05 \quad \eps_v=0.05 \quad \sigma =10 \quad \alpha=-0.1 \quad \beta =0.75 \quad \gamma=0,
\end{equation}
and the stabilized constants are  $S_u=S_v=1$.  We plot in Fig.\,\ref{energy_compare} evolution of the total energy computed with different time steps. We observe that all  energies  decay with time as expected.   For $\delta t=8\times 10^{-4}, 4\times 10^{-4},  2\times 10^{-4}, 10^{-4}$,  the four energy curves essentially coincide.  However, significant differences appear with   $\delta t=2\times 10^{-3}$ and  $\delta t=8\times 10^{-3}$.   So in order to achieve good accuracy for numerical simulations, we should  choose  $\delta t\leq  8\times 10^{-4}$. Note that the required time step would be much smaller without the stabilization, i.e., $S_u=S_v=0$.
%\subsubsection{Convergence test for coupled Cahn-Hilliard equations}

\begin{figure}[htbp]
\centering
\includegraphics[width=0.60\textwidth,clip==]{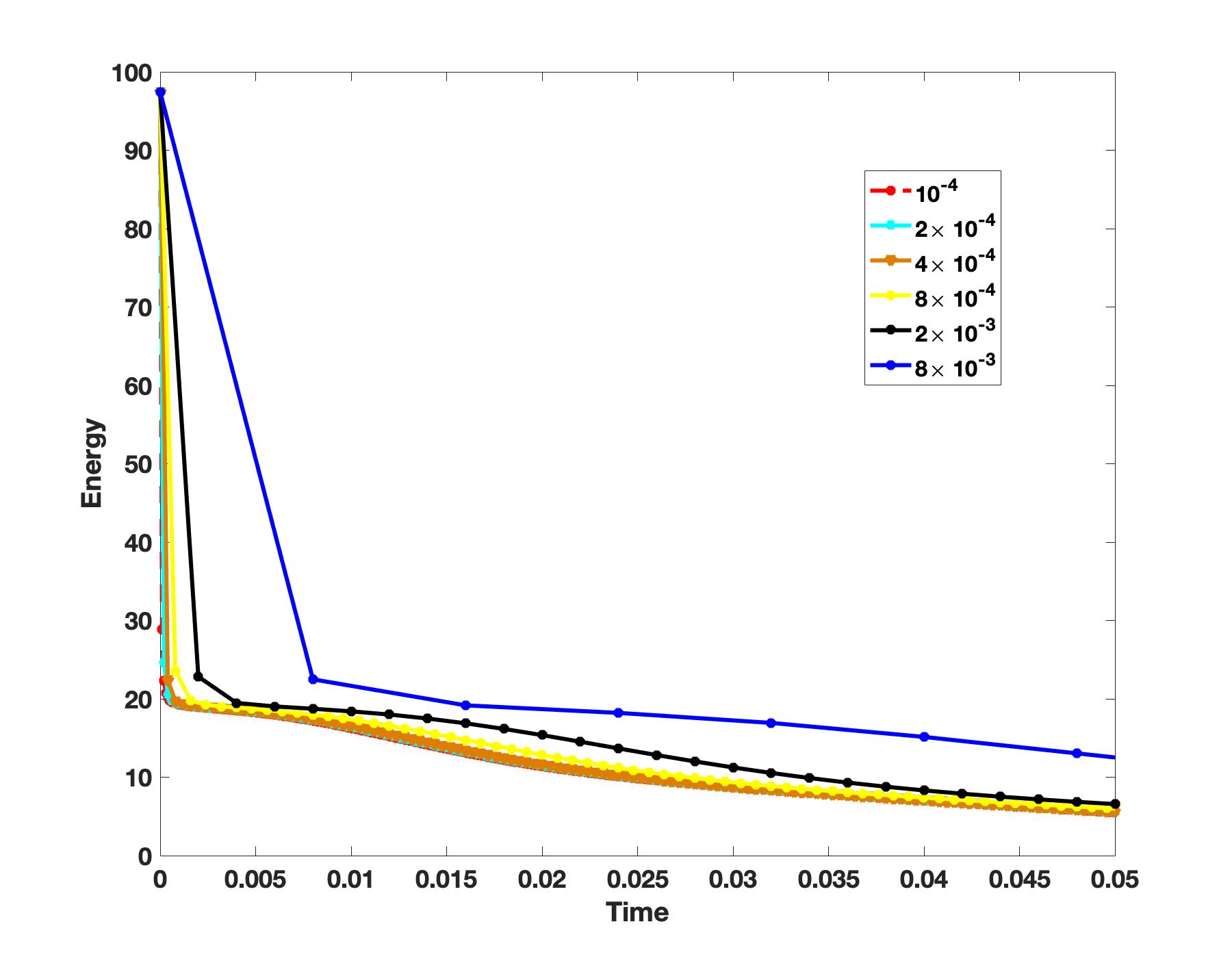}
\caption{Evolution of the free energy using the stabilized scheme with various time steps.}\label{energy_compare}
\end{figure}

\section{Adaptive time stepping}
A main advantage of unconditional energy stable numerical schemes is that one can combine them with  an adaptive time stepping so that their efficiency can be further improved.  This has been demonstrated with many examples for  the SAV approach \cite{shen2017new}. We can also combine an adaptive time stepping strategy with the new Lagrange multiplier approach. 

Since it is more  difficult to devise an adaptive time stepping for the scheme based on BDF2, we  construct below a second-order scheme with variable time steps based on a second-order Crank-Nicolson scheme. To fix the idea, we take the coupled problem 
\eqref{ch:1}-\eqref{ch:4} as an example.

A stabilized second-order Crank-Nicolson scheme with variable time step is as follows:
 assuming that $u^{n-1}$, $u^{n}$ and $v^{n-1}$, $v^{n}$ are known,  set $a_n=\frac{\delta t_{n}}{\delta t_{n-1}}$ and  $\delta t_n=t^{n+1}-t^n$,   we solve
$u^{n+1}$ and  $v^{n+1}$ as follows:
%\begin{eqnarray}
%&&\frac{(2a+1)u^{n+1}-(a+1)^2u^n+a^2u^{n-1}}{a(a+1)\delta t^{n}}=M_u\Delta\mu^{n+a}_u \label{ad:bdf2:1}\\
%&&\mu^{n+a}_u=-\eps_u^2\Delta u^{n+1} +(\frac{\delta W}{\delta u})^{\star,n}\eta^{n+1}-A\delta t(u^{n+1}-u^n),\label{ad:bdf2:2}\\
%&&\frac{(2a+1)v^{n+1}-(a+1)^2v^n+a^2v^{n-1}}{a(a+1)\delta t^{n}}=M_v\Delta\mu^{n+a}_v,\label{ad:bdf2:3} \\
%&&\mu^{n+a}_v=-\eps_v^2\Delta v^{n+1} +(\frac{\delta W}{\delta v})^{\star,n}\eta^{n+1}-\sigma\Delta^{-1}(v^{n+1}-\overline v)\label{ad:bdf2:4},\\
%&&((\frac{\delta W}{\delta u})^{\star,n}\eta^{n+1},\frac{2a+1}{a}u^{n+1}-\frac{(a+1)^2}{a}u^n+au^{n-1})\\&&+((\frac{\delta W}{\delta v})^{\star,n}\eta^{n+1},\frac{2a+1}{a}v^{n+1}-\frac{(a+1)^2}{a}v^n+av^{n-1})\nonumber\\&&=(W(u^{n+1},v^{n+1})-W(u^n,v^n),1)+(W((a+1)u^{n+1}-au^n,(a+1)v^{n+1}-av^n)\nonumber\\&&-W((a+1)u^n-au^{n-1},(a+1)v^n-av^{n-1}),1) \nonumber.\label{ad:bdf2:5}
%\end{eqnarray}
\begin{eqnarray}
&&\frac{u^{n+1}-u^n}{\delta t_{n}}=M_u\Delta\mu^{n+\frac 12}_u, \label{ad:ch:cn1}\\
&&\mu^{n+\frac 12}_u=-\eps_u^2\Delta u^{n+\frac 12}+S_u u^{n+\frac 12} +(\frac{\delta \tilde W}{\delta u})^{\star,n}\eta^{n+\frac 12},\label{ad:ch:cn2}\\
&&\frac{v^{n+1}-v^n}{\delta t_n}=M_v\Delta\mu^{n+\frac 12}_v,\label{ad:ch:cn3} \\ %+S_u(u^{n+1}-u^n)
&&\mu^{n+\frac 12}_v=-\eps_v^2\Delta v^{n+\frac 12} +S_v v^{n+\frac 12} +(\frac{\delta \tilde W}{\delta v})^{\star,n}\eta^{n+\frac 12}-\sigma\Delta^{-1}(v^{n+\frac 12}-\overline v)\label{ad:ch:cn4},\\  %+S_v(v^{n+1}-v^n)
&&\eta^{n+\frac 12}((\frac{\delta\tilde  W}{\delta u})^{\star,n},u^{n+1}-u^n)+((\frac{\delta \tilde W}{\delta v})^{\star,n},v^{n+1}-v^n)\}\nonumber\\&&\hskip 1 cm=(\tilde W(u^{n+1},v^{n+1})-\tilde W(u^n,v^n),1),\label{ad:ch:cn5}
\end{eqnarray}
where $f^{n+\frac 12}=\frac 12(f^{n+1}+f^n)$ and $f^{\star,n}=(1+\frac{a_n}{2})f^n-\frac{a_n}{2}f^{n-1}$  for any function $f$.

The above scheme can also be efficiently implemented as the scheme \eqref{ch:bdf2:1}-\eqref{ch:bdf2:5}, and the following result  can be easily established.
\begin{thm}
The  scheme \eqref{ad:ch:cn1}-\eqref{ad:ch:cn5} is unconditional energy stable  in the sense that 
\begin{eqnarray}
\frac{E^{n+1}-E^n}{\delta t_n}= -(M_u\|\Grad\mu_u^{n+\frac 12}\|+M_v\|\Grad\mu^{n+\frac 12}_v\|^2),
\end{eqnarray}
where $E^{n+1}=\int_{\Omega}\frac{\eps^2_u}{2}|\Grad u^{n+1}|^2+\frac{\eps^2_v}{2}|\Grad v^{n+1}|2+W(u^{n+1},v^{n+1})+\frac{\sigma}{2}|(-\Delta)^{-\frac 12}(v^{n+1}-\overline v^{n+1})|^2d\bx$.
\end{thm}

Then we use the following algorithm to choose time steps adaptively. 

\medskip
\noindent{\bf Algorithm for adaptive time stepping:}\\
\rule[4pt]{14.3cm}{0.05em}\\
\textbf{Given}  Solutions at time steps $n$ and $n-1$; parameters $tol$ and $\rho$, and the preassigned minimum and maximin allowable time steps $\delta t_{min}$ and $\delta t_{max}$.
\begin{description}
\item[Step 1] Compute $u_1^{n+1}, v_1^{n+1}$  by a first-order Lagrange multiplier scheme with $\delta t_n$.
\item[Step 2] Compute $u_2^{n+1}, v_2^{n+1}$ by adaptive Crank-Nicolson scheme \eqref{ad:ch:cn1}-\eqref{ad:ch:cn5}  with $\delta t_n$.
\item[Step 3] Calculate $e_{n+1}=\max\{\frac{\|u^{n+1}_2-u^{n+1}_1\|}{\|u^{n+1}_2\|},\frac{\|v^{n+1}_2-v^{n+1}_1\|}{\|v^{n+1}_2\|}\}$.
\item[Step 4] \textbf{if} $e_{n+1}>tol$, \textbf{then}\\
Recalculate time step $\delta t_n \leftarrow \max\{\delta t_{min},\min\{A_{dp}(e_{n+1},\delta t_n),\delta t_{max}\}\}$.
\item[Step 5] \textbf{goto} Step 1
\item[Step 6] \textbf{else}\\
Update time step  $\delta t_{n+1}\leftarrow \max\{\delta t_{min},\min\{A_{dp}(e_{n+1},\delta t_n),\delta t_{max}\}\}$,
\item[Step 7] \textbf{endif}
\end{description}
\rule[12pt]{14.3cm}{0.05em}\\
where $A_{dp}(e,\tau)$ is a suitable function, e.g.,  \cite{gomez2011provably,shen2016maximum}:
\begin{equation}
A_{dp}(e,\tau)=\rho(\frac{tol}{e})^{\frac 12}\tau.
\end{equation}

%\subsubsection{Adaptive time stepping method}
%A  main advantage of unconditional energy stable schemes is to design adaptive time stepping method which can improve the efficiency of computation dramatically since we can choose the time step according to the accuracy instead of stability.
As an example, we used the above adaptive  scheme with $S_u=S_v=5$ to  solve the  coupled Cahn-Hilliard equations with the following  parameters
\begin{equation}
\eps_u=0.075 \quad \eps_v=0.075 \quad \sigma =10 \quad \alpha=-0.23 \quad \beta =0.5 \quad \gamma=0.
\end{equation}
In the first figure of Fig.\,\ref{time_adaptive}, we plot the energy curves with uniform time step sizes $\delta t=10^{-2}$, $\delta t=10^{-5}$  and  with adaptive time steps. It is observed   that the curve by  adaptive time stepping coincides with the reference curve by $\delta t=10^{-5}$, while the curve with larger time step $10^{-2}$ deviates from the reference curve. In the second figure,  we present the evolution of adaptive time steps  and observe that a wide range of time steps are used with the largest time step almost two-order of magnitude larger than the smallest time step. In the last figure, we plot evolution of  the Lagrange multiplier $\eta$, and  observe that $\eta$ oscillates around $1$ but remains to be positive.

\begin{figure}[htbp]
\centering
\includegraphics[width=0.45\textwidth,clip==]{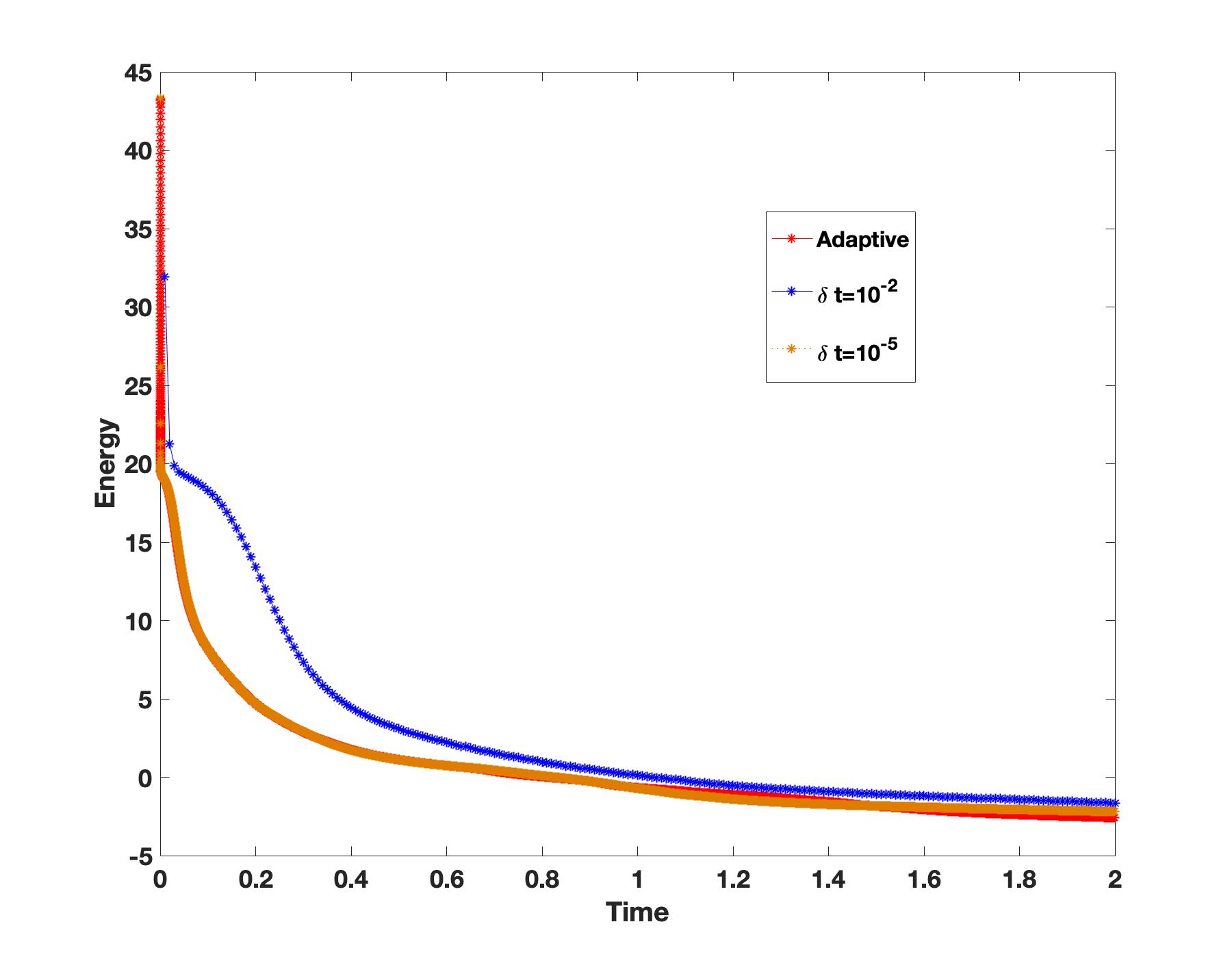}
\includegraphics[width=0.45\textwidth,clip==]{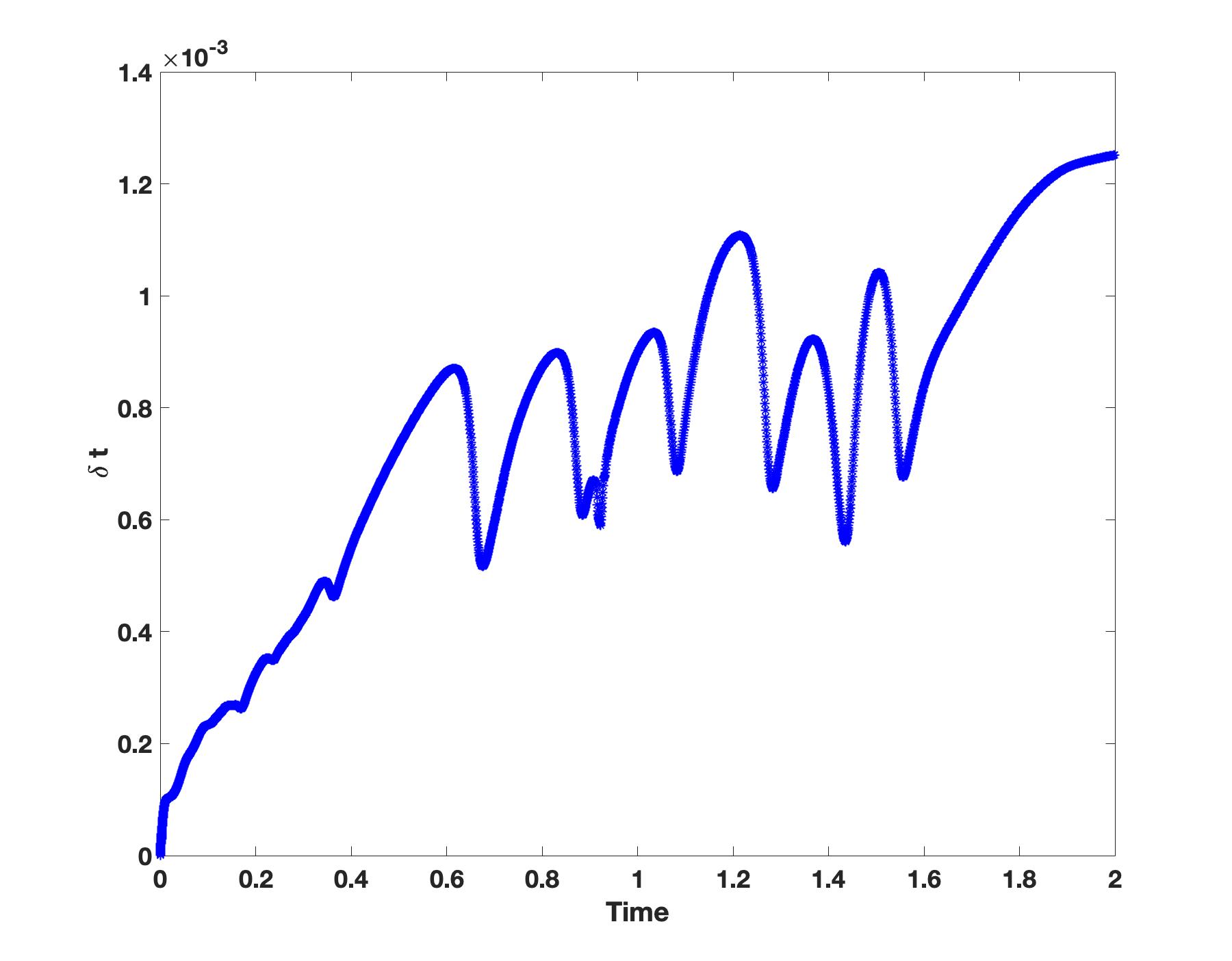}
\includegraphics[width=0.45\textwidth,clip==]{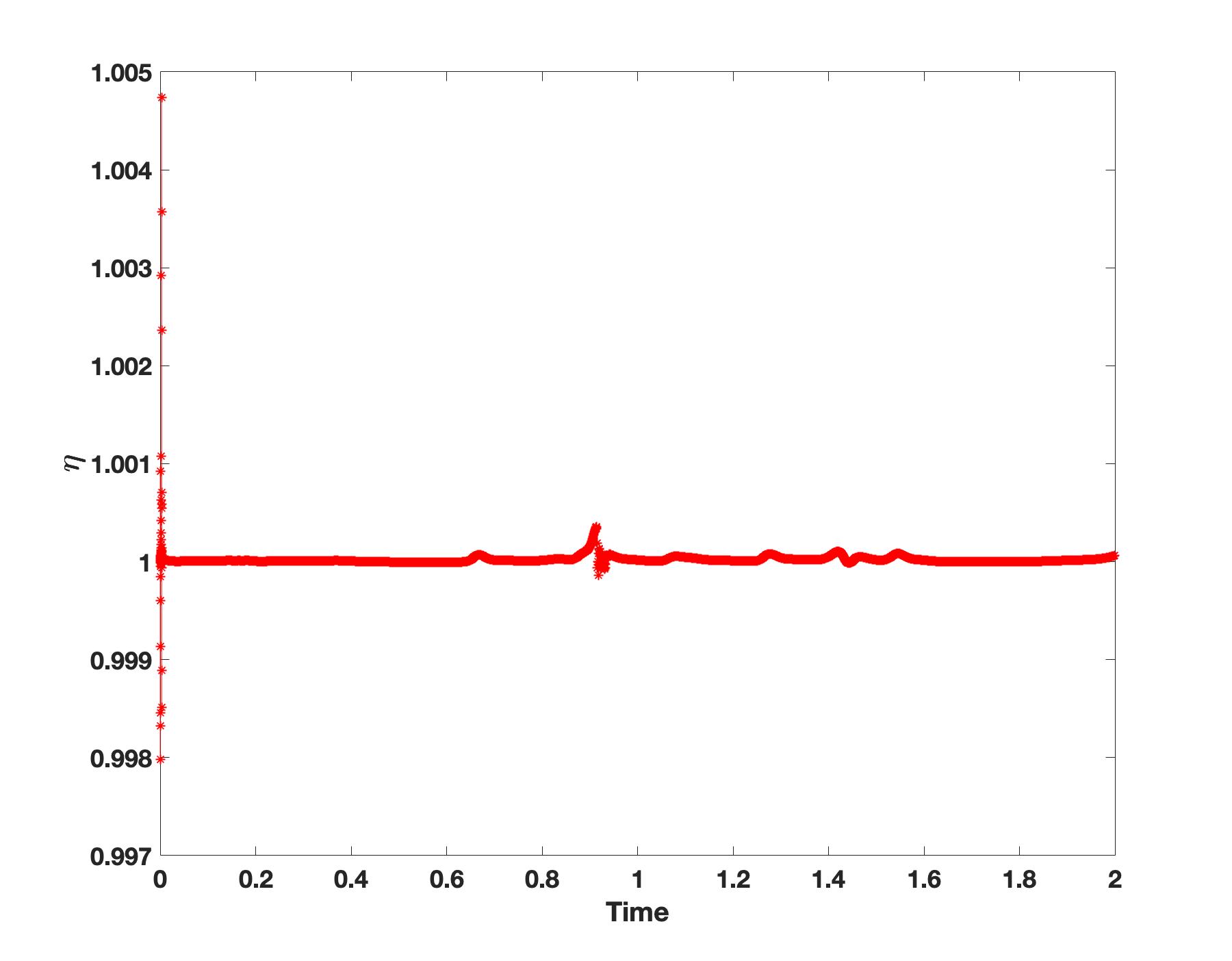}
\caption{First: evolution of the free energy with uniform time steps and   adaptive time stepping; Second: evolution of the time steps with  adaptive time stepping; Third: evolution of the Lagrange multiplier $\eta$ with  adaptive time stepping.}\label{time_adaptive}
\end{figure}

\begin{comment}
\begin{figure}
\centering
\subfigure[$\bu:t=0.0012$.]{\includegraphics[width=0.22\textwidth,clip==]{adaptive_u0.jpg}\hskip 0cm}
\subfigure[$\bv:t=0.0012$.]{\includegraphics[width=0.22\textwidth,clip==]{adaptive_v0.jpg}\hskip 0cm}
\subfigure[$\bu:t=0.1096$.]{\includegraphics[width=0.22\textwidth,clip==]{adaptive_u1.jpg}\hskip 0cm}
\subfigure[$\bv:t=0.1096$.]{\includegraphics[width=0.22\textwidth,clip==]{adaptive_v1.jpg}\hskip 0cm}
\subfigure[$\bu:t=0.5225$.]{\includegraphics[width=0.22\textwidth,clip==]{adaptive_u2.jpg}\hskip 0cm}
\subfigure[$\bv:t=0.5225$.]{\includegraphics[width=0.22\textwidth,clip==]{adaptive_v2.jpg}\hskip 0cm}
\subfigure[$\bu:t=1.3234$.]{\includegraphics[width=0.22\textwidth,clip==]{adaptive_u3.jpg}\hskip 0cm}
\subfigure[$\bv:t=1.3234$.]{\includegraphics[width=0.22\textwidth,clip==]{adaptive_v3.jpg}\hskip 0cm}
\caption{The 2D dynamical evolution of the phase variable $\bu,\bv$ for the Coupled-BCP model by using adaptive time stepping method.}\label{adaptive:spin}
\end{figure}
\end{comment}

\section{Numerical simulations of block copolymers}\label{BCP}
 In this section,  we consider 
   the coupled non-local Cahn-Hilliard system \eqref{ch:1}-\eqref{ch:4} introduced in \cite{avalos2016frustrated} to describe the dynamics of block copolymers. The unknowns are $u$ and $v$ 
where  $u$ describes macrophase separation into homopolyer and copolymer domains, while $v$ describes
the microphase separation within the copolymer domain;  parameters $M_u$ and $M_v$  are mobility constants that control the speed at which order parameters $u$ and $v$ change; $\eps_u$ represents the interfacial width while $\eps_v$ is a parameter  related to the temperature; and $\alpha,\beta,\gamma, \sigma$ are modeling parameters.  The phase parameter $\sigma$ represents a long range interaction that is needed to obtain particles with a fine structure, such as ellipsoids. 
 The coupling parameter $\alpha$ causes symmetry-breaking between microphase separated domains and by changing its value we are able to control the interaction between the confined copolymer and the confining surface. The coupling parameter $\beta$ affects the free energy depending of the  sign of $u$. The parameter $\gamma$ plays a similar role as $\beta$, which affects the free energy depending of the  sign of $v$. So we take  $\gamma=0$ in our simulations.

Three-dimensional simulations based on the coupled BCP model have produced very interesting  phenomena which are also observed in experiments  \cite{avalos2016frustrated,varadharajan2018surface}. For example, Fig. \ref{refer_2} presented in  \cite{varadharajan2018surface} shows that the model can correctly capture the  dynamical transformation of striped ellipsoids into onion like shapes as interfacial width decreases which corresponds to temperature increase. 
However, simple semi-implicit numerical schemes used in  \cite{avalos2016frustrated,varadharajan2018surface} for the above system lead to severe time step constraints in certain parameter regimes, so they are very expensive to run, particularly in three dimension. 

We present below several two-dimensional numerical simulations for the coupled non-local Cahn-Hilliard system \eqref{ch:1}-\eqref{ch:4} by using the stabilized scheme \eqref{ad:ch:cn1}-\eqref{ad:ch:cn5} with adaptive time stepping. 
 In all simulations, we take the uniformly distributed random function in $[-1,1]^2$ as the initial condition.

 \begin{figure}
\centering
\includegraphics[width=0.5\textwidth,clip==]{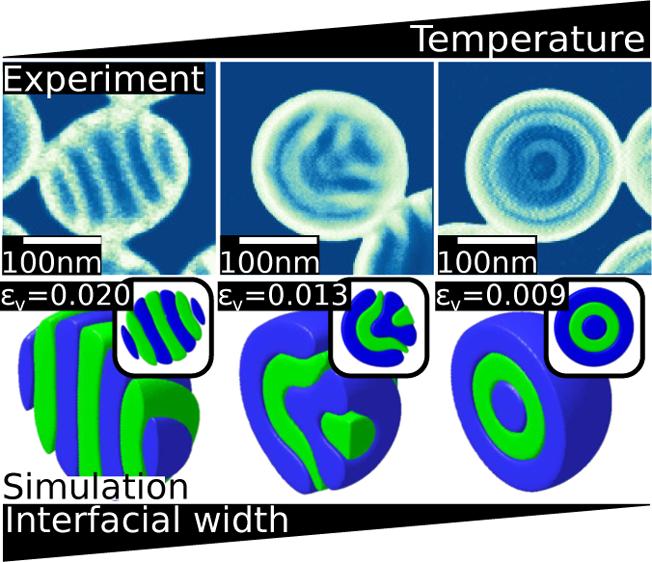}\hskip 0cm
\caption{Experimental results  and simulations at various temperature, taken from \cite{varadharajan2018surface}.}
\label{refer_2}
\end{figure}

\subsection{Annealing of Block Copolymer}
Annealing of block copolymers is an important process in reconfiguration of nanoparticles.  The coupled Cahn-Hilliard equations \eqref{ch:1}-\eqref{ch:4}   can correctly describe the process of annealing  block copolymers \cite{varadharajan2018surface,avalos2016frustrated},  and makes it possible to numerically simulate a variety of experimental conditions involving  nanoparticles undergoing a heating process.

In our numerical simulations, we assume that a block copolymer  particle is immersed in an external medium, which could be a homopolymer, a solvent, or water. To fix the idea,  we shall consider the external medium as the homopolymer, with  the order parameter    $u=-1$ representing the homopolymer rich domain and $u=1$ representing the BCP-rich domain, along with a smooth transitional layer of thickness $\eps_u$. The order parameter $v$  describes micro separation inside the BCP
domain which also acquires values from interval $[-1,1]$ with the end points corresponding to A-type BCP and B-type BCP.   
The heating process starts at a given temperature $T$ which in turn depends on the interface width $\eps_v$ as  $T \propto  \frac{1}{\eps_v^2}$.

 In the first simulation, we set the parameters in \eqref{ch:1}-\eqref{ch:4} to be
 \begin{equation}\label{second}
\eps_u=0.075 \quad \eps_v=0.05 \quad \sigma =10 \quad \alpha=0.1 \quad \beta =-0.75 \quad \gamma=0.
\end{equation}
%The computed morphological evolution is shown in Fig.\,\ref{spin_2}. 
Fig.\,\ref{spin_2} shows  the morphological transformations at different times. During the phase transformation, the order parameters $u$ and $v$ vary  in a complicated manner. However at $t=2$ which  essentially reached the steady state, it is observed that  $v$ exhibits    locally striped shapes, with the yellow bulk presenting  A-BCP particles and blue bulk representing  the B-BCP particles. The steady state morphology resembles  the striped ellipsoid shape in the experimental results depicted  in Fig.\,\ref{refer_1}.

 \begin{figure}
\centering
\subfigure[$u:t=0.02$.]{\includegraphics[width=0.22\textwidth,clip==]{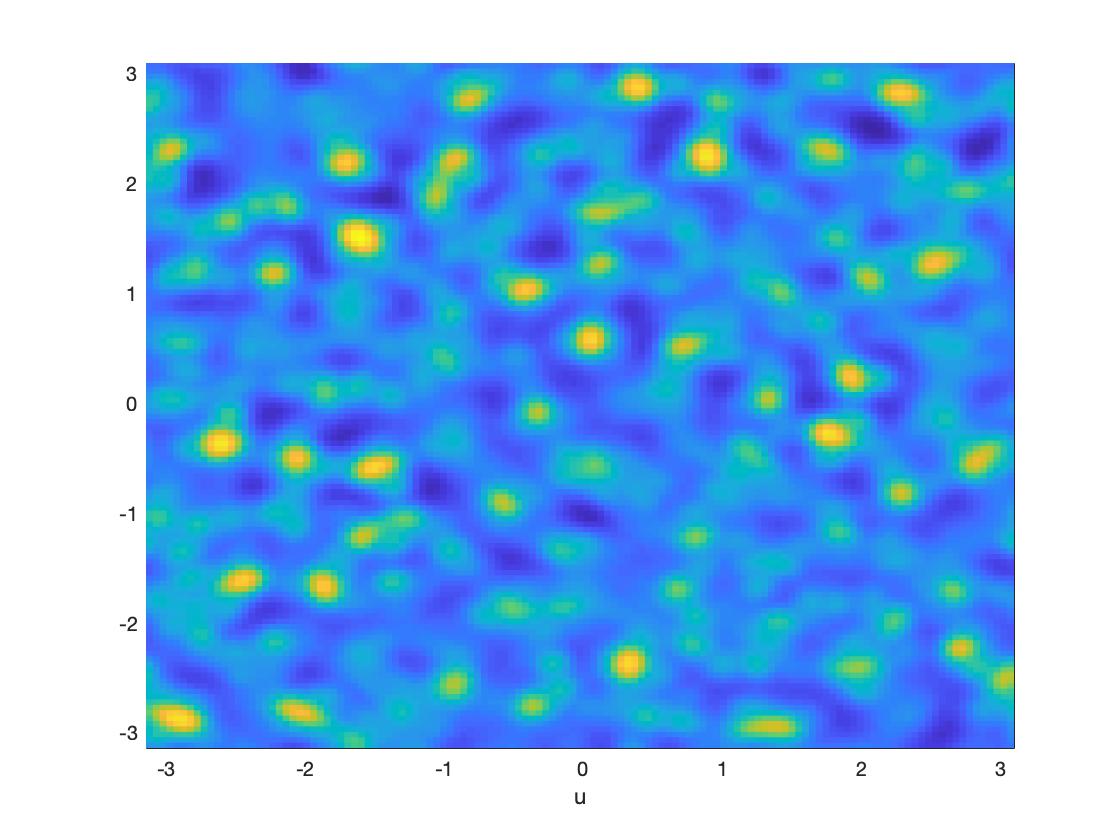}\hskip 0cm}
\subfigure[$v:t=0.02$.]{\includegraphics[width=0.22\textwidth,clip==]{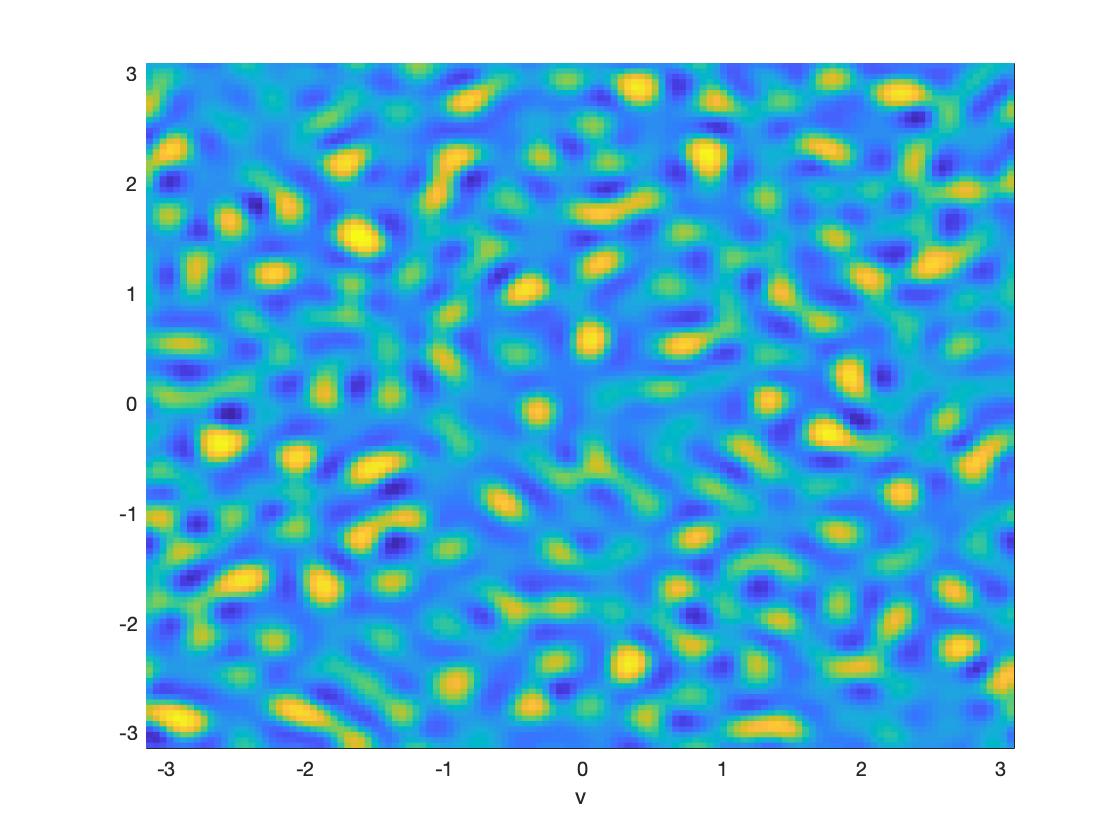}\hskip 0cm}
\subfigure[$u:t=0.1$.]{\includegraphics[width=0.22\textwidth,clip==]{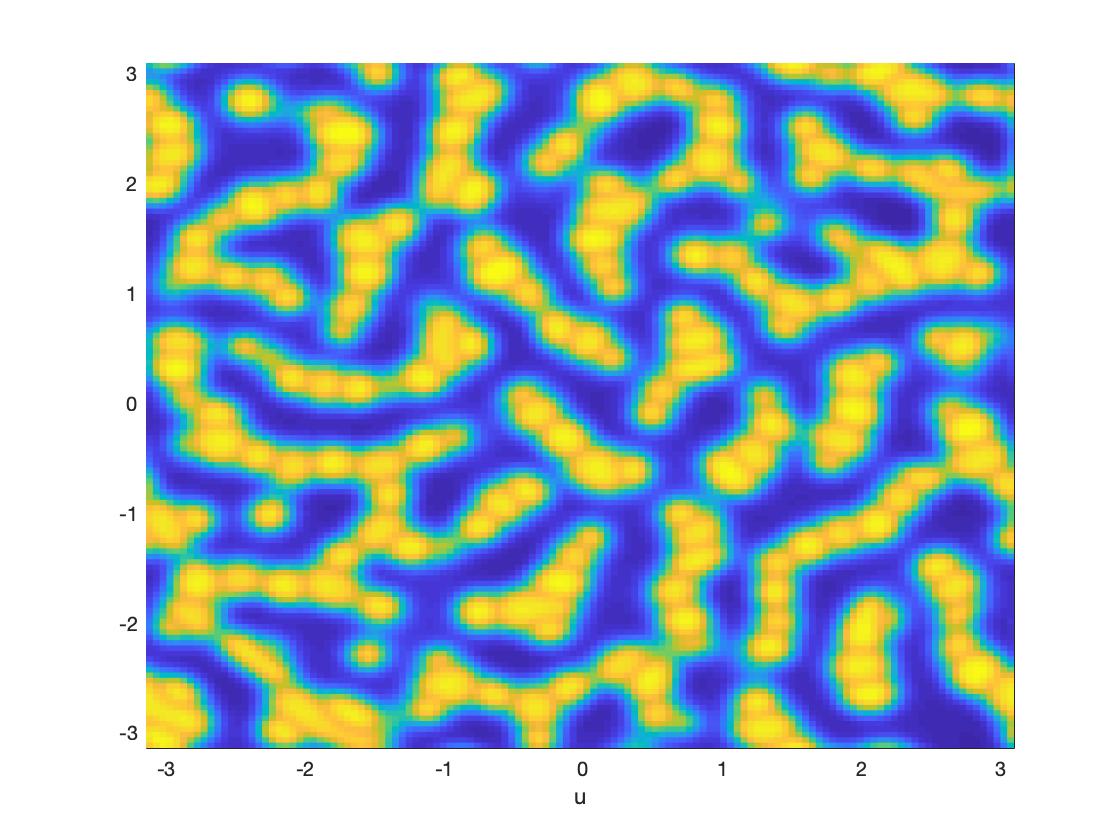}\hskip 0cm}
\subfigure[$v:t=0.1$.]{\includegraphics[width=0.22\textwidth,clip==]{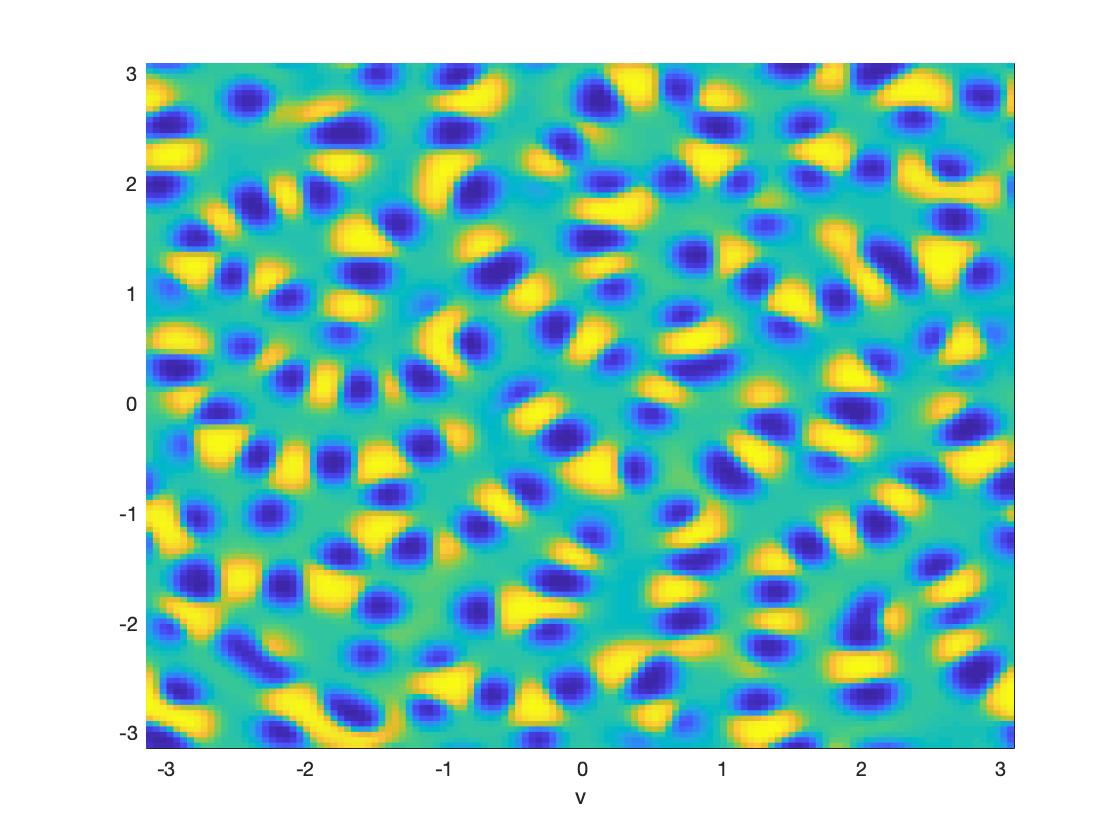}\hskip 0cm}
\subfigure[$u:t=0.5$.]{\includegraphics[width=0.22\textwidth,clip==]{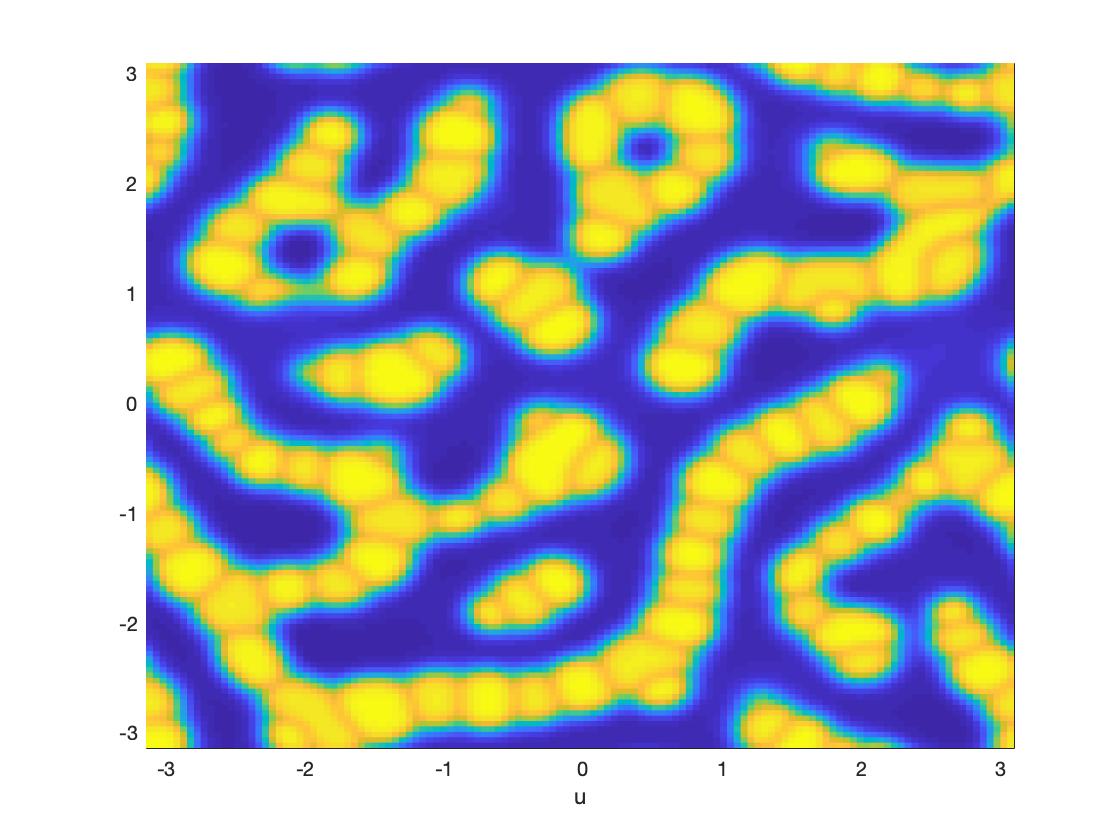}\hskip 0cm}
\subfigure[$v:t=0.5$.]{\includegraphics[width=0.22\textwidth,clip==]{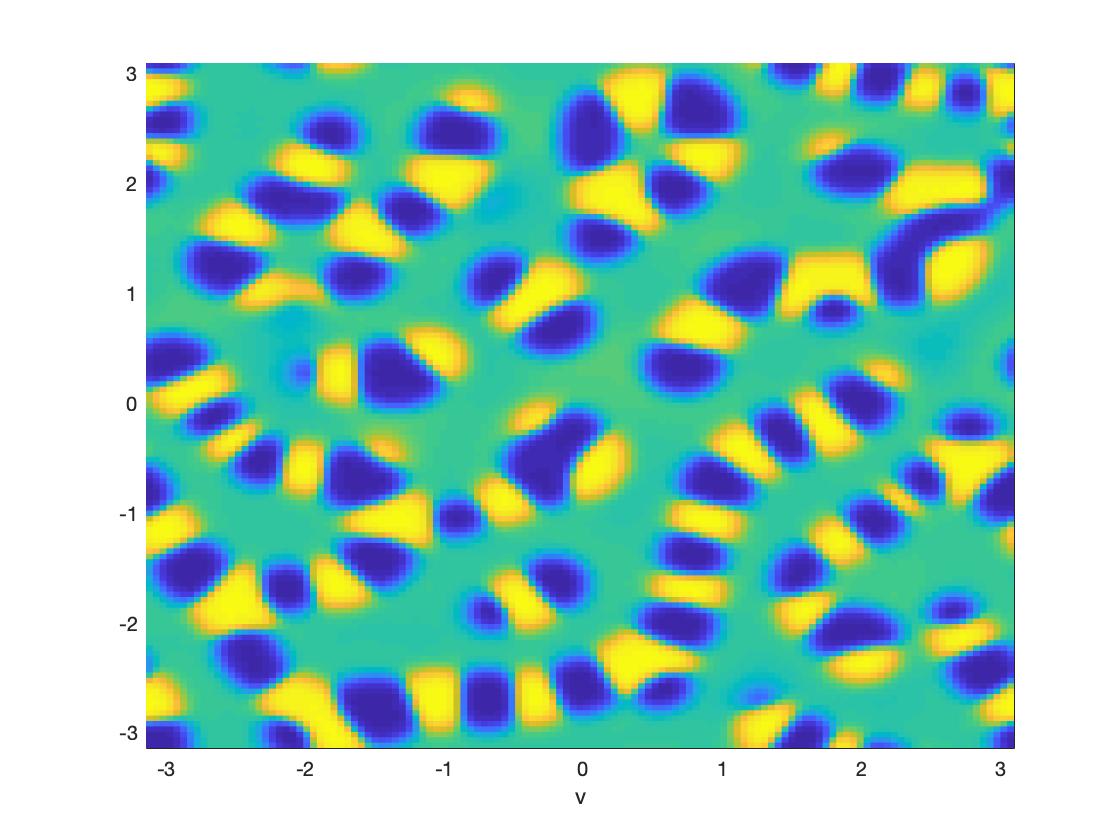}\hskip 0cm}
\subfigure[$u:t=2$.]{\includegraphics[width=0.22\textwidth,clip==]{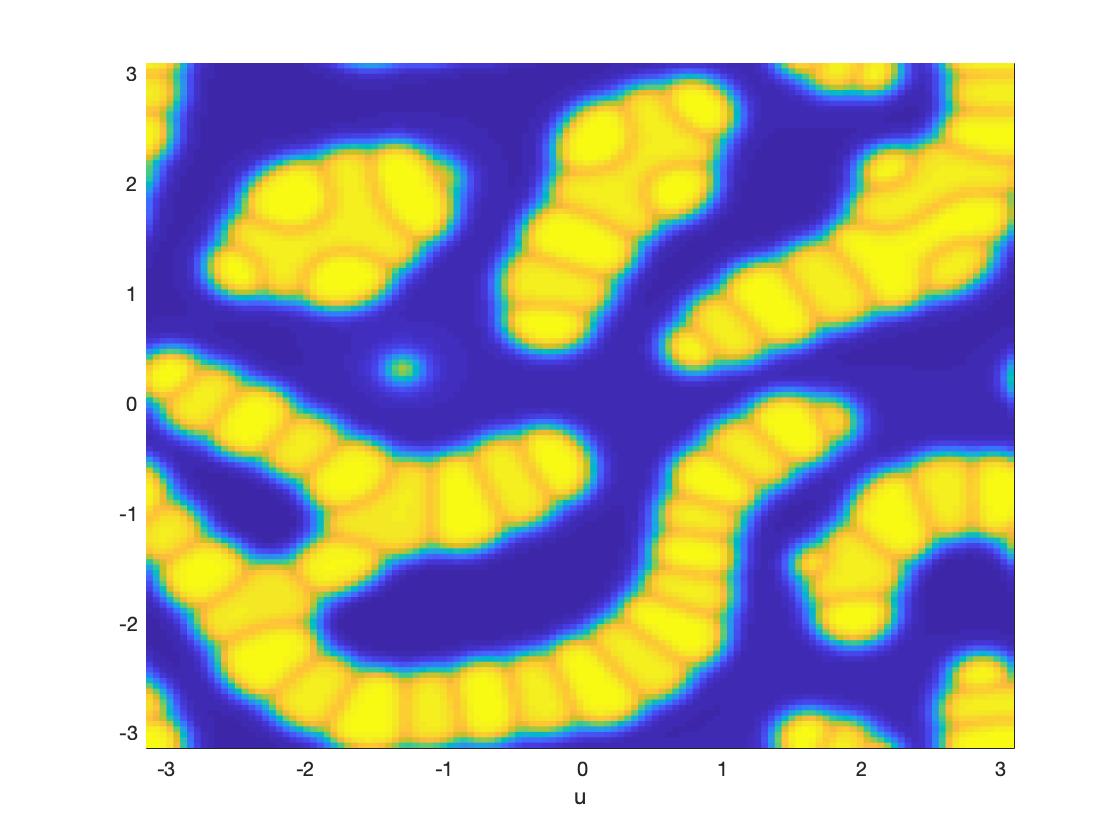}\hskip 0cm}
\subfigure[$v:t=2$.]{\includegraphics[width=0.22\textwidth,clip==]{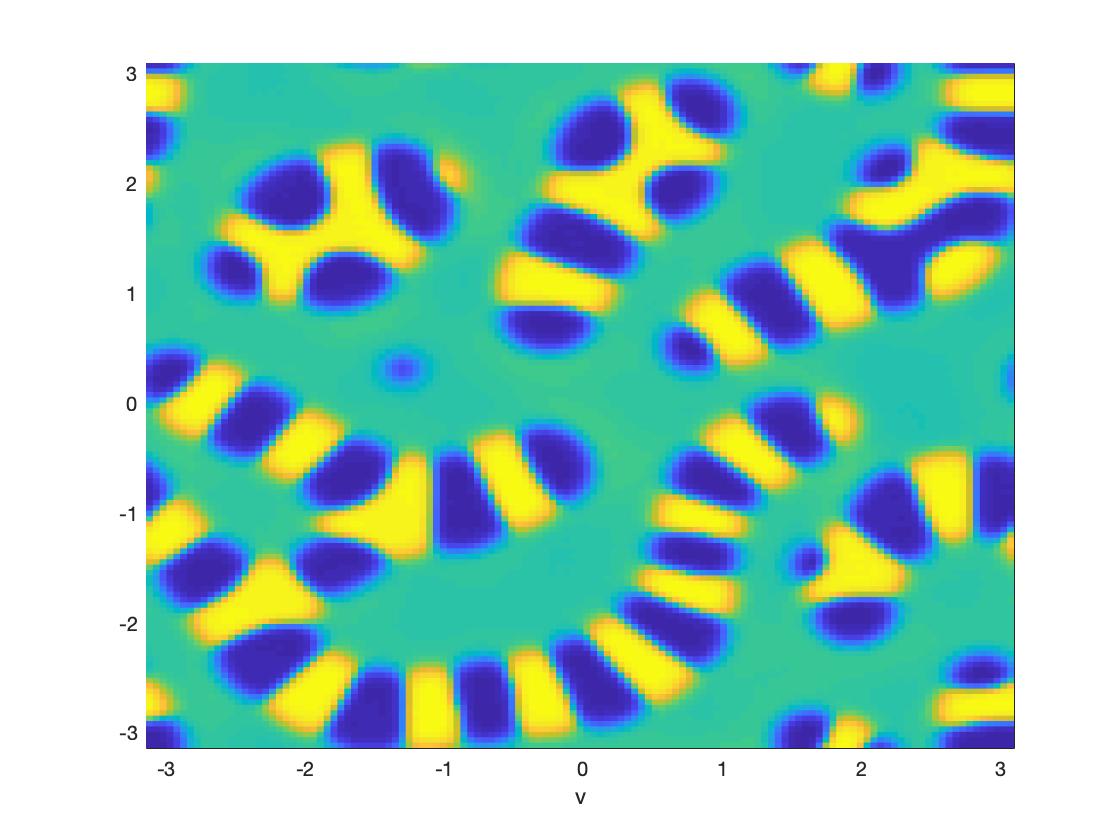}\hskip 0cm}
\caption{Dynamical evolution of the phase variables $u,v$ for the Coupled-BCP model with  parameters in \eqref{second}.}\label{spin_2}
\end{figure}

 \begin{figure}
\centering
\includegraphics[width=0.5\textwidth,clip==]{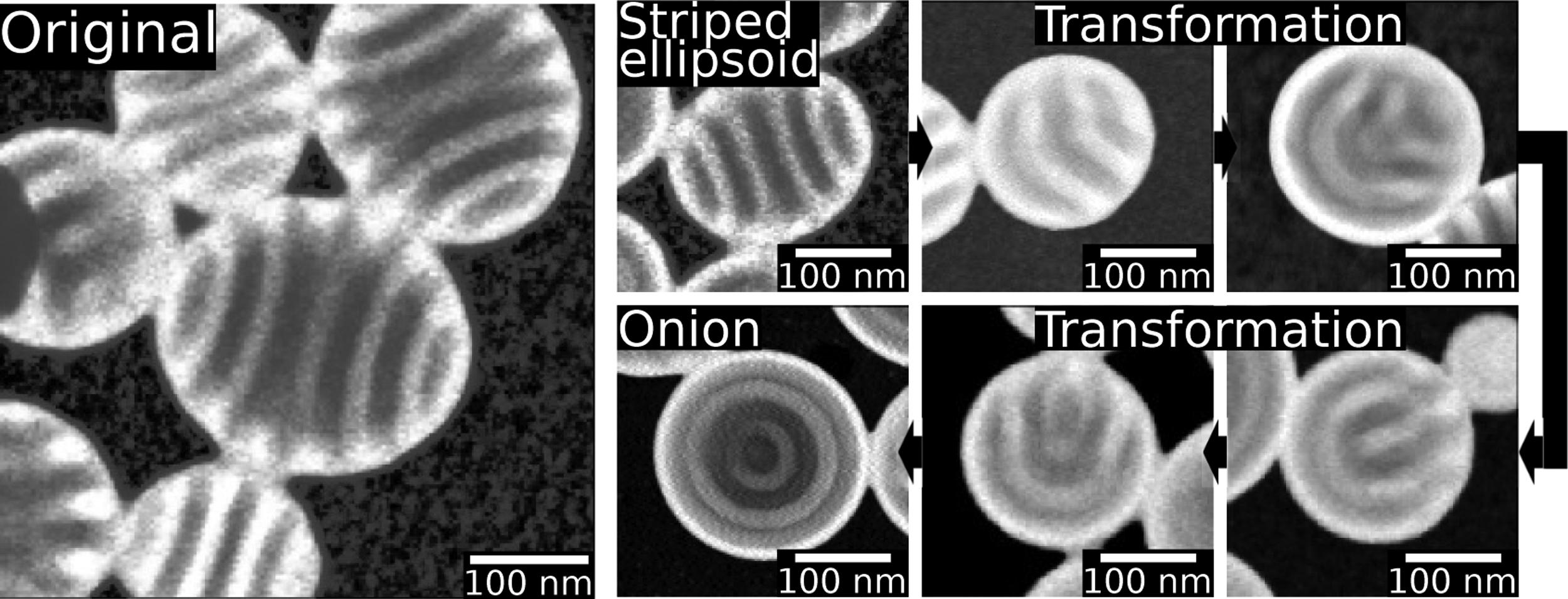}\hskip 0cm\\
\includegraphics[width=0.5\textwidth,clip==]{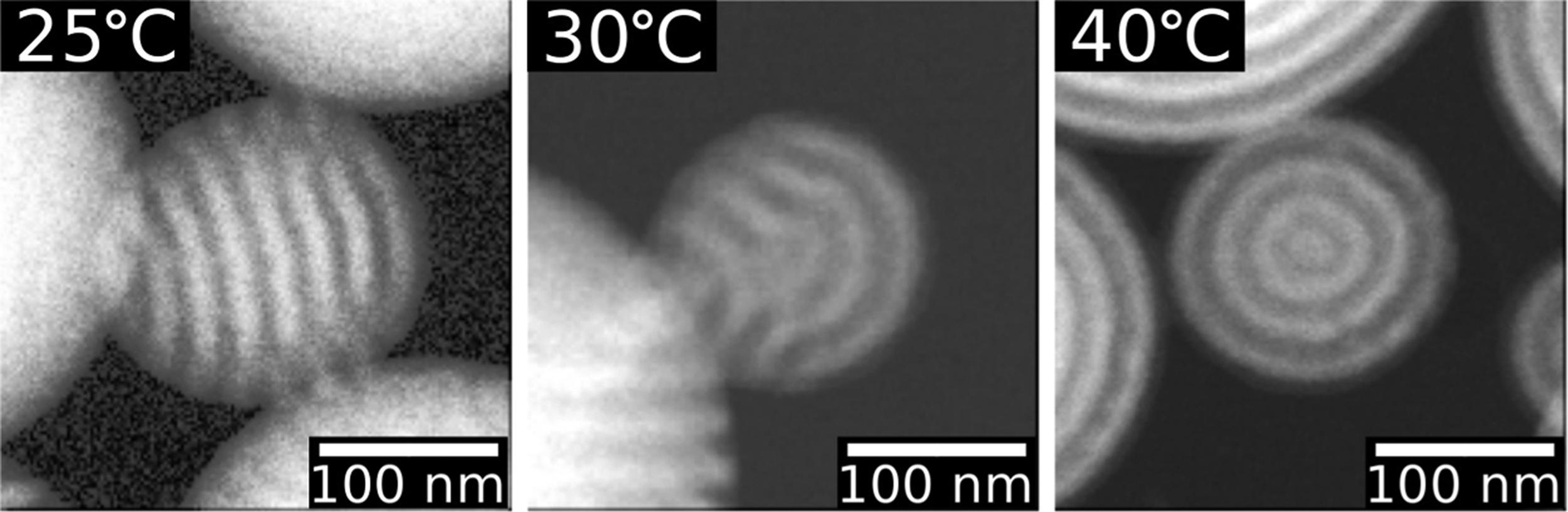}\hskip 0cm
\caption{Experimental results of annealing process at various temperature presented in \cite{varadharajan2018surface}.}
\label{refer_1}
\end{figure}

 \begin{comment}

  It is observed from Fig.\,\ref{time_adaptive_spin_1} the free energy will decay with time growth since the numerical scheme we proposed is unconditional energy stable. We also depict the evolution of time steps, a larger time step like $0.01$ is permitted for long time computation in Fig.\,\ref{time_adaptive_spin_1}.  We also observe the Lagrange Multiplier $\eta$ will stay around $1$ to keep energy decay.
  \begin{figure}[htbp]
\centering
%\includegraphics[width=0.45\textwidth,clip==]{adaptive_energy.jpg}
\includegraphics[width=0.45\textwidth,clip==]{spin_1_energy_long.jpg}
\includegraphics[width=0.45\textwidth,clip==]{spin_1_dt_long.jpg}
\includegraphics[width=0.45\textwidth,clip==]{spin_1_eta_long.jpg}
\caption{Evolution of the free energy and time step  by using adaptive time stepping method and  adaptive parameters are $tol=10^{-3}$, $\delta t_{min}=10^{-6}$, $\delta t_{max}=10^{-2}$ and $p=0.6$. }\label{time_adaptive_spin_1}
\end{figure}
\end{comment}

In the second simulation presented in  Fig.\,\ref{spin_1}, we used  the following set of parameters:
\begin{equation}\label{first}
\eps_u=0.075 \quad \eps_v=0.05 \quad \sigma =10 \quad \alpha=0.1 \quad \beta =0.75 \quad \gamma=0.
\end{equation}
 Note  that we switched the sign of $\beta$ while keeping other parameters unchanged.  We observe similar configurations as in Fig.\,\ref{spin_2}, except that the positions of yellow bulk and blue bulk  in the profiles for $v$ appear to be  interchanged.

\begin{figure}
\centering
\subfigure[$u:t=0.02$.]{\includegraphics[width=0.22\textwidth,clip==]{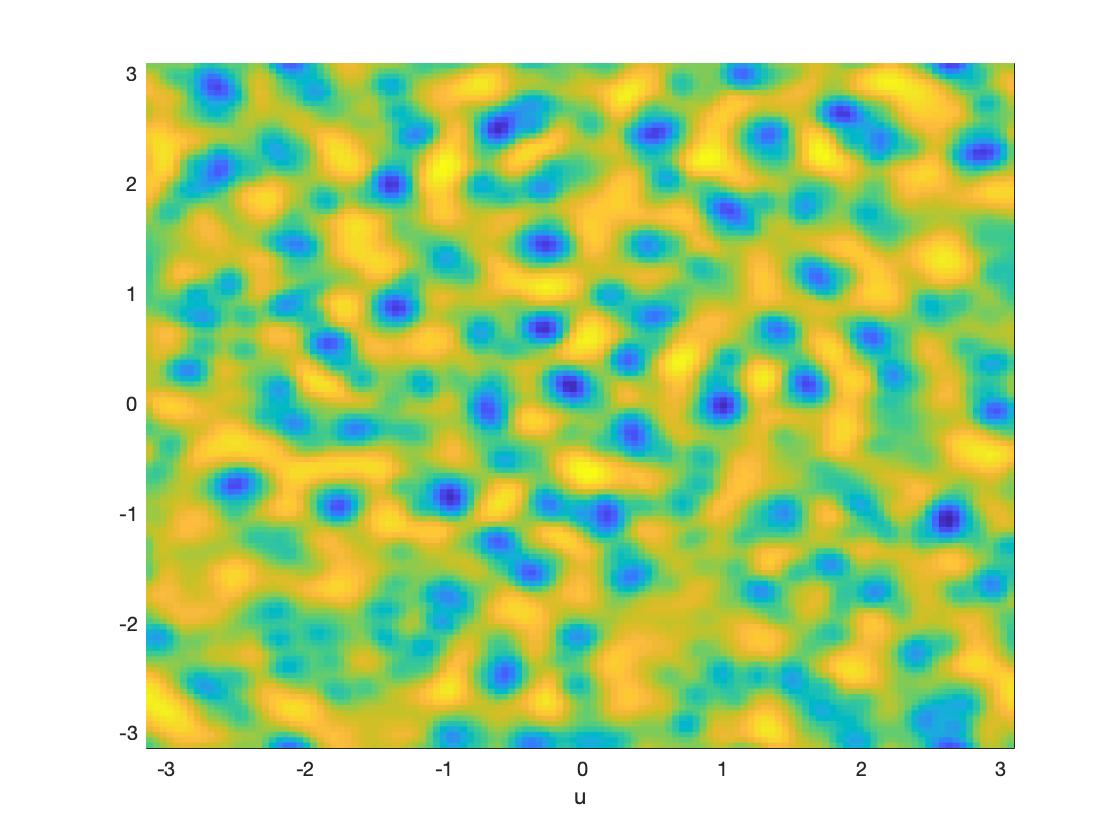}\hskip 0cm}
\subfigure[$v:t=0.02$.]{\includegraphics[width=0.22\textwidth,clip==]{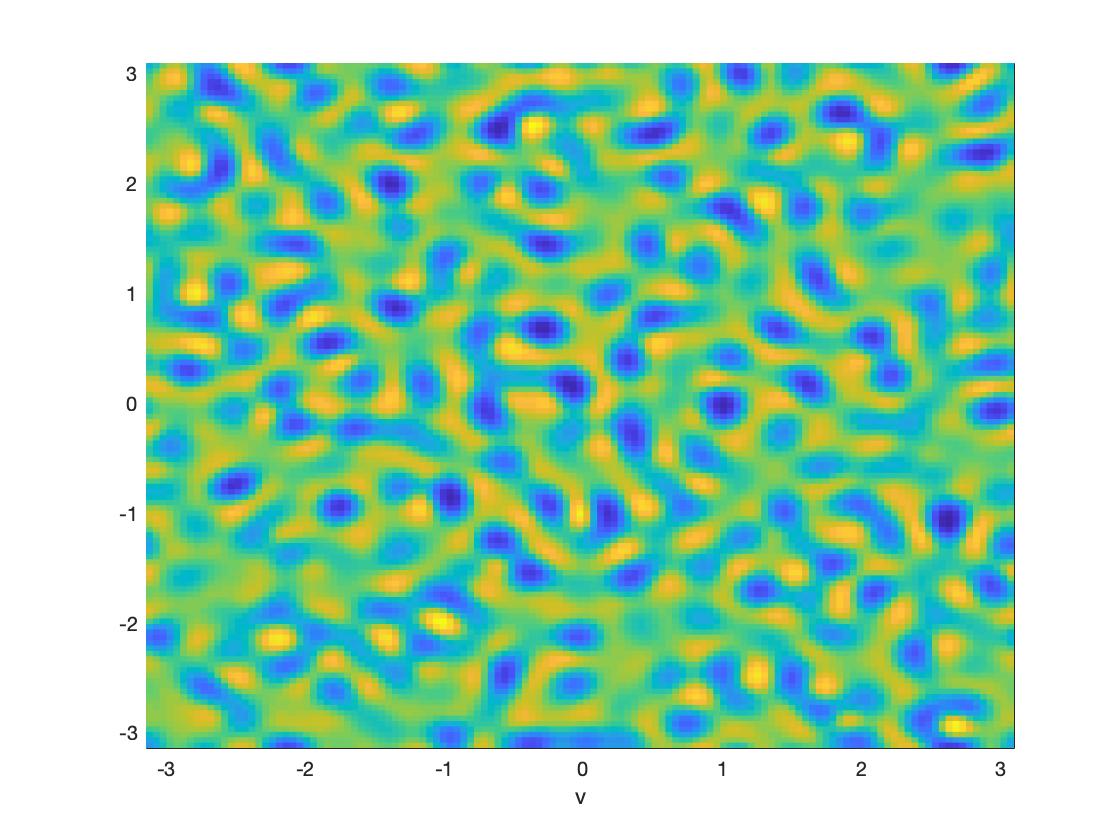}\hskip 0cm}
\subfigure[$u:t=0.1$.]{\includegraphics[width=0.22\textwidth,clip==]{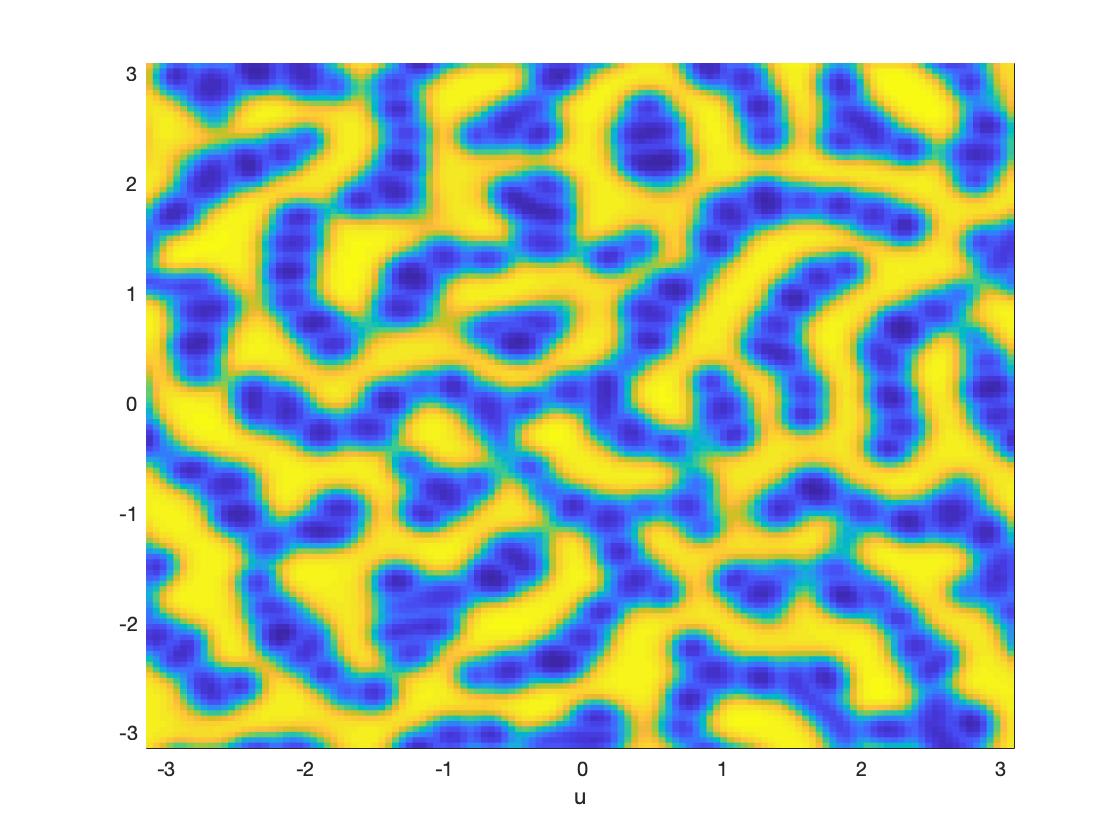}\hskip 0cm}
\subfigure[$v:t=0.1$.]{\includegraphics[width=0.22\textwidth,clip==]{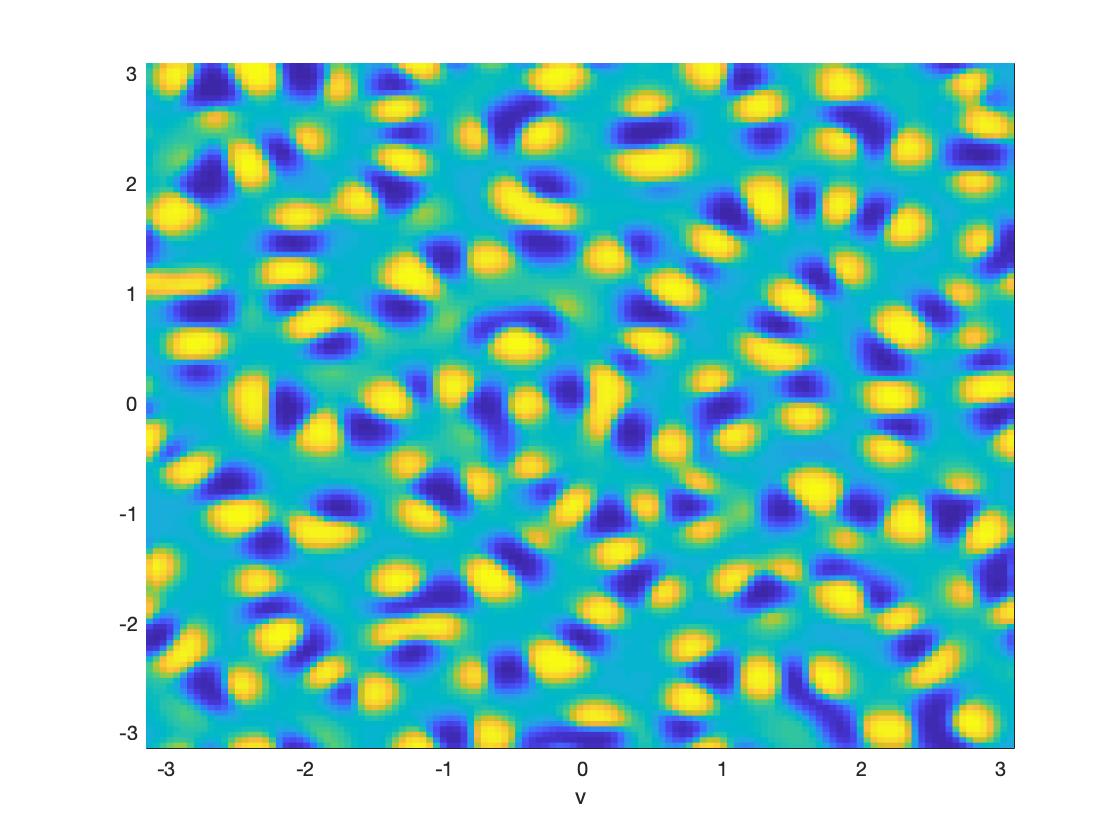}\hskip 0cm}
\subfigure[$u:t=0.5$.]{\includegraphics[width=0.22\textwidth,clip==]{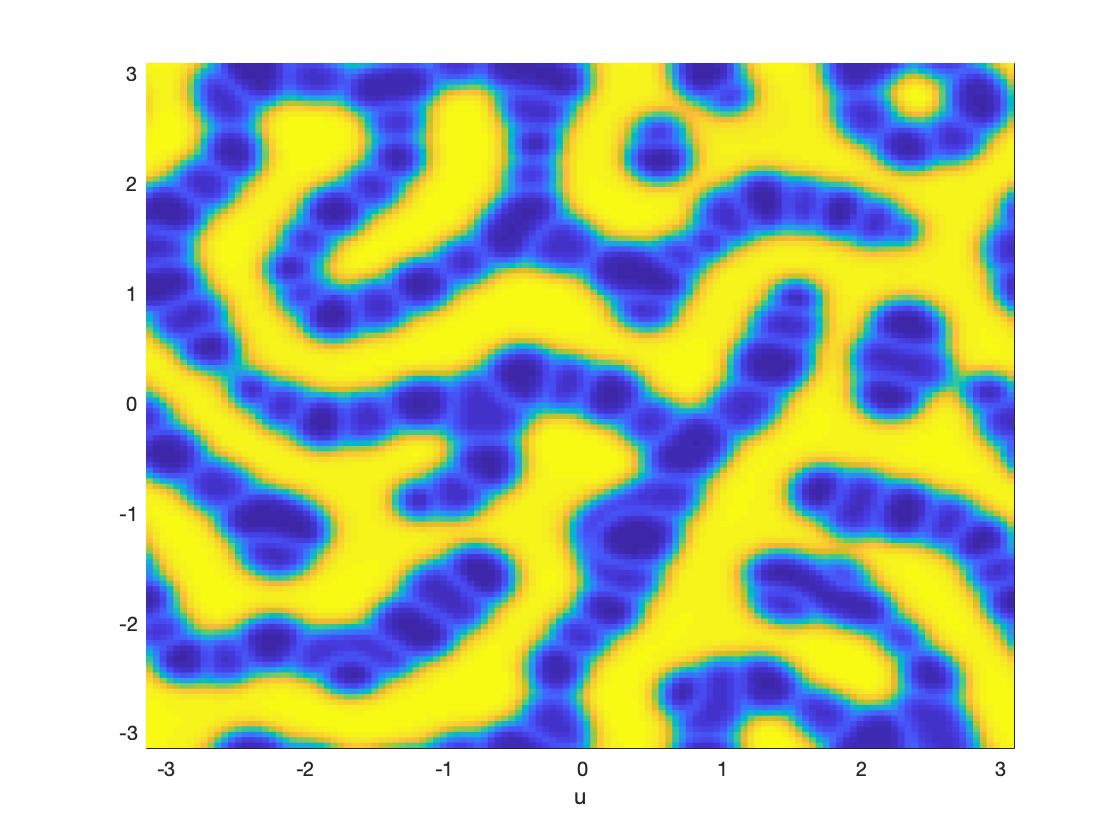}\hskip 0cm}
\subfigure[$v:t=0.5$.]{\includegraphics[width=0.22\textwidth,clip==]{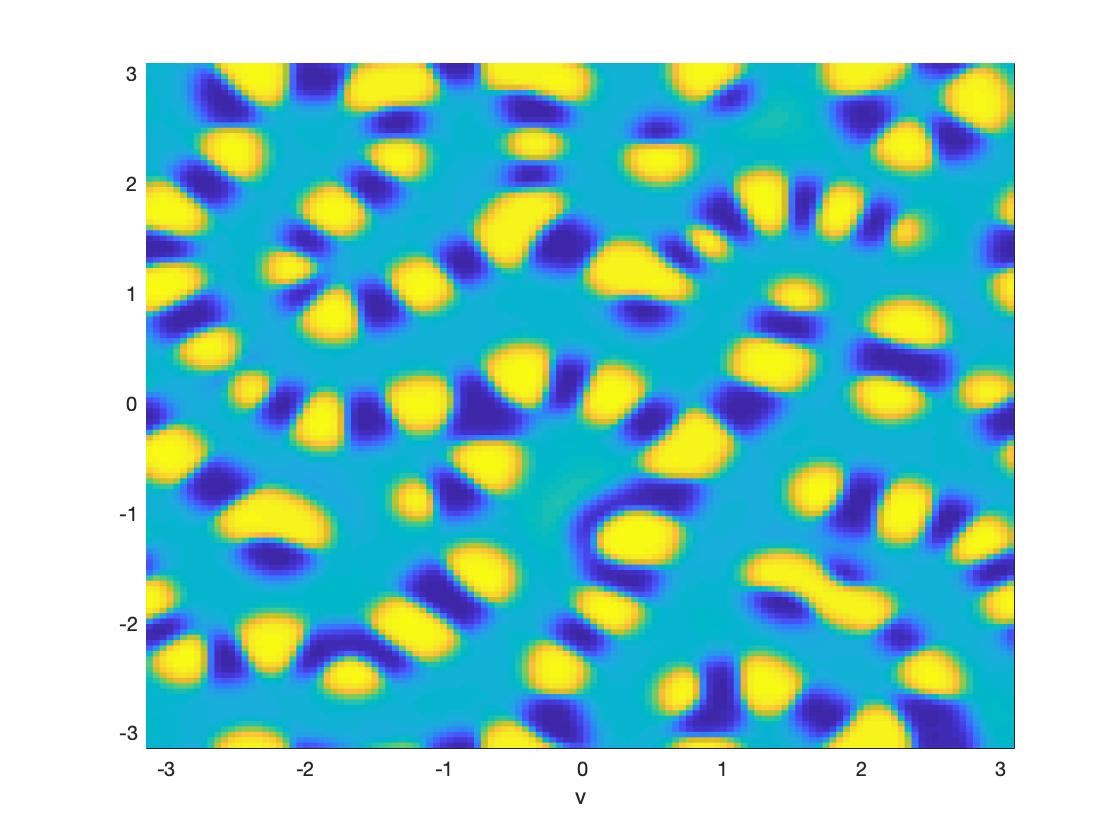}\hskip 0cm}
\subfigure[$u:t=2$.]{\includegraphics[width=0.22\textwidth,clip==]{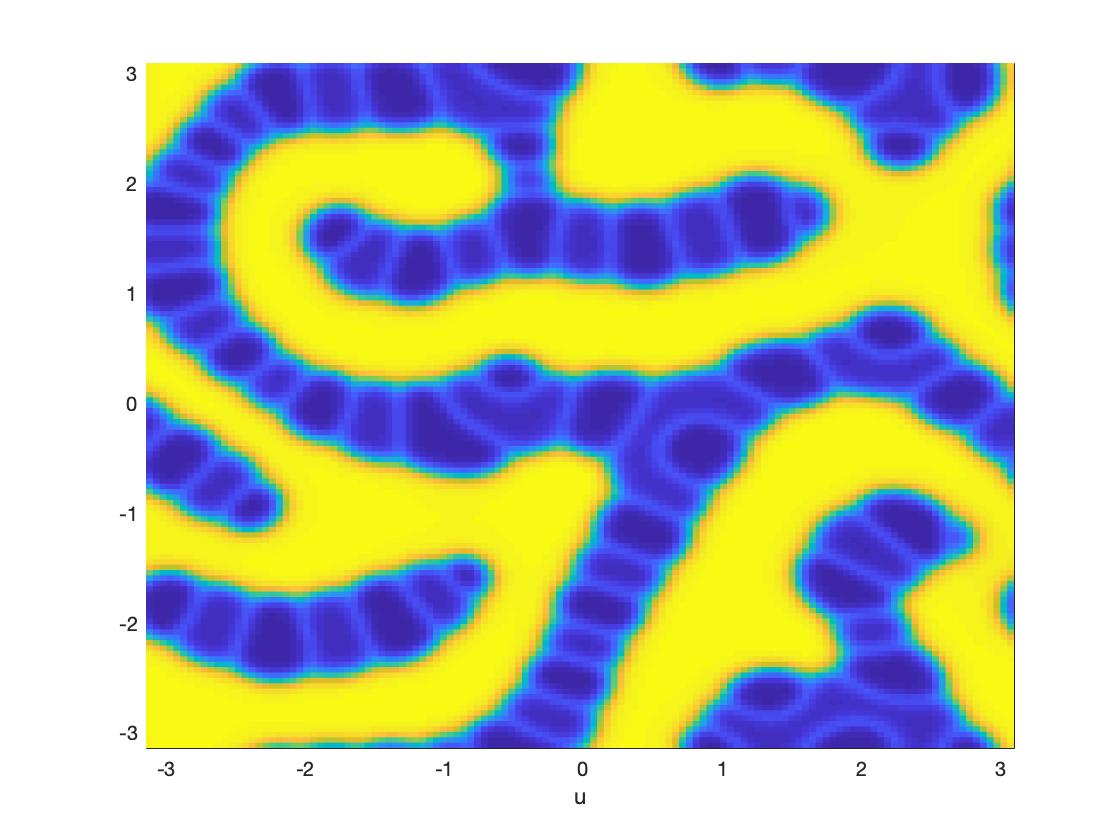}\hskip 0cm}
\subfigure[$v:t=2$.]{\includegraphics[width=0.22\textwidth,clip==]{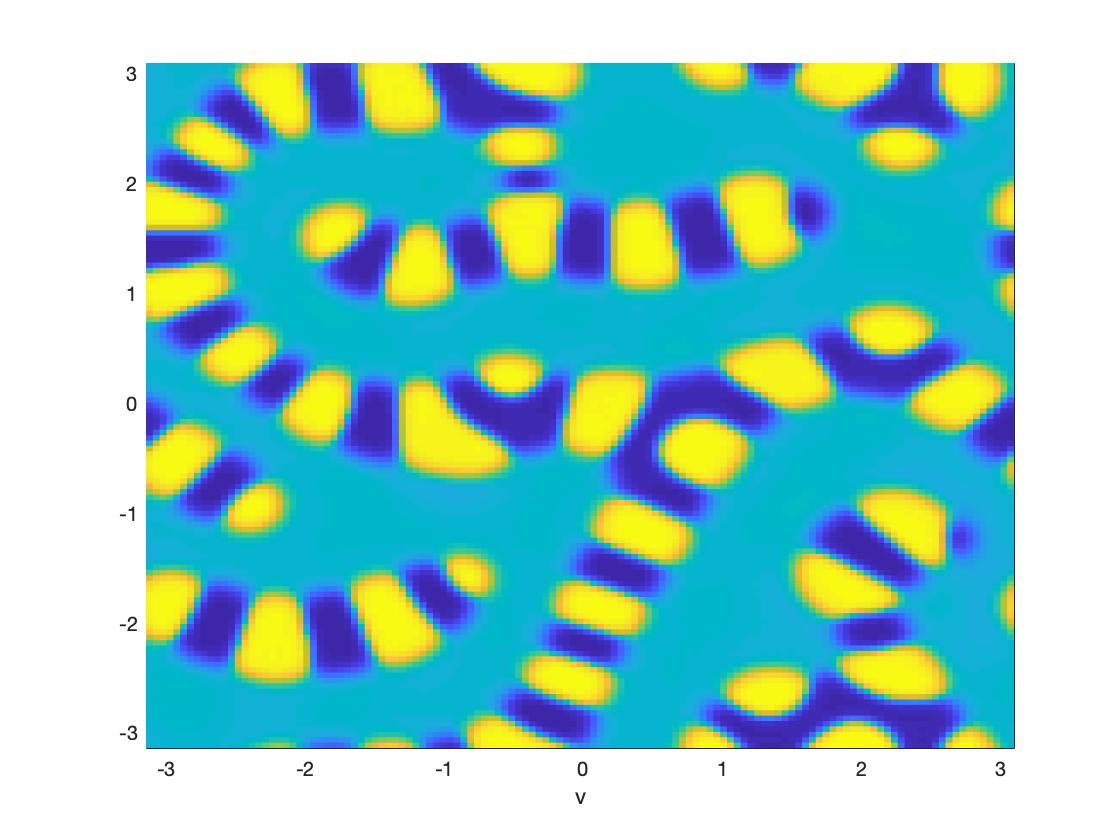}\hskip 0cm}
\caption{Dynamical evolution of the phase variables $u,v$ for the coupled-BCP model with  parameters in \eqref{first}.}\label{spin_1}
\end{figure}

\subsection{Annealing process at higher temperature}
To offer some insight into how the interaction between copolymer segments changes with temperature. We increased the temperature $T$, by decreasing the interfacial width  $\eps_v$, in the third simulation with  the parameters 
\begin{equation}\label{third}
\eps_u=0.03 \quad \eps_v=0.01 \quad \sigma =10 \quad \alpha=0.33 \quad \beta =0.5 \quad \gamma=0.
\end{equation}
 Fig.\,\ref{spin_3} shows various stages of morphology transformation at different times.  We observe that  the coarsening  is much more pronounced than in the first two simulations. Moreover, in addition to  striped ellipsoids, we also observe the appearance of transformation shapes, similar to   the experimental results of annealing in water  depicted in Fig.\,\ref{refer_1}.   %Fig.\,\ref{time_adaptive_spin_3} shows the energy decreasing with time evolution. Large time steps will gradually be chosen in areas with small energy changes. The Lagrange Multiplier $\eta$ will also oscillate around $1$ with  larger magnitude in areas which posse bigger gradient of  free energy. 

\begin{figure}
\centering
\subfigure[$u:t=0.01$.]{\includegraphics[width=0.22\textwidth,clip==]{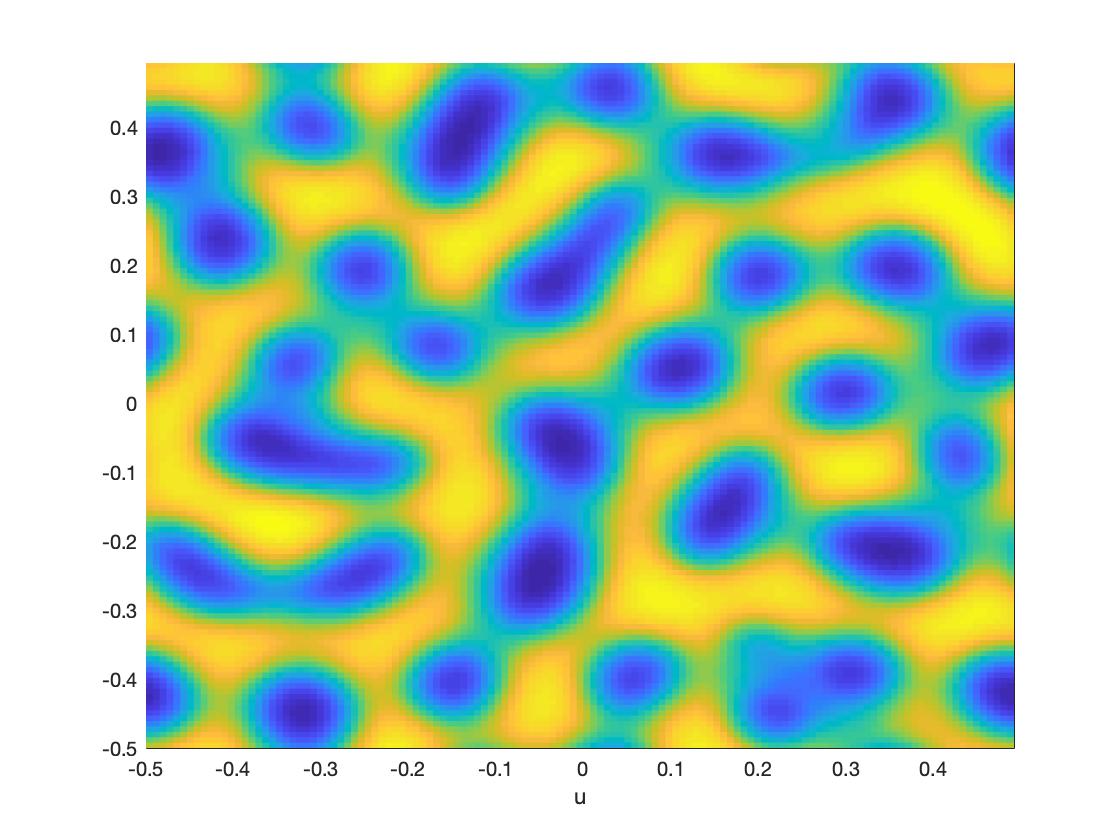}\hskip 0cm}
\subfigure[$v:t=0.01$.]{\includegraphics[width=0.22\textwidth,clip==]{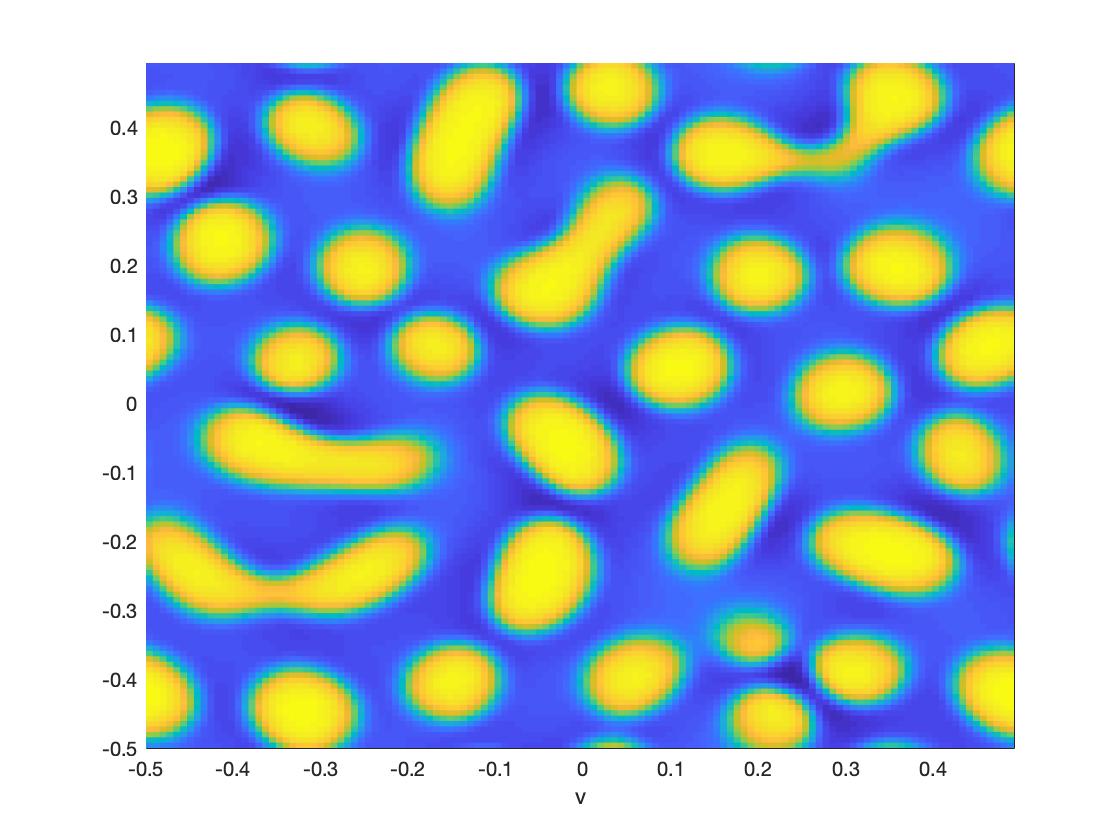}\hskip 0cm}
\subfigure[$u:t=0.1$.]{\includegraphics[width=0.22\textwidth,clip==]{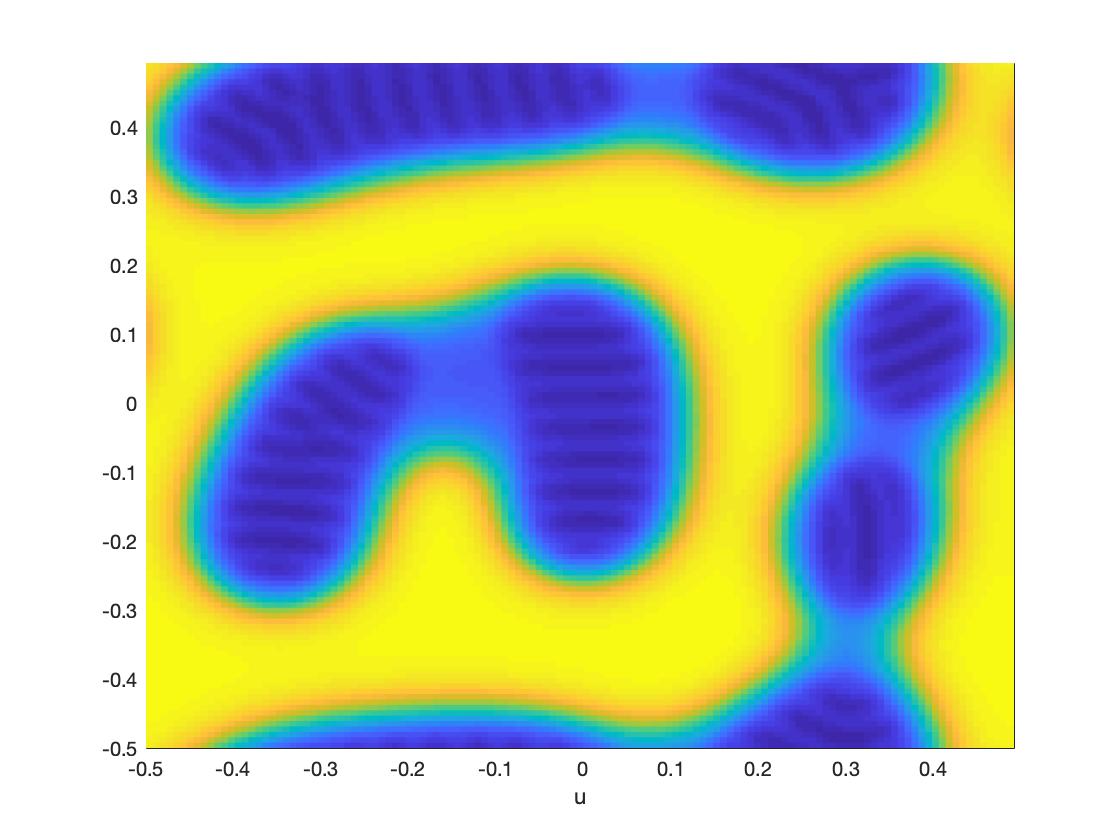}\hskip 0cm}
\subfigure[$v:t=0.1$.]{\includegraphics[width=0.22\textwidth,clip==]{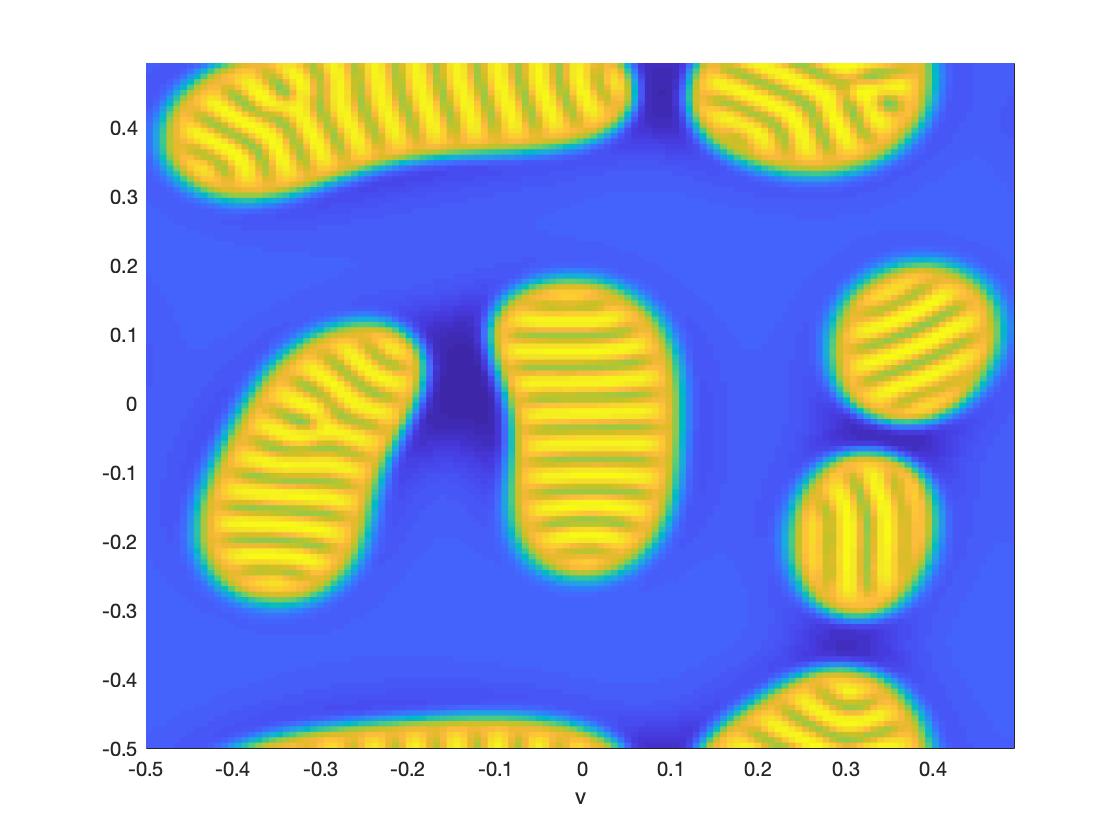}\hskip 0cm}
\subfigure[$u:t=0.25$.]{\includegraphics[width=0.22\textwidth,clip==]{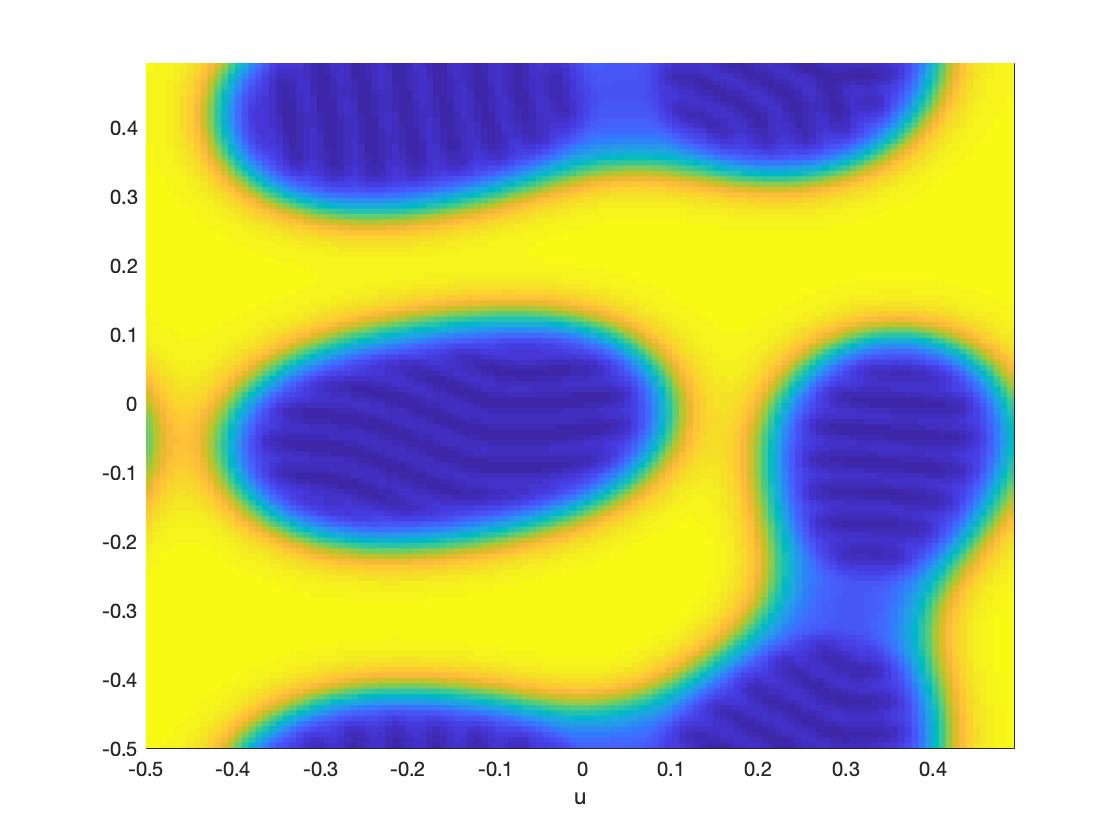}\hskip 0cm}
\subfigure[$v:t=0.25$.]{\includegraphics[width=0.22\textwidth,clip==]{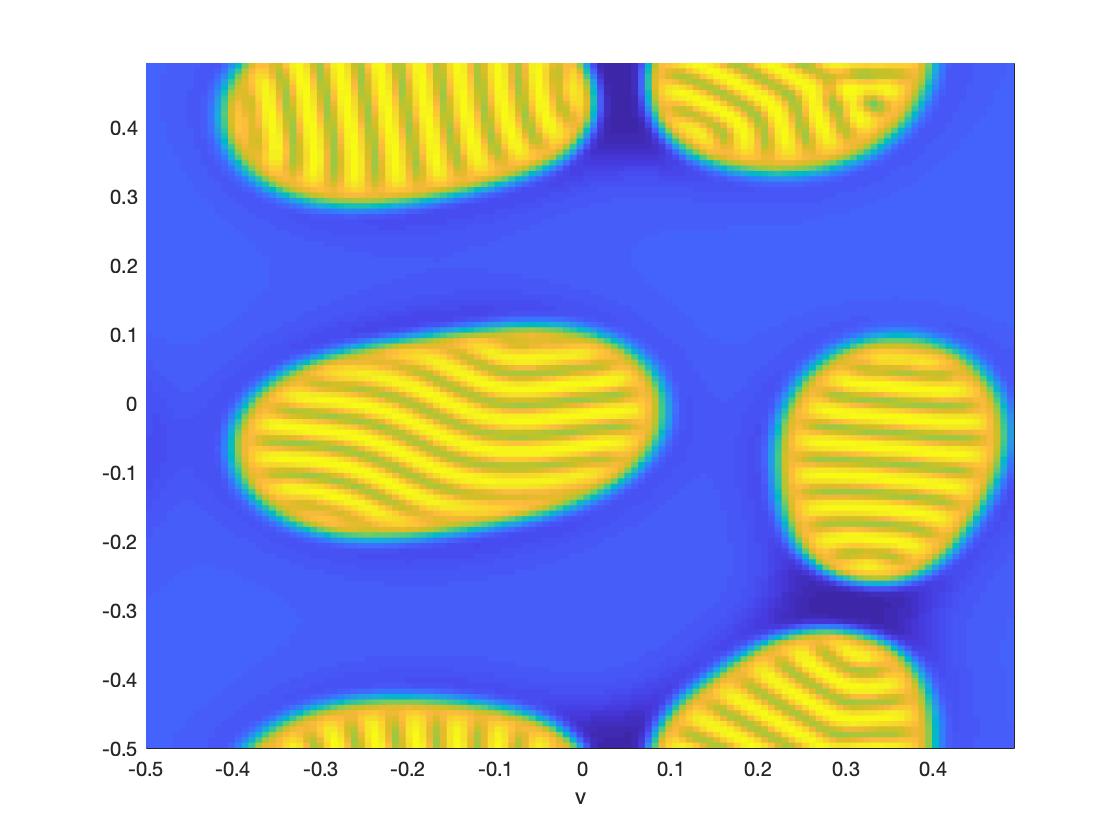}\hskip 0cm}
\subfigure[$u:t=3$.]{\includegraphics[width=0.22\textwidth,clip==]{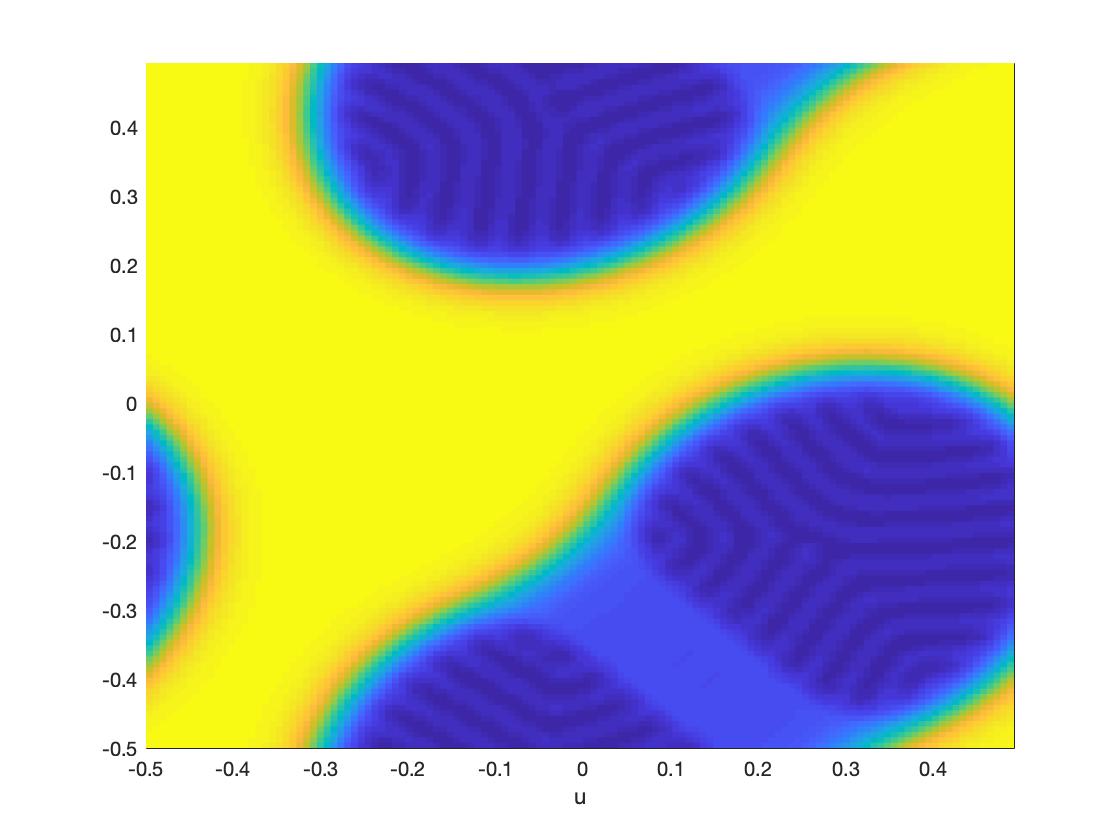}\hskip 0cm}
\subfigure[$v:t=3$.]{\includegraphics[width=0.22\textwidth,clip==]{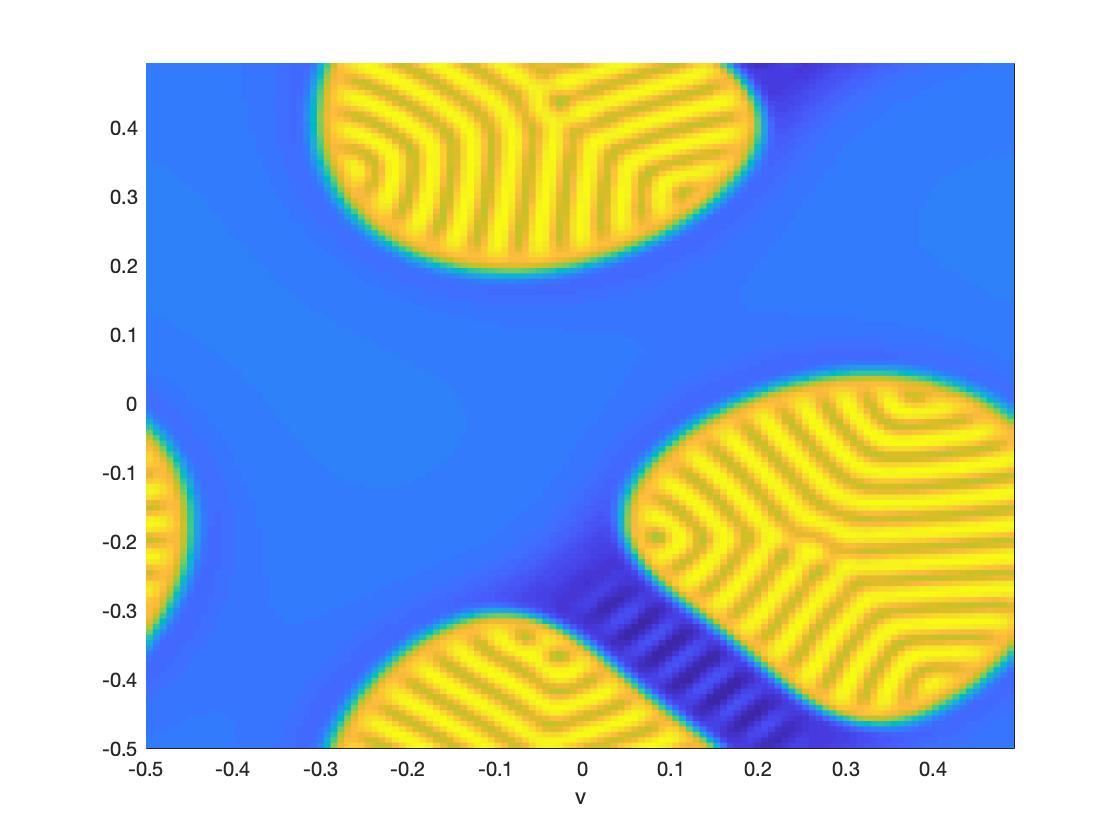}\hskip 0cm}
\caption{The 2D dynamical evolution of the phase variable $u,v$ for the Coupled-BCP model with parameters in \eqref{third}.}\label{spin_3}
\end{figure}

\begin{comment}
\begin{figure}[htbp]
\centering
%\includegraphics[width=0.45\textwidth,clip==]{adaptive_energy.jpg}
\includegraphics[width=0.45\textwidth,clip==]{spin_3_energy.jpg}
\includegraphics[width=0.45\textwidth,clip==]{spin_3_dt.jpg}
\includegraphics[width=0.45\textwidth,clip==]{spin_3_eta.jpg}
\caption{Evolution of the free energy and time step  by using adaptive time stepping method and  adaptive parameters are $tol=10^{-3}$, $\delta t_{min}=10^{-6}$, $\delta t_{max}=10^{-2}$ and $p=0.6$. }\label{time_adaptive_spin_3}
\end{figure}
\end{comment}

In the last simulation, we switched the sign of $\beta$ while keeping other parameters unchanged, i.e., we 
used the following  parameters:
\begin{equation}\label{fourth}
\eps_u=0.03 \quad \eps_v=0.01 \quad \sigma =10 \quad \alpha=0.33 \quad \beta =-0.5 \quad \gamma=0.
\end{equation}
 The results are presented in  Fig.\,\ref{spin_4}. We observe similar  striped ellipsoids and transformation shapes as in Fig.\,\ref{spin_3} except that the positions of yellow bulk and blue bulk  appeared to be  interchanged.

\begin{figure}
\centering
\subfigure[$u:t=0.01$.]{\includegraphics[width=0.22\textwidth,clip==]{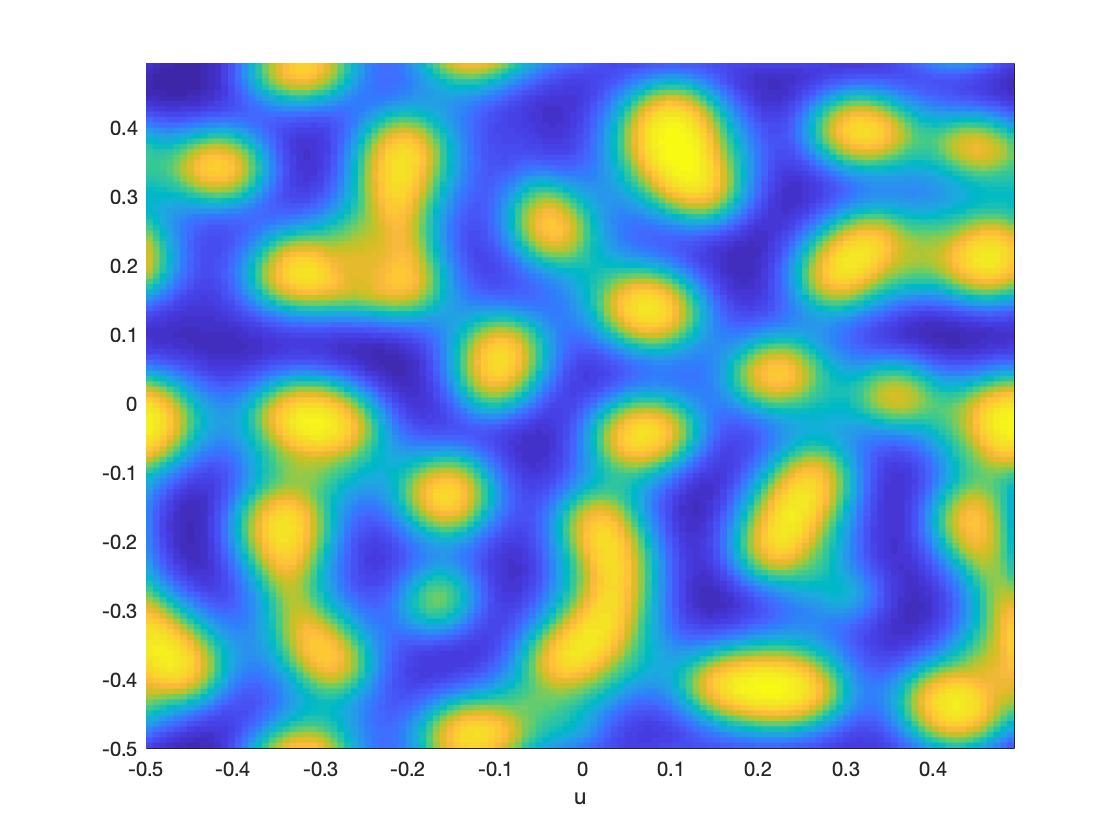}\hskip 0cm}
\subfigure[$v:t=0.01$.]{\includegraphics[width=0.22\textwidth,clip==]{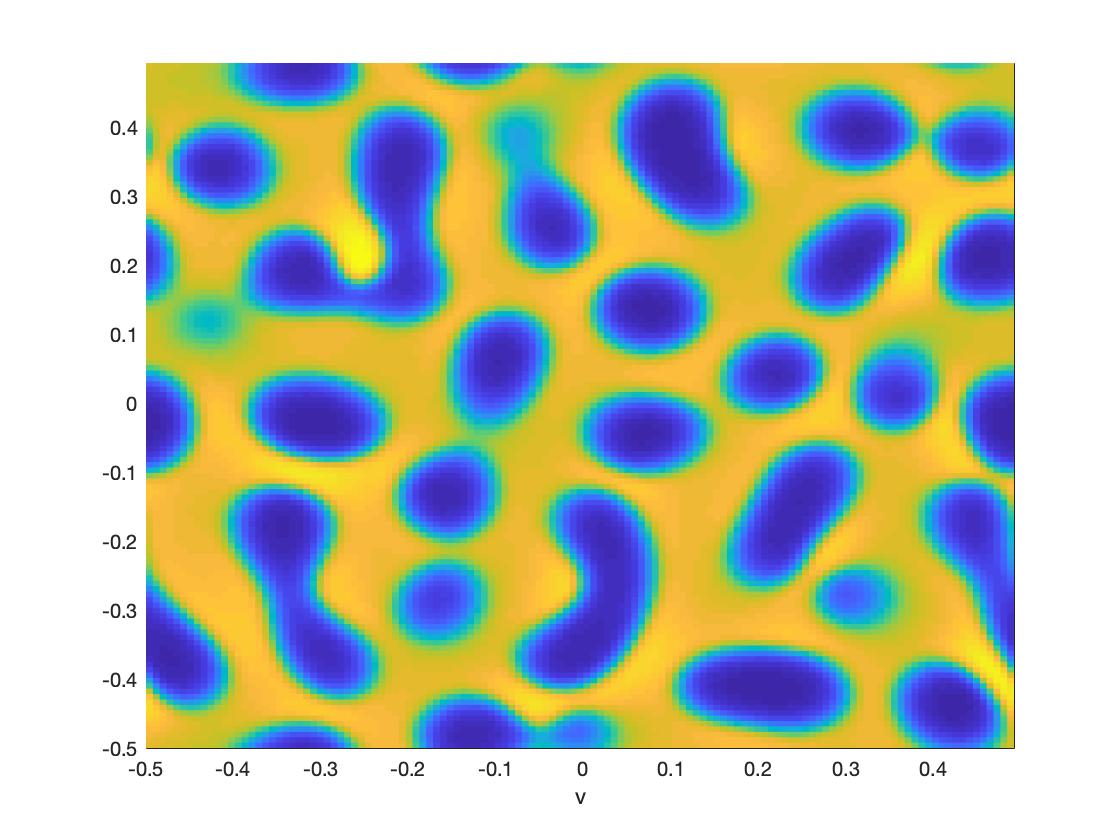}\hskip 0cm}
\subfigure[$u:t=0.1$.]{\includegraphics[width=0.22\textwidth,clip==]{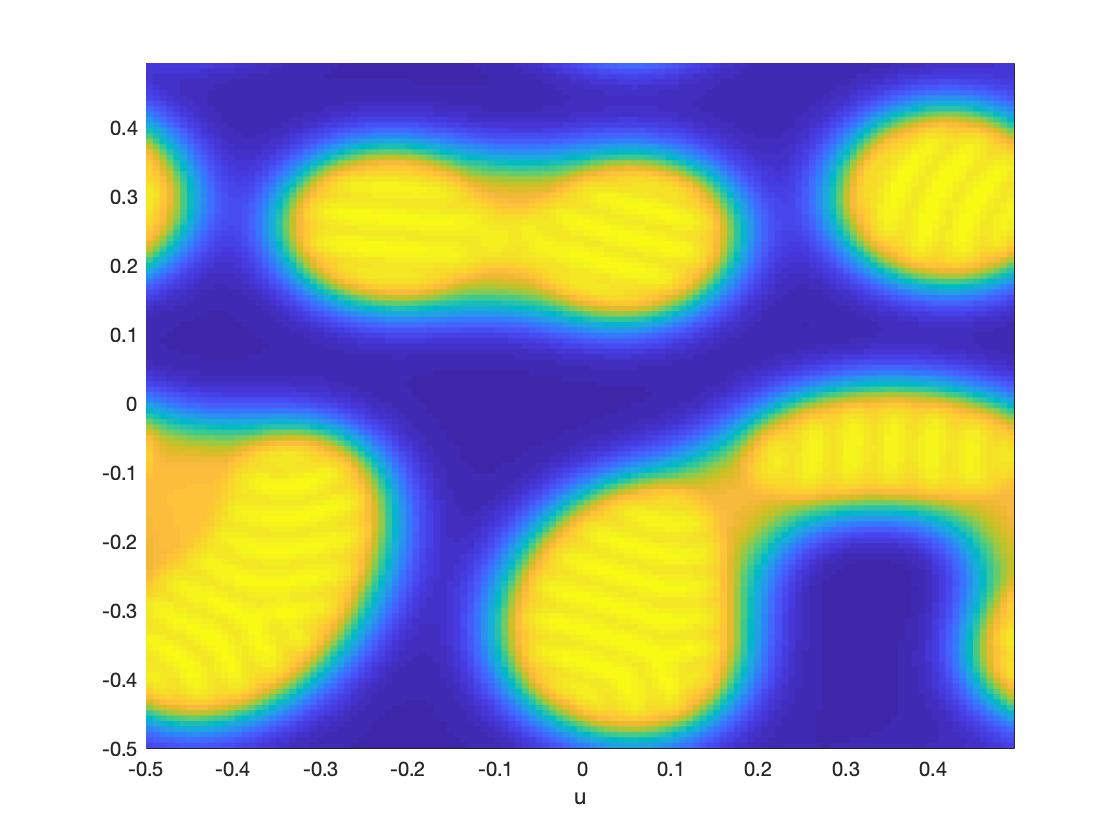}\hskip 0cm}
\subfigure[$v:t=0.1$.]{\includegraphics[width=0.22\textwidth,clip==]{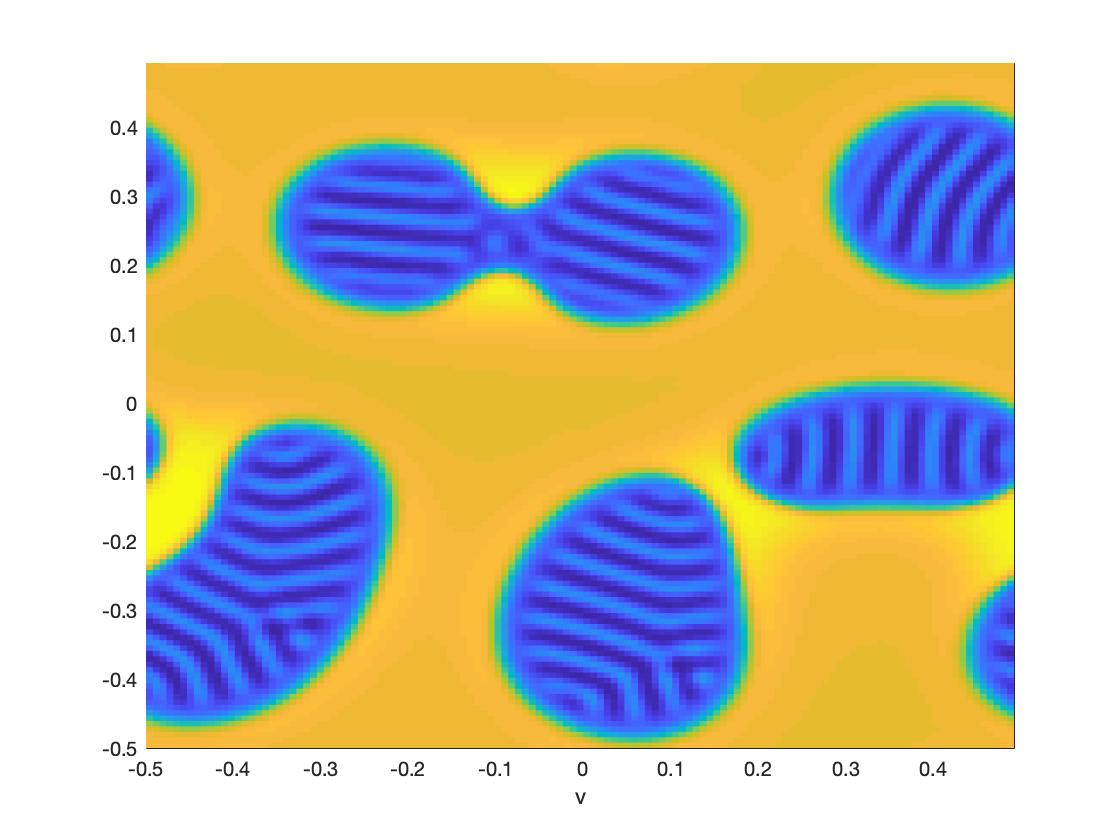}\hskip 0cm}
\subfigure[$u:t=0.25$.]{\includegraphics[width=0.22\textwidth,clip==]{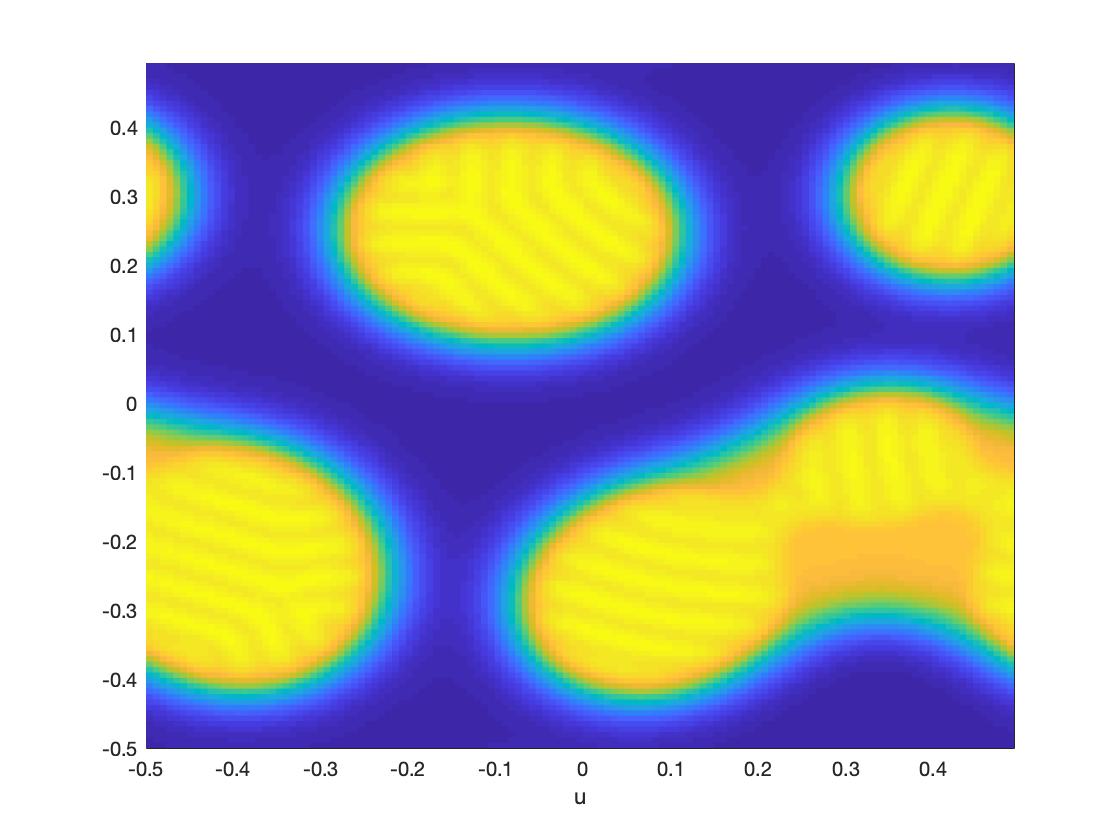}\hskip 0cm}
\subfigure[$v:t=0.25$.]{\includegraphics[width=0.22\textwidth,clip==]{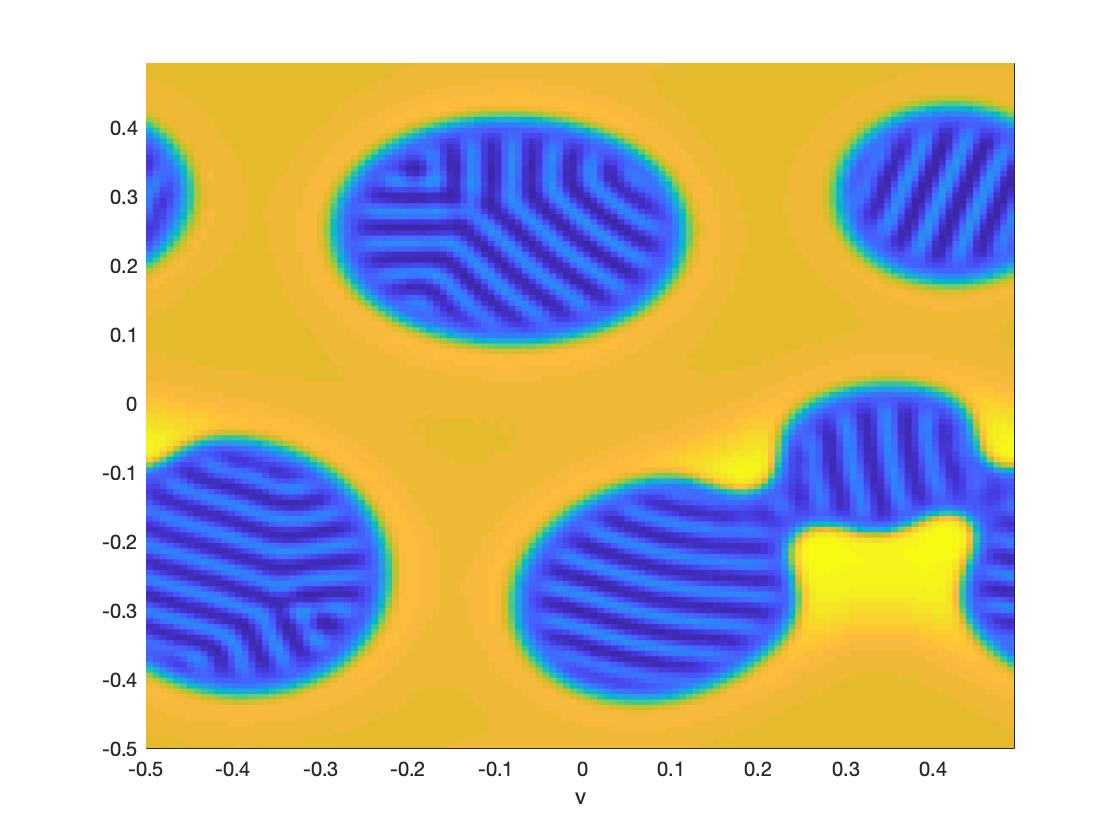}\hskip 0cm}
\subfigure[$u:t=3$.]{\includegraphics[width=0.22\textwidth,clip==]{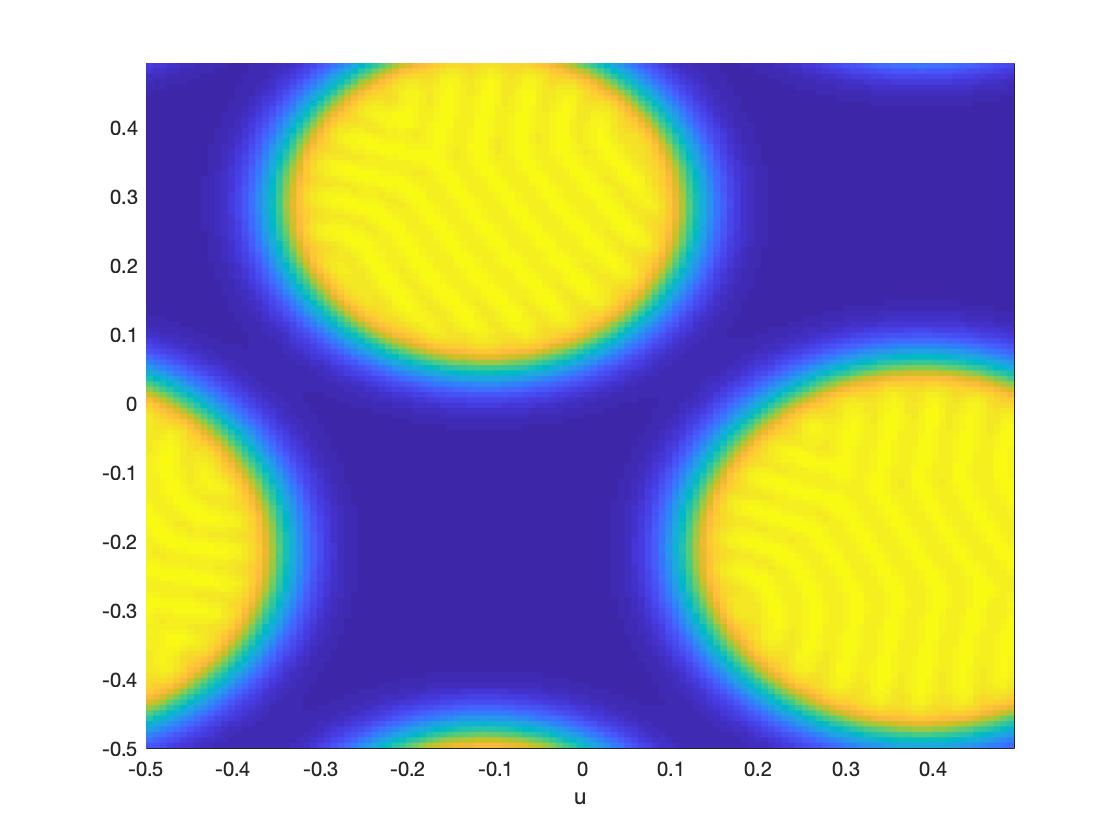}\hskip 0cm}
\subfigure[$v:t=3$.]{\includegraphics[width=0.22\textwidth,clip==]{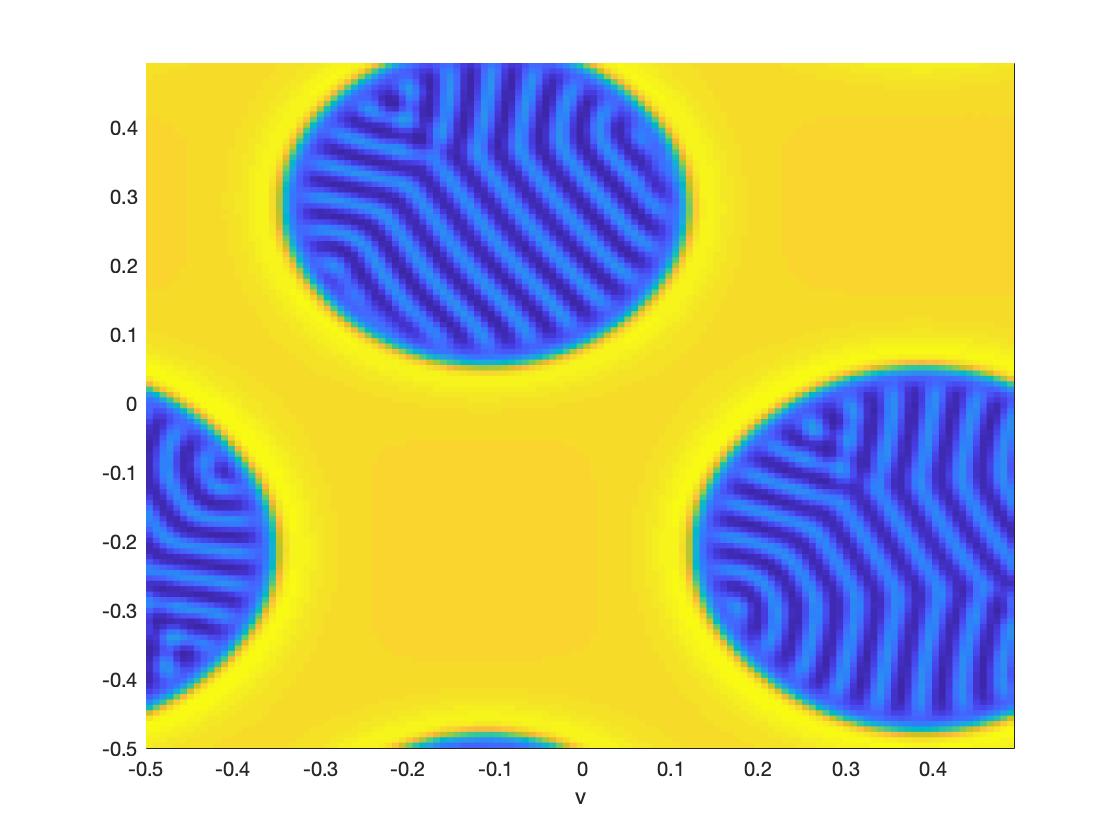}\hskip 0cm}
\caption{The 2D dynamical evolution of the phase variable $u,v$ for the Coupled-BCP model with parameters in \eqref{fourth}.}\label{spin_4}
\end{figure}

\section{Concluding remarks}
We proposed in this paper a new Lagrange multiplier approach for general gradient flows. The new approach shares most of the nice features with the SAV approach, such as, it leads to at least second-order unconditionally energy stable schemes which can be efficiently solved  with similar computational  cost as the SAV approach. It also offers two additional advantages: (i) it does not require that the integral of nonlinear functional in the free energy to be bounded from below, and (ii) it dissipates the original energy.
On the other hand, the new Lagrange approach requires solving a nonlinear algebraic equation with   negligible additional computational cost. but it may fail to converge when the time step is large and leads to additional difficulty in its convergence and error analysis.

We presented ample  numerical results  to validate the stability, efficiency  and accuracy of these schemes.  It is found that as long as the time steps are within a reasonable range which is required for accuracy, the Newton's iteration with one as initial condition for the nonlinear algebraic equation for the Lagrange multiplier always converges in a few steps and the solution is always close to one, which is the exact solution.

The new Lagrange Multiplier approach should be considered as an alternative for the SAV approach. In cases the  integral of nonlinear functional in the free energy is not  bounded from  below or it is difficult to find a lower bound or it is desirable to have a strictly monotonic decreasing energy in the original form, the new Lagrange Multiplier approach should be used.
  %It is clear that the new Lagrange Multiplier approach presented here can be used to deal with other dissipative or conservative systems.

Although we considered only time-discretized schemes in this paper, the stability results here can be carried over to any consistent finite dimensional Galerkin type approximations since the proofs are all based on a variational formulation with all test functions in the same space as the trial functions.

We also presented  two-dimensional numerical simulations for the coupled BCP model by using our new Lagrangian multiplier scheme with adaptive time stepping, and was able to capture some of the configuration shapes observed in the experiments as well as in  three-dimensional simulations of \cite{varadharajan2018surface}. %It is interesting to note that  our two-dimensional numerical simulations   captured some of the configuration shapes observed in the experiments as well as in  three-dimensional simulations of \cite{varadharajan2018surface}.
 Since our two-dimensional  simulation can be performed with very little computational cost, compared with three-dimensional simulations using less effective schemes, it is hoped that  we can use this computational tool to fine tune modeling parameters to better understand the nature of the annealing process. This is a task which should be undertaken in collaboration with materials scientists who have expertise in these types of BCPs.

\bigskip 
\noindent{\bf Acknowledgements.} The authors would like to thank Professor 
Yasumasa Nishiura for suggesting us to consider the coupled BCP model, and thank Dr. Edgar Avalos for enlightening discussions.

\bibliographystyle{plain}

\bibliography{vesicle_ref}

\end{document}